\documentclass[a4paper, 11pt]{amsart}
\usepackage[utf8]{inputenc}
\usepackage{microtype}
\usepackage{amsmath}
\usepackage{amssymb}
\usepackage{amsthm}
\usepackage{amsfonts}
\usepackage{booktabs}
\usepackage[final]{graphicx}
\usepackage{setspace}
\usepackage[center]{caption}
\usepackage{tikz}
\usetikzlibrary{shapes}
\usepackage{tikz-cd}
\usepackage[active]{srcltx}
\usepackage{mathtools}
\usepackage{mathrsfs}
\usepackage{ytableau}
\usepackage{mleftright}
\usepackage{comment}
\usepackage[colorlinks=true, allcolors=blue]{hyperref}
\usepackage[capitalise]{cleveref}
\usepackage[top=3cm,bottom=2cm,left=3cm,right=3cm,marginparwidth=1.75cm]{geometry}
\usepackage{xltabular}

\usepackage[colorinlistoftodos]{todonotes}

\usepackage[shortlabels]{enumitem}
\setlist[1]{itemsep=2pt}
\usepackage[capitalise]{cleveref}

\usepackage{autonum} %to number only \eqref'ed equations in a cref compatible way

\usepackage{bm}
\usepackage{dsfont}
\usepackage{enumitem}
\usepackage{relsize}
\theoremstyle{plain}\newtheorem{theorem}{Theorem}[]
\theoremstyle{plain}\newtheorem*{theorem*}{Theorem}
\theoremstyle{plain}\newtheorem{proposition}[theorem]{Proposition}
\theoremstyle{plain}\newtheorem{lemma}[theorem]{Lemma}
\theoremstyle{plain}\newtheorem{corollary}[theorem]{Corollary}
\theoremstyle{plain}\newtheorem*{corollary*}{Corollary}
\theoremstyle{definition}\newtheorem{definition}[theorem]{Definition}
\theoremstyle{definition}\newtheorem{remark}[theorem]{Remark}
\theoremstyle{definition}\newtheorem{example}[theorem]{Example}
\theoremstyle{definition}
\theoremstyle{definition}
\theoremstyle{definition}

\theoremstyle{thm}

\DeclarePairedDelimiter{\norm}{\|}{\|}
\newcommand{\R}{\mathds{R}}

\newcommand{\N}{\mathds{N}}
\newcommand{\Z}{\mathds{Z}}

\newcommand{\dist}{\delta}
\renewcommand{\P}{\mathds{R}\mathrm{P}}

\title{The Grassmann Distance Complexity}

\author[A.~Lerario]{Antonio Lerario}
\address[Lerario]{Scuola Internazionale Superiore di Studi Avanzati (SISSA),
via Bonomea, 265,
34136 Trieste, Italy}
\email{lerario@sissa.it}

\author[A.~Rosana]{Andrea Rosana}
\address[Rosana]{Scuola Internazionale Superiore di Studi Avanzati (SISSA),
via Bonomea, 265,
34136 Trieste, Italy}
\email{arosana@sissa.it}

\begin{document}

\begin{abstract}

Motivated by the concept of Euclidean Distance Degree, which measures the complexity of finding the nearest point to an algebraic set in Euclidean space, we introduce the notion of Grassmann Distance Complexity (GDC). This concept quantifies the complexity of solving the nearest point problem for subanalytic sets in the Grassmannian, using the intrinsic Riemannian distance. Unlike the Euclidean case, the Grassmannian distance is neither smooth nor semialgebraic, and its study requires using Lipschitz critical point theory and o--minimal geometry.

We establish fundamental properties of GDC, including computable bounds for real algebraic varieties and conditions ensuring the finiteness of critical points. Our results also include a nonlinear version of the classical Eckart–Young theorem, which characterizes critical points of the distance function from a generic 
$k$--plane to simple Schubert varieties. 

%  Motivated by the concept of Euclidean Distance Degree, that measures the complexity of the nearest point problem to an algebraic set in Euclidean space, we introduce the notion of Grassmann Distance Complexity, that serves the same purposes for the nearest point problem to an algebraic set in the Grassmannian and with respect to the \emph{intrinsic} Riemannian distance. This distance is not semialgebraic and the problem requires the use of different techniques, such as Lipschitz critical point theory and the theory of Pfaffian sets from o--minimal geometry. We also prove a nonlinear version of Eckart--Young Theorem, for the critical points of the distance function from a generic $k$--plane to a simple Schubert variety in the Grassmannian.
\end{abstract}

\maketitle 
    \tableofcontents

\section{Introduction}

 The \emph{Euclidean Distance Degree}  of a real algebraic variety $X\subset \R^n$, denoted by $\mathrm{EDD}(X)$, was introduced in \cite{EDdegree}  as a measure of the complexity of finding the closest point in $X$ to a given point $z\in \R^n$. The definition is the following: first, one considers the equations for the constrained critical points of the Euclidean Distance function from  $z$ restricted to $X$. These equations are \emph{algebraic}, and the $\mathrm{EDD}(X)$ is defined to be the number of their \emph{complex} solutions, for the generic $z\in \R^n$. The Euclidean Distance Degree is an upper bound for the number of real critical points and therefore it gives a measure of the complexity of the constrained optimization problem.
 
A natural question arises: can we define a similar notion when the ambient space is the Grassmannian \( G(k, n) \) of \( k \)-dimensional planes in \( \mathbb{R}^n \), and the Euclidean distance is replaced by the Riemannian distance induced by the orthogonally invariant metric on \( G(k, n) \)? In this paper, we propose a notion that serves this purpose, which we call the \emph{Grassmann Distance Complexity} (GDC).

Our approach differs from the Euclidean case because the Riemannian distance on the Grassmannian is neither smooth nor semialgebraic, preventing the direct application of techniques from smooth analysis and algebraic geometry. Instead, we employ Clarke's notion of critical points for locally Lipschitz functions and study the subdifferential of the distance function and its singularities using o-minimal geometry. (This is why we use the term ``complexity'' rather than ``degree'': we transition from the algebraic setting to the o-minimal one.)

The main contributions of this paper are as follows:
\begin{enumerate}
    \item We define the notion of Grassmann Distance Complexity for subanalytic submanifolds of \( G(k, n) \) and prove that this is generically finite. We provide computable bounds for the Grassmann Distance Complexity in the case of real algebraic hypersurfaces.
    \item We analyze the subdifferential of the Grassmannian distance function, particularly at points in the cut locus, and describe its structure.
    
    \item We apply our results and ideas to the problem of optimizing the distance to simple Schubert varieties, relating our notions to classical results such as the Eckart--Young theorem.
\end{enumerate}

\subsection{The Riemannian Distance Function on the Grassmannian}

We begin by recalling some necessary background on the Grassmannian \( G(k, n) \), its Riemannian geometry, and the concept of principal angles between subspaces. For more details, see Section~\ref{grassmanngeometrysection}.

Given two \( k \)-dimensional planes (subspaces) \( \mathbf{E}, \mathbf{L} \in G(k, n) \), their Riemannian distance is given by
\begin{equation}\label{anglesdistanceintro}
\dist(\mathbf{E}, \mathbf{L}) = \left( \sum_{i=1}^k \theta_i^2 \right)^{1/2},
\end{equation}
where \( 0 \leq \theta_1 \leq \cdots \leq \theta_k \leq \frac{\pi}{2} \) are the principal angles between \( \mathbf{E} \) and \( \mathbf{L} \). Throughout this paper, we assume for simplicity that \( k \leq n - k \).

\begin{remark}\label{remark1}
It is useful to establish a connection between the principal angles and the singular values of certain matrices. There exists a natural isomorphism
\[
\mathbb{R}^{(n - k) \times k} \simeq T_{\mathbf{E}} G(k, n),
\]
constructed as follows: fix orthonormal bases for \( \mathbf{E} \) and its orthogonal complement \( \mathbf{E}^\perp \). Any matrix \( A \in \mathbb{R}^{(n - k) \times k} \) represents a linear map \( \mathbf{E} \to \mathbf{E}^\perp \) in these bases, and we denote by \( v_A \in T_{\mathbf{E}} G(k, n) \) the corresponding tangent vector. Under this isomorphism, the Riemannian exponential map
\[
\exp_{\mathbf{E}}: T_{\mathbf{E}} G(k, n) \longrightarrow G(k, n)
\]
maps the set of matrices with singular values less than \( \frac{\pi}{2} \) surjectively onto the whole Grassmannian, and the principal angles between \( \mathbf{E} \) and \( \exp_{\mathbf{E}}(v_A) \) coincide with the singular values of \( A \) (see Proposition~\ref{boundedprop}). We will explore more geometric properties of this construction later.
\end{remark}

\subsection{Critical Points and the Subdifferential of the Distance Function}

Suppose \( X \subseteq G(k, n) \) is a submanifold. For a generic point \( \mathbf{L} \notin X \), we are interested in minimizing the distance from \( \mathbf{L} \) to \( X \), that is, minimizing the function
\begin{equation}\label{distancefunctionintro}
\dist_{\mathbf{L}}: G(k, n) \longrightarrow \mathbb{R}_{\geq 0}, \quad \mathbf{E} \mapsto \dist(\mathbf{L}, \mathbf{E})
\end{equation}
subject to \( \mathbf{E} \in X \).

If \( \dist_{\mathbf{L}} \) were smooth, one would find the critical points by solving \( \nabla \dist_{\mathbf{L}}|_X = 0 \). However,  \( \dist_{\mathbf{L}} \) is not differentiable everywhere. Indeed, it is not even semialgebraic: while the principal angles are semialgebraic functions (being singular values of \( A \)), the inverse of the Riemannian exponential involved in computing \( \dist(\mathbf{L}, \mathbf{E}) \) is only real analytic.

Despite the lack of differentiability, \( \dist_{\mathbf{L}} \) is locally Lipschitz, and thus differentiable almost everywhere by Rademacher's Theorem. We can therefore define critical points using Clarke's generalized gradient. The subdifferential \( \partial_{\mathbf{E}} (\dist_{\mathbf{L}}) \subset T_{\mathbf{E}} G(k, n) \) is defined as
\begin{equation}\label{subdiffdefdist}
\partial_{\mathbf{E}} (\dist_{\mathbf{L}}) := \mathrm{co} \left\{ \lim_{\mathbf{E}_n \to \mathbf{E}} \nabla \dist_{\mathbf{L}} (\mathbf{E}_n) \mid \mathbf{E}_n \in \omega(\dist_{\mathbf{L}}), \ \text{limit exists} \right\},
\end{equation}
where \( \mathrm{co} \) denotes the convex hull, \( \omega(\dist_{\mathbf{L}}) \) is the set of differentiability points of \( \dist_{\mathbf{L}} \), and \( \nabla \dist_{\mathbf{L}} \) is the gradient at those points.

The function \( \dist_{\mathbf{L}} \) is smooth on \( G(k, n) \setminus (\mathrm{cut}(\mathbf{L}) \cup \{\mathbf{L}\}) \), where \( \mathrm{cut}(\mathbf{L}) \) denotes the cut locus of \( \mathbf{L} \). At differentiability points \( \mathbf{E} \), the gradient \( \nabla \dist_{\mathbf{L}}(\mathbf{E}) \) corresponds to the tangent vector at \( \mathbf{E} \) of the unique length-minimizing geodesic from \( \mathbf{L} \) to \( \mathbf{E} \).

It is known (see \cite{WongGrassmann}) that the cut locus \( \mathrm{cut}(\mathbf{L}) \) consists of planes \( \mathbf{E} \) where at least one principal angle between \( \mathbf{E} \) and \( \mathbf{L} \) equals \( \frac{\pi}{2} \). The cut locus is a codimension-one Schubert variety (Lemma~\ref{cutlemma}), stratified as
\[
\mathrm{cut}(\mathbf{L}) = \bigsqcup_{j=1}^k \Omega(j),
\]
where \( \Omega(j) \) consists of \( k \)-planes with exactly \( j \) principal angles equal to \( \frac{\pi}{2} \).

We analyze the structure of the subdifferential \( \partial_{\mathbf{E}} \dist_{\mathbf{L}} \) for \( \mathbf{E} \in \mathrm{cut}(\mathbf{L}) \) in detail (Proposition~\ref{subdiffpropo}), leading to the following result (Corollary~\ref{subdiffcoro}):

\begin{theorem}\label{subdiffcorointro}
For any \( \mathbf{L} \in G(k, n) \) and \( \mathbf{E} \in \Omega(j) \subseteq \mathrm{cut}(\mathbf{L}) \) with \( j \in \{1, \dots, k\} \), the subdifferential \( \partial_{\mathbf{E}} \dist_{\mathbf{L}} \) is a convex set linearly isomorphic to the convex hull of the orthogonal group \( O(j) \). In particular,
\[
\dim \left( \partial_{\mathbf{E}} \dist_{\mathbf{L}} \right) = j^2.
\]
\end{theorem}

When \( X \subset G(k, n) \) is a smooth submanifold, the subdifferential of the restricted function \( \dist_{\mathbf{L}}|_X \) can be described using projections. For generic \( \mathbf{L} \), we have the following (Lemma~\ref{subdiffprojectionlemma}):

\begin{proposition}\label{subdiffprojectionlemmaintro}
Let \( X \) be a smooth submanifold of \( G(k, n) \). For a generic \( \mathbf{L} \in G(k, n) \) and every \( \mathbf{E} \in X \), we have
\[
\partial_{\mathbf{E}} (\dist_{\mathbf{L}}|_X) = \mathrm{proj}_{T_{\mathbf{E}} X} \left( \partial_{\mathbf{E}} \dist_{\mathbf{L}} \right),
\]
where \( \mathrm{proj}_{T_{\mathbf{E}} X} \) denotes the orthogonal projection onto \( T_{\mathbf{E}} X \).
\end{proposition}

\subsection{Grassmann Distance Complexity}

Following convex analysis, a point \( \mathbf{E} \in X \) is said to be \emph{critical} for \( \dist_{\mathbf{L}}|_X \) if the subdifferential \( \partial_{\mathbf{E}} (\dist_{\mathbf{L}}|_X) \) contains the zero vector. Minima of \( \dist_{\mathbf{L}}|_X \) are necessarily critical points. In general, critical points may not be isolated; for example, if \( X \) is a simple Schubert variety, there could be positive-dimensional families of critical points. However, for generic \( \mathbf{L} \), these points cannot be local minima (Theorem~\ref{nominimacut}).

We show in Theorem~\ref{finitecritical} that if \( X \) is compact, the set of critical points of \( \dist_{\mathbf{L}}|_X \) not in \( \mathrm{cut}(\mathbf{L}) \) is finite. The same holds for subanalytic \( X \), even if not compact. Since \( \dist_{\mathbf{L}} \) is subanalytic, we can define the following:

\begin{definition}[Grassmann Distance Complexity]\label{GDDdefintro}
Let \( X \) be a subanalytic submanifold of \( G(k, n) \). The \emph{Grassmann Distance Complexity} of \( X \) is defined as
\[
\mathrm{GDC}(X) = \max_{\mathbf{L} \in G(k, n)} \# \left\{ \mathbf{E} \in X \setminus \mathrm{cut}(\mathbf{L}) \mid \mathbf{E} \text{ is critical for } \dist_{\mathbf{L}}|_X \right\},
\]
where the maximum is taken over generic \( \mathbf{L} \in G(k, n) \).
\end{definition}

We prove that \( \mathrm{GDC}(X) \) is bounded by a constant depending only on \( X \) (Theorem~\ref{uniformboundedness}). Moreover, when \( X \) is a real algebraic variety, this bound is effectively computable since \( \dist_{\mathbf{L}} \) is definable in a Pfaffian o-minimal structure. For example, for hypersurfaces of degree \( d \), we obtain the following (Theorem~\ref{boundthm}):

\begin{theorem}\label{boundthmintro}
For every \( 0 \leq k \leq n \), there exist constants \( c_1(k, n), c_2(k, n) > 0 \) such that for every smooth hypersurface \( X \hookrightarrow G(k, n) \) of degree \( d \), the following bound holds:
\[
\mathrm{GDC}(X) \leq c_1(k, n) d^{c_2(k, n)}.
\]
\end{theorem}

\subsection{Optimizing the Distance to Simple Schubert Varieties}

We apply our framework to optimize the distance to simple Schubert varieties. Fix a \( k \)--plane \( \mathbf{W} \in G(k, n) \) and define
\[
\Omega_s := \left\{ \mathbf{E} \in G(k, n) \mid \dim (\mathbf{E} \cap \mathbf{W}) \geq s \right\},
\]
for \( s \in \{1, \dots, k - 1\} \). Note that \( \Omega_k = \{\mathbf{W}\} \).

Recalling Remark~\ref{remark1}, the principal angles between \( \mathbf{W} \) and \( \mathbf{E} \) correspond to the singular values of a matrix \( A \in \mathbb{R}^{(n - k) \times k} \). Let \( \mathcal{M}_{\leq k - s} \) be the set of matrices with rank at most \( k - s \) and \( R_{\frac{\pi}{2}} \) the set of matrices with singular values less than \( \frac{\pi}{2} \). Under the isomorphism, we have (Proposition~\ref{boundedrankomega})
\[
\Omega_s = \exp_{\mathbf{W}} \left( \mathcal{M}_{\leq k - s} \cap R_{\frac{\pi}{2}} \right).
\]

For \( 0 < \epsilon < \frac{\pi}{2} \), let \( \widetilde{B} = B(0, \epsilon) \subset \mathbb{R}^{(n - k) \times k} \) and \( B = \exp_{\mathbf{W}} (\widetilde{B}) \). Locally, we have a diffeomorphism of pairs
\begin{equation}\label{eq:expdiff}
\left( \widetilde{B}, \mathcal{M}_{\leq k - s} \cap \widetilde{B} \right) \stackrel{\exp_{\mathbf{W}}}{\longrightarrow} \left( B, \Omega_s \cap B \right).
\end{equation}

Adding metrics to the picture, note that for \( \mathbf{L} \in B \), there exists a unique \( A_{\mathbf{L}} \in \widetilde{B} \) such that \( \exp_{\mathbf{W}} (A_{\mathbf{L}}) = \mathbf{L} \), and the principal angles between \( \mathbf{W} \) and \( \mathbf{L} \) are the singular values of \( A_{\mathbf{L}} \). Therefore,
\[
\dist(\mathbf{L}, \mathbf{W}) = \| A_{\mathbf{L}} \|_{\mathrm{F}},
\]
where \( \| \cdot \|_{\mathrm{F}} \) denotes the Frobenius norm.

In the radial directions, the Riemannian distance and the Euclidean distance on the tangent space coincide. However, the metrics \( \exp_{\mathbf{W}}^* g_{G(k, n)} \) and the Euclidean metric \( \langle \cdot, \cdot \rangle_{\mathrm{F}} \) only coincide at the origin; they converge to each other on small balls after rescaling (Proposition~\ref{metricconvergence}).

When \( \mathbf{L} \notin \Omega_s \) (implying \( A_{\mathbf{L}} \notin \mathcal{M}_{\leq k - s} \)), there are two natural optimization problems:
\begin{enumerate}
    \item Minimizing the distance from \( A_{\mathbf{L}} \) to \( \mathcal{M}_{\leq k - s} \) with respect to the Frobenius norm.
    \item Minimizing the distance from \( \mathbf{L} \) to \( \Omega_s \) in the Grassmannian.
\end{enumerate}

Under the diffeomorphism \eqref{eq:expdiff}, these problems are set--wise the same but involve different metrics. The first is a linear problem solved using the Eckart--Young theorem (Theorem~\ref{EY}). We investigate the relationship between the solutions of the linear and nonlinear problems.

Using the convergence of metrics at small scales, we show that Grassmannian critical points are perturbations of the Eckart--Young critical points (Proposition~\ref{convergenceprop}). Surprisingly, the Eckart--Young critical points are also Grassmannian critical points.

\begin{theorem}\label{standardEYomegasintro}
Let \( \mathbf{L} \notin \Omega_s \) be generic and \( A_{\mathbf{L}} \in R_{\frac{\pi}{2}} \) such that \( \exp_{\mathbf{W}} (A_{\mathbf{L}}) = \mathbf{L} \). For any subset \( I \subseteq \{1, \dots, k\} \) with \( |I| = k - s \), let \( A_{\mathbf{L}, I} \in \mathcal{M}_{\leq k - s} \) be the rank--\( (k - s) \) approximation of \( A_{\mathbf{L}} \) obtained via the Eckart--Young theorem. Define \( \mathbf{L}_I = \exp_{\mathbf{W}} (A_{\mathbf{L}, I}) \in \Omega_s^{\mathrm{sm}} \). Then \( \mathbf{L}_I \) is a critical point for \( \dist_{\mathbf{L}}|_{\Omega_s} \).
\end{theorem}

In fact, the optimization problem can be completely solved using the Eckart--Young critical points (see Theorem~\ref{conjecture}):

\begin{theorem}[Nonlinear Eckart--Young]
For generic \( \mathbf{L} \notin \Omega_s \), the restriction of \( \dist_{\mathbf{L}} \) to \( \Omega_s \) admits a unique global minimizer, corresponding to the minimum given by the Eckart--Young theorem.
\end{theorem}

Since the Eckart--Young theorem yields \( \binom{k}{s} \) critical points, Theorem~\ref{standardEYomegasintro} provides a lower bound for the Grassmann Distance Complexity of a simple Schubert variety:
\[
\mathrm{GDC}(\Omega_s) \geq \binom{k}{s}.
\]

Globally, there are more critical points because \( \Omega_s \) is compact. As for local maxima, they cannot be found using this procedure---they lie at infinity. Nevertheless, we prove:

\begin{theorem}\label{maximatheoremintro}
For generic \( \mathbf{L} \notin \Omega_s \), the restriction of \( \dist_{\mathbf{L}} \) to \( \Omega_s \) admits a Grassmannian \( G(k - s, n - k - s) \) of global maximizers, which can be explicitly computed. If \( 0 < \theta_1(\mathbf{L}) < \dots < \theta_k(\mathbf{L}) < \frac{\pi}{2} \) are the principal angles between \( \mathbf{L} \) and \( \mathbf{W} \), then
\[
\max \dist_{\mathbf{L}}|_{\Omega_s} = \left( \theta_{k - s + 1}(\mathbf{L})^2 + \dots + \theta_k(\mathbf{L})^2 + (k - s) \left( \frac{\pi}{2} \right)^2 \right)^{1/2}.
\]
\end{theorem}

\section{Preliminaries}
In this section we introduce some useful concepts and notation that we will use in the rest of the paper. Most of the material covered in this specific section is standard and can be found, for instance, in \cite{matrixanalysis, Kozlov, grassmannhandbook, metric}. We first recall the well--known Eckart--Young Theorem in the context of low--rank approximation for matrices. Then we introduce the Grassmann manifold as a quotient of the orthogonal group and discuss its standard Riemannian structure and the corresponding notion of distance between planes. 

\subsection{Low--rank matrix approximation and the Eckart--Young Theorem}
 A recurring problem in applications is to find a suitable way of approximating a matrix with another one of smaller rank. This has to be done in such a way that the loss of information is the minimum possible. Geometrically, this corresponds to minimizing the distance between the original matrix and its approximation with respect to some meaningful metric. 
        
A natural metric on the space of $n \times m$ matrices is the one induced by the \emph{Frobenius scalar product}. This is defined as follows: for $A=(a_{ij}) \in \R^{n\times m}$ and $B=(b_{ij})\in \R^{n\times m}$ the Frobenius scalar product between $A$ and $B$ is given by 
        \begin{align}\label{eq:frobs}
            \langle A,B\rangle_F := \mathrm{tr}(A^{\top}B)  = \sum_{i,j=1}^{n,m}a_{ij}b_{ij}, 
        \end{align}
where $\mathrm{tr}$ denotes the trace of a matrix. Remark that this is just the Euclidean scalar product when you vectorize a matrix and see it as a vector in $\R^{nm}$.
The induced Frobenius distance on $\R^{n\times m}$ is 
        \begin{align}\label{eq:frobd}
            \widehat{\dist}(A,B):=\| A-B\|_F=\bigg(\langle A-B, A-B \rangle_F\bigg)^{\frac{1}{2}}=\bigg(\sum_{i,j=1}^{n,m}(a_{ij}-b_{ij})^2\bigg)^{\frac{1}{2}}.
        \end{align}
        
        In order to rigorously state the approximation problem, we introduce the following notation for sets of matrices with rank constraints.
        \begin{definition}\label{def:rankconstraint}Let $r, m, n\in \N$ with $r\leq \min\{n, m\}$. We define:
\begin{align}\label{boundednotation}
    \mathcal{M}_{\leq r}&:=\{A \in \R^{n \times m} \ | \ \mathrm{rank}(A)\leq r\}, \\
    \mathcal{M}_{r}&:=\{A \in \R^{n \times m} \ | \ \mathrm{rank}(A)=r\}.
\end{align}
\end{definition}

Recall that the set $\mathcal{M}_{\leq r}$ is a real algebraic variety (defined in $\R^{n\times m}$ by the vanishing of all $(r+1)\times (r+1)$ minors of a matrix), whose set of smooth points is $\mathcal{M}_r$.

        Using these notions, the low--rank approximation problem becomes the following: given a matrix $A \in \R^{n\times m}$ and $r<\mathrm{rank}(A)$, find a matrix of rank $r$ that minimizes the Frobenius distance from $A$. In other words: we are asked to minimize the function 
        \[\widehat{\dist}_A:\R^{n\times m} \to \R, \quad \widehat{\dist}_A(X):=\|X-A\|_F,\]
        constrained to the set $\mathcal{M}_{\leq r}.$

A complete answer to this problem is given by the Eckart--Young Theorem \cite{EckartYoung}. In order to state the result, recall that any matrix $A \in \R^{n\times m}$ admits a \emph{Singular Values Decomposition} (SVD). Assuming without loss of generality that $m \leq n$, this is a factorization $A=U\Sigma V^{\top}$ where $U \in O(n)$, $V \in O(m)$ and $\Sigma=\mathrm{diag}(\sigma_1,\dots,\sigma_m) \in \R^{n\times m}$, with $\sigma_i\geq 0$ for all $i=1, \ldots, m$. The elements $\sigma_1,\dots,\sigma_m$ are called \emph{singular values} of $A$ and they are the square roots of the eigenvalues of $A^{\top}A$. One can easily show that the Frobenius norm of a matrix $A$ can be expressed in terms of its singular values as 

        \begin{align}
            \|A\|_F = \bigg(\sum_{i=1}^m \sigma_i ^2 \bigg)^{\frac{1}{2}}. 
        \end{align}
        
        \begin{remark}[Non-uniqueness of the SVD]
            The SVD of a matrix $A \in \R^{n \times m}$ is not unique. A first source for the lack of uniqueness comes from the different possible orderings of the singular values and consequent reordering of the corresponding columns of $U$ and $V$. Even if we fix an ordering of the singular values, non-uniqueness still arises in the case of repeated eigenvalues. Lastly, if $m < n$, then the last $n-m$ columns of the orthogonal matrix $U \in O(n)$ give no contribution in the decomposition $A=U\Sigma V^{\top}$. Therefore different orthogonal completions of the first $m$ columns of $U$ will still give the same matrix $A$. 
        \end{remark}
        
We can now state the Eckart--Young Theorem. 

\begin{theorem}[\cite{EckartYoung}]\label{EY}
Let $A \in \R^{n \times m}$ be a matrix with distinct singular values. Let also $A=U\Sigma V^{\top}$ be a singular values decomposition with the singular values ordered as $\sigma_1>\dots>\sigma_m>0$ (we assume $m\leq n$). For any subset $I\subseteq \{1,\dots,m\}$ with $| I | =r$ define $\Sigma_I$ as the rectangular $n\times m$ diagonal matrix with entries $(\Sigma_I)_{ii}=\sigma_i$ if $i \in I$ and $(\Sigma_I)_{ii}=0$ if $i \notin I$. Define also $A_I = U \Sigma_I V^{\top}$, that is a rank--$r$ matrix. Then:
\begin{enumerate}[label=(\roman*)]
\item the critical points of $\widehat{\dist}_A$, the distance function from $A$, constrained to the smooth manifold $\mathcal{M}_r$ are all and only the matrices $A_I$ with $| I | = r$. In particular, restricted to this manifold,  the function $\widehat{\dist}_A$ has $\binom{m}{r}$ critical points.
\item $\forall I$ with $| I | =r$ we have 
\begin{align}
    \| A-A_I \|_F^2=\sum_{i \notin I} \sigma_i^2.
\end{align}
In particular $A_{\{1,\dots,r\}}$ is a global minimum of the restriction of $\widehat{\dist}_A$ to $\mathcal{M}_{\leq r}$ and 
\[
    \widehat{\dist}(A,\mathcal M_{\leq r}) = \| A - A_{\{1,\dots,r\}} \| = \sqrt{\sum_{i=r+1}^m \sigma_i^2}.
\]
\end{enumerate}
        \end{theorem}
    
    \begin{remark}
        A slightly modified version of \cref{EY} covers the case when the singular values of $A$ are not distinct. Notice also that, since a generic matrix $A \in \R^{n\times m}$ satisfies the hypothesis of the statement, it follows that the Euclidean Distance Degree of the algebraic variety $\mathcal M_{\leq r}$ is $\binom{m}{r}$.
    \end{remark}
    
    In \cref{Schubertsection} we will connect the rank approximation problem for matrices with the nearest point problem to a Schubert variety in the real Grassmannian. In particular, we will use the Eckart--Young Theorem to find some of the critical solutions and, among them, the optimal ones.

    \subsection{The Riemannian metric on the real Grassmannian}\label{grassmanngeometrysection}    
        We denote by $G(k,n)$ the real Grassmannian of $k$--planes in $\R^n$,
        \begin{align}
            G(k,n):=\{\mathbf E \ | \ \mathbf E \ \text{is a $k$--dimensional linear subspace of } \R^n\}.
        \end{align}
        From now on, we will always assume for simplicity that $k \leq n-k$. 
        
        In this section we introduce a natural Riemannian metric on real Grassmannians and study the corresponding induced Riemannian distance.
There at least two equivalent ways to do this: (1) viewing $G(k,n)$ as a Riemannian homogeneous space (this is the approach we take in this paper); (2) restricting the Fubini--Study metric to the image of the Pl\"ucker embedding. We refer the reader to \cite{Kozlov} and \cite[Section 4.8]{metric} for more details on comparing these two approaches. 
        
        \subsubsection{The Grassmannian as a homogeneous space}We introduce the following notation: for orthonormal vectors $v_1,\dots,v_k \in \R^n$,  we denote by $[v_1 \dots v_k]$ the $k$--plane spanned by $v_1,\dots,v_k$. The orthogonal group $O(n)$ acts transitively on $G(k,n)$ by rotations. Explicitly, if $R \in O(n)$ and $\mathbf E=[v_1 \dots v_k]$, the orthogonal matrix $R$ acts as $R \cdot \mathbf E = [Rv_1 \dots Rv_k]$. Given $\mathbf E_0=[e_1 \dots e_k]$, where $e_1,\dots,e_n$ is the standard basis of $\R^n$, it is easy to see that the isotropy subgroup of $\mathbf E_0$ is
        \begin{align}
        \mleft\{\mleft(\begin{array}{c|c}
        A_1 & 0 \\ \hline
        0 & A_2
        \end{array}\mright) \ \bigg| \ A_1 \in O(k), \ A_2 \in O(n-k)\mright\} \cong O(k) \times O(n-k).
        \end{align}
        By \cite[Theorem 21.20]{smoothmanifoldslee},  the Grassmannian $G(k,n)$ has a unique smooth manifold structure making the action of $O(n)$ smooth and it has dimension $\mathrm{dim}(G(k,n)) = \mathrm{dim}(O(n))-\mathrm{dim}(O(k)\times O(n-k))=k(n-k)$.  We denote by 
        \begin{align}
        p: O(n) &\longrightarrow G(k,n) \\
        R &\longmapsto R \cdot \mathbf E_0,
        \end{align}
       the quotient map, which projects a matrix $R$ to the span of its first $k$ columns. This map is a smooth submersion with fibers isomorphic to $O(k)\times O(n-k)$. As a homogeneous space,
        \begin{align}
        G(k,n) \cong \frac{O(n)}{O(k)\times O(n-k)}.
        \end{align}

          \subsubsection{Useful descriptions of tangent spaces to the Grassmannian}We introduce in this section useful isomorphisms between the tangent space $T_\mathbf{E}G(k,n)$, the space of matrices $\R^{(n-k)\times k}$ and $\mathrm{Hom}(\mathbf{E}, \mathbf{E}^\perp)$.
          
          To start with, we endow $O(n)$ with the restriction of the (rescaled) Frobenius metric: given $A,B \in T_{R}O(n)$ we define 
      \begin{align}
        \langle A,B \rangle_{\widetilde F} := \tfrac{1}{2}\mathrm{tr}(A^{\top}B).
        \end{align}
         This Riemannian metric on $O(n)$ is both left and right invariant (\cite[Proposition 4.45]{metric}). (The reason for the ``$\tfrac{1}{2}$'' scaling factor will be clear below.)
          
          For every $\mathbf E \in G(k,n)$, we denote by $\widehat{E} \in p^{-1}(\mathbf E)$ a fixed representative in $O(n)$. As $p$ is a submersion, the differential 
        \begin{align}
        D_{\widehat{E}}p : T_{\widehat{E}}O(n) \longrightarrow T_{\mathbf E}G(k,n)
        \end{align}        
        is surjective at every point and in particular it defines a linear isomorphism between the orthogonal complement of its kernel and $T_{\mathbf E}G(k,n)$. Recall that the tangent spaces to $O(n)$ are just translations $T_{\widehat{E}}O(n) = \widehat{E} \cdot \mathfrak{so}(n)$ and it is easy to see that 
        \begin{align}
        \mathrm{Ker}(D_{\widehat{E}}p)&=\mleft\{\widehat{E} \mleft(\begin{array}{c|c}
             M_1 & 0 \\ \hline
             0 & M_2
        \end{array}\mright) \ \bigg| \ M_1 \in \mathfrak{so}(k), \ M_2 \in \mathfrak{so}(n-k)\mright\} \cong T_{\widehat{E}}p^{-1}(\mathbf E), \\ \label{orthogonalkernel}
        (\mathrm{Ker}(D_{\widehat{E}}p))^{\perp}&=\mleft\{\widehat{E} \mleft(\begin{array}{c|c}
             0 & -A^{\top} \\ \hline
             A & 0
        \end{array}\mright) \ \bigg| \ A \in \R^{(n-k)\times k}\mright\}\simeq T_\mathbf{E}G(k,n).
        \end{align}
For every matrix $A \in \R^{(n-k)\times k}$ we introduce now the notation 
        \begin{align}
           \label{eq:MA} M_A := \mleft( \begin{array}{c|c}
                0 & -A^{\top} \\ \hline
                A & 0
            \end{array}\mright).
        \end{align}
        Then the map $A\mapsto\widehat{E} M_A$ gives a linear isomorphism between $\R^{(n-k)\times k}$ and   $(\mathrm{Ker}(D_{\widehat{E}}p))^{\perp}$. 
        One can easily check that this gives an isometry of Euclidean spaces
        \[\label{eq:isomAMA}\left(\R^{(n-k)\times k}, \langle \cdot, \cdot\rangle_{\mathrm{F}}\right)\simeq \left((\mathrm{Ker}(D_{\widehat{E}}p))^{\perp},\tfrac{1}{2}\langle \cdot, \cdot\rangle_{\mathrm{F}}\right). \]
        (Notice the ``$\tfrac{1}{2}$'' factor on the right hand side.)
        We use this isomoprhism to define the following isomorphism, which allows us to identify tangent vectors to the Grassmannian with matrices when needed.
        \begin{definition}[The isomorphism $A\mapsto v_A$]\label{def:vA}Let $\mathbf{E}\in G(k,n)$ and choose a representative $\widehat{E}\in p^{-1}(\mathbf{E})$. Then, we define the linear isomorphism
        \begin{align}\label{matrixiso}
            \R^{(n-k)\times k}  \stackrel{\simeq}{\longrightarrow} T_{\mathbf E}G(k,n), \quad A \mapsto v_A := D_{\widehat{E}}p (\widehat{E}M_A).
        \end{align}
                (We stress that this isomorphism depends on the fixed representative $\widehat{E}$.) 
        \end{definition}
        
        We can also establish a correspondence between $\mathrm{Ker}(D_{\widehat{E}}p)^{\perp}$ and the set of linear maps $\mathrm{Hom}(\mathbf E, \mathbf E^{\perp})$ as follows.
        \begin{definition}[The isomoprhism $A\mapsto \varphi_A$]\label{def:hom}Let $\mathbf{E}\in G(k,n)$ and choose a representative $\widehat{E}\in p^{-1}(\mathbf{E})$. Then, we define the linear isomorphism
        \[
           \R^{(n-k)\times k} \stackrel{\simeq}{\longrightarrow} \mathrm{Hom}(\mathbf{E}, \mathbf{E}^\perp)  , \quad A \mapsto \varphi_A  \]
defined by identifying a matrix $A$ with the linear map $\varphi_A: \mathbf E \longrightarrow \mathbf E^{\perp}$ represented by the matrix $A$ with respect to basis corresponding to $\widehat{E}$.  \end{definition}

        Summarizing,  for a fixed representative $\widehat{E}\in p^{-1}(\mathbf{E})$      and for any $A \in \R^{(n-k)\times k}$ we have isomorphic representations of the  tangent vectors to $G(k,n)$ at $\mathbf E$:
        \begin{align}\label{tangentequivalence}
          \mathrm{Hom}(\mathbf{E}, \mathbf{E}^\perp)\simeq \R^{(n-k)\times k}\simeq T_\mathbf{E}G(k,n), \quad \varphi_A \longleftrightarrow A \longleftrightarrow v_A.
        \end{align}
        In the folllowing we will use all these different representations, depending on the context. 
        \subsubsection{The orthogonally invariant Riemannian metric}We introduce now an orthogonally invariant Riemannian metric on $G(k,n)$. In fact, this metric is unique up to multiples.
        \begin{theorem}For every $k,n\in \mathbb{N}$ there exists a unique (up to multiples) Riemannian metric on $G(k,n)$ that is invariant under the action of the orthogonal group $O(n)$.
        \end{theorem}
        \begin{proof} See \cite[Corollary 5.34]{metric} and \cite[Remark 5.35]{metric}.\end{proof}
        \begin{remark}If we only require that the Riemannian metric on $G(k,n)$ is invariant under the action of $SO(n)$, then this is still unique (up to multiples), except for the case $(k,n)=(2,4)$. In this case there is a one dimensional family of non--equivalent invariant metrics, see \cite[Example 5.20 and Corollary 5.34]{metric}.
        \end{remark}
         The desired Riemannian metric is defined as follows.         
         \begin{definition}[The Grassmann metric]\label{def:riemmetric}We define the Riemannian metric on $G(k,n)$ as the unique one such that $p:O(n)\to G(k,n)$ becomes a Riemannian submersion, i.e. such that for every $\mathbf{E}\in G(k,n)$ the map
        \begin{align}
            D_{\widehat{E}}p: (\mathrm{Ker}(D_{\widehat{E}}p))^{\perp} \longrightarrow T_{\mathbf E}G(k,n)
        \end{align}
        is a linear isometry (this definition is well posed, see \cite[Proposition 4.50]{metric}). We will refer to this metric as the \emph{Grassmann metric}.
        \end{definition}
        More explicitly, choose a representative $\widehat{E}\in p^{-1}(\mathbf{E})$, and denote by $A_1, A_2\in \R^{(n-k)\times k}$ matrices representing the vectors $v_{1}, v_{2}\in T_{\mathbf{E}}G(k,n)$ under the isomorphism from \cref{def:vA}. Then, using \eqref{eq:isomAMA}, one easily sees that the Riemannian structure form the previous definition is the scalar product defined on $T_\mathbf{E}G(k,n)$ by
        \[\label{eq:gE}g_{\mathbf{E}}(v_1, v_2):=\langle A_1, A_2\rangle_{\mathrm{F}}.\]
       
       \subsubsection{Geodesics and distance function} 
Now we recall the  description of Riemannian geodesics on $G(k,n)$ with respect to the Grassmann metric we just introduced. Observe that the projection map $p:O(n) \longrightarrow G(k,n)$  does not preserve geodesics in general. Nevertheless, if a geodesic $\gamma(t)$ on $O(n)$ is horizontal, which means that at each point $\dot \gamma(t) \in (\mathrm{Ker}(D_{\gamma(t)}p))^{\perp}$, then the projected curve $p(\gamma(t))$ is a geodesic in $G(k,n)$. Viceversa, every geodesic in $G(k,n)$ admits a horizontal lift. Therefore we can exploit the knowledge of the geodesics on $O(n)$ with respect to the Frobenius metric to compute the exponential map on $G(k,n)$. 
        
        To this end recall that the Riemannian exponential map of $O(n)$ at the identity coincides with  the usual matrix exponential (\cite[Exercise 3, Chapter 3]{doCarmo}), defined for any square matrix $X$ as 
        \begin{align}
            e^X:=\sum_{i=0}^{+\infty}\frac{X^i}{i!}.
        \end{align}
        Given $\mathbf E \in G(k,n)$, the Riemannian exponential map of $G(k,n)$ at $\mathbf E$ is 
        \begin{align}
            \exp_{\mathbf E}:T_{\mathbf E}G(k,n) &\longrightarrow G(k,n) \\
            v &\longmapsto \gamma_v(1),
        \end{align}
        where $\gamma_v$ is the geodesic starting at $\mathbf E$ with velocity $v$. We want to give now an explicit description of this map.
        
        For any $R \in O(n)$ we denote by $\exp_R^{O(n)}$ the Riemannian exponential map of $O(n)$ at the point $R$. Let us fix a representative $\widehat{E}\in p^{-1}(\mathbf{E})$ and let us identify the tangent space $T_\mathbf{E}G(k,n)$ with $\R^{(n-k)\times k}$ as in \cref{def:vA}. Then we have 
        \begin{align}
            \exp_{\mathbf E}(v_A)=& \exp_{\mathbf E}\big(D_{\widehat{E}}p (\widehat{E}M_A)\big) = p\big(\exp_{\widehat{E}}^{O(n)}(\widehat{E}M_A)\big) = \\ &p\big(\widehat{E} \exp_{\mathds 1}^{O(n)}(M_A)\big)=\widehat{E}e^{M_A}\cdot \mathbf E_0.
        \end{align}
        Let now $A=U\Sigma V^{\top}$ be a SVD of $A$ with $\Sigma=\mathrm{diag}(\mu_1,\dots,\mu_k)\in \R^{(n-k)\times k}$. Then we have 
        \begin{align}\label{M_ASVD}
            M_A=\mleft(\begin{array}{c|c}
                V & 0 \\ \hline
                0 & U
            \end{array}\mright)\mleft(\begin{array}{c|c}
                0 & -\Sigma^{\top} \\ \hline
                \Sigma & 0
            \end{array}\mright)\mleft(\begin{array}{c|c}
                V & 0 \\ \hline
                0 & U
            \end{array}\mright)^{\top} =: R_A M_{\Sigma}{R_A}^{\top}.
        \end{align}
        Remark that if $R \in O(n)$, then for every $B \in \R^{n\times n}$ we have $e^{RBR^{\top}}=Re^BR^{\top}$. From \eqref{M_ASVD} it follows
        \begin{align}
            e^{M_A}=R_Ae^{M_{\Sigma}}{R_A}^{\top}.
        \end{align}
        By a straightforward computation, using the Taylor expansions of the sine and cosine functions, we see that 
        \begin{align}\label{eq:lift}
            e^{M_{\Sigma}}=\mleft(\begin{array}{ccc|ccc|c}
                 \cos(\mu_1) & & & -\sin(\mu_1) & & &  \\
                 & \ddots & & & \ddots & & 0 \\
                 & &\cos(\mu_k) & & &-\sin(\mu_k) &  \\ \hline 
                 \sin(\mu_1) & & & \cos(\mu_1) & & &  \\
                 & \ddots & & & \ddots & & 0 \\
                 & & \sin(\mu_k) & & & \cos(\mu_k) &  \\ \hline 
                  & 0 & & & 0 & & 0
            \end{array}\mright).
        \end{align}
        From our computation we get the following description of geodesics  and the Riemannian exponential map in $G(k,n)$. 
        \begin{proposition}\label{geodesicprop}
            Let $\mathbf E \in G(k,n)$ and choose a representative $\widehat{E}\in p^{-1}(\mathbf{E})$. Let also $v_A \in T_{\mathbf E}G(k,n)\simeq \R^{(n-k)\times k}$ be the tangent vector given by the isomorphism \cref{def:vA}. Let $A=U\Sigma V^{\top}$ be a SVD of $A$, with $U=(u_1 \dots u_{n-k})$, $V=(v_1 \dots v_k)$ and $\Sigma=\mathrm{diag}(\mu_1,\dots,\mu_k)\in \R^{(n-k)\times k}$. Denote by $\gamma_A:(-\infty,+\infty) \longrightarrow G(k,n)$ the geodesic such that $\gamma_A(0)=\mathbf E$ and $\dot \gamma_A(0)=v_A$. Then 
            \begin{align}\label{explicitgeodesic}
                \gamma_A(t)=\left[\widehat{E} \mleft(\begin{array}{ccc}
                    \cos(t\mu_1)v_1 & \dots & \cos(t\mu_k)v_k \\ 
                    \sin(t\mu_1)u_1 & \dots & \sin(t\mu_k)u_k
                \end{array}\mright)\right].
            \end{align}
            In particular,
            \[\mathrm{exp}_\mathbf{E}(v_A)=\left[\widehat{E} \mleft(\begin{array}{ccc}
                    \cos(\mu_1)v_1 & \dots & \cos(\mu_k)v_k \\ 
                    \sin(\mu_1)u_1 & \dots & \sin(\mu_k)u_k
                \end{array}\mright)\right].\]
        \end{proposition}
      \begin{proof}As we already observed, geodesic in $G(k,n)$ coincides with projections of horizontal geodesics under $p:O(n)\to G(k,n)$, i.e. of curves of the form \eqref{eq:lift}.
      \end{proof}
                
        \subsubsection{Principal angles}\label{sec:angles}In this section we recall the notion of \emph{principal vectors} and \emph{principal angles}, following the algorithmic definition from \cite{Lim}. 

\begin{definition}[Principal vectors and principal angles]\label{def:principal}Let $\mathbf E_1$, $\mathbf E_2 \in G(k,n)$. The principal vectors $(p_1,q_1),\dots,(p_k,q_k)\in \mathbf{E}_1\times \mathbf{E}_2$ and the principal angles $0\leq \theta_1\leq \cdots \leq \theta_k\leq \frac{\pi}{2}$ between $\mathbf{E_1}$ and $\mathbf{E}_2$ are defined by the following recursive algorithm. The couple $(p_1, q_1)$ is defined as a solution to the optimization problem: 
\begin{align}
    \text{maximize}& \ \ \langle p,q \rangle \\
    \text{subject to}& \ \ p \in \textbf E_1, \ \norm{p}=1 \\
    & \ \  q \in \textbf E_2, \ \norm{q}=1,
\end{align}
where $\langle \cdot, \cdot\rangle$ is the Euclidean scalar product of $\R^n$. 
Inductively, for $i=2, \ldots, k$, the couple $(p_i,q_i)$ is defined as a solution to the optimization problem
\begin{align}
    \text{maximize}& \ \ \langle p,q \rangle \\
    \text{subject to}& \ \ p \in \textbf E_1, \ p \perp \{p_1, \dots, p_{i-1}\}, \ \norm{p}=1 \\
    & \ \  q \in \textbf E_2, \ q\perp\{q_1,\dots,q_{i-1}\}, \ \norm{q}=1,
\end{align}
The obtained couples $(p_1,q_1),\dots,(p_k,q_k)$ are called \emph{principal vectors} between $\mathbf E_1$ and $\mathbf E_2$. Recall that maximizing the scalar product between unit vectors is the same as minimizing the angle between them. It follows that principal vectors provide orthonormal basis for $\mathbf E_1$ and $\mathbf E_2$ such that the angles between corresponding vectors are the least possible. These angles are denoted by
\begin{align}
    \theta_i := \arccos(\langle p_i,q_i\rangle), \quad i=1, \ldots, k,
\end{align}
are called \emph{the principal angles} between $\mathbf E_1$ and $\mathbf E_2$. 
\end{definition}

It is clear that the principal angles satisfy $0\leq \theta_1\leq\dots\leq\theta_k\leq\frac{\pi}{2}$. Moreover we can observe the following properties: \begin{itemize}[label={--}]
    \item if there are exactly $\ell$ principal angles that are $0$, i.e. $\theta_1=\dots=\theta_\ell=0$ but $\theta_{\ell+1}\neq 0$, then $\mathrm{dim}(\mathbf E_1 \cap \mathbf E_2)=\ell$; 
    \item if there are exactly $\ell$ principal angles that are $\frac{\pi}{2}$, i.e. $\theta_{k-\ell+1}=\dots=\theta_k=\frac{\pi}{2}$ but $\theta_{k-\ell}\neq \frac{\pi}{2}$, then $\mathrm{dim}(\mathbf E_1 \cap \mathbf E_2^{\perp})=\ell$; 
    \item the planes $\pi_i=\mathrm{span}\{p_i,q_i\}$ are pairwise orthogonal for $i=1,\dots,k$ (in case $p_i=q_i$ the plane $\pi_i$ is just a line). 
\end{itemize} 
To actually compute principal vectors and principal angles between two $k$--planes, we use again the singular values decomposition. Assume that $\mathbf E_1 = [a_1 \dots a_k]$ and $\mathbf E_2 = [b_1 \dots b_k]$ and let $A,B \in \R^{n\times k}$ be the matrices with columns $a_i$'s and $b_i$'s. Consider a SVD of the product 
\begin{align}\label{anglesSVD}
    A^{\top}B=U\Sigma V^{\top}
\end{align}
with $U,V \in O(k)$ and $\Sigma=\mathrm{diag}(\sigma_1,\dots,\sigma_k)\in \R^{k\times k}$ with singular values ordered as $\sigma_1 \geq \dots \geq \sigma_k \geq 0$. Then \cite{Bjorck}
\begin{itemize}[label={--}]
    \item $AU=(p_1 \dots p_k)$ contains as columns the principal vectors for $\mathbf E_1$; 
    \item $BV=(q_1 \dots q_k)$ contains as columns the principal vectors for $\mathbf E_2$;
    \item $\sigma_i = \cos(\theta_i)$ for $i=1,\dots,k$, where $0\leq \theta_1\leq\dots\leq\theta_k\leq \frac{\pi}{2}$ are the principal angles between $\mathbf E_1$ and $\mathbf E_2$. 
\end{itemize}
Therefore given two $k$-planes $\mathbf E_1$ and $\mathbf E_2 \in G(k,n)$, to compute the principal vectors and angles between them it is enough to pick two specific orthonormal basis, stack them into matrices representing the planes, take the product between them as in \eqref{anglesSVD} and run an SVD to find the changes of basis needed to get principal vectors and to find the cosines of the principal angles. 
\subsubsection{Length minimizing geodesics and the Grassmann distance}Using principal vectors and angles, it is also easy to describe the length minimizing geodesic connecting two points in $G(k,n)$ and the function giving the distance between them.

To start with, we give an answer to the following question: given $\mathbf E_1$, $\mathbf E_2 \in G(k,n)$, how can we compute a tangent vector $v_A \in T_{\mathbf E_1}G(k,n)$ such that $\exp_{\mathbf E_1}(v_A)=\mathbf E_2$? 

This is done as follows. Given the  two $k$--planes $\mathbf{E}_1=[a_1\cdots a_k]$ and $\mathbf{E}_2=[b_1\cdots b_k]$, fix a representative $\widehat{E}_1=(a_1,\dots,a_k,w_1,\dots,w_{n-k}) \in O(n)$ for $\mathbf E_1$, and take a decomposition as in \eqref{anglesSVD}. Define $\Theta:=\mathrm{diag}(\theta_1,\dots,\theta_k) \in \R^{(n-k)\times k}$, the diagonal matrix containing the principal angles  between $\mathbf{E}_1$ and $\mathbf{E}_2$ on the diagonal. Recall that the matrix  $U$ from \eqref{anglesSVD} gives the change of basis that one needs to apply to the orthonormal basis $\{a_1, \dots, a_k\}$ of $\mathbf E_1$ to obtain the principal vectors $\{p_1,\dots,p_k\}$. Assume that $\theta_1=\dots=\theta_\ell=0$, so that $p_1=q_1,\dots,p_\ell=q_\ell$  and $\mathbf E_1 \cap \mathbf E_2$ has basis $\{p_1,\dots,p_\ell\}$. For every $i=\ell+1,\dots,k$ define $n_i \in \pi_i=\mathrm{span}\{p_i,q_i\}$ such that $\norm{n_i}=1$, $n_i \perp p_i$ and $\langle n_i,q_i\rangle >0$. Then by definition of principal angles we have 
\begin{align}\label{rotationvector}
    q_i = p_i\cos(\theta_i) + n_i \sin(\theta_i).
\end{align}
Moreover, by the orthogonality relations for principal vectors, we also have $n_i \in \mathbf E_1^{\perp}$ for every $i=\ell+1,\dots,k$. Complete the set $\{n_{\ell+1},\dots,n_{k}\}$ to an orthonormal basis $\{n_1,\dots,n_{n-k}\}$ of $\mathbf E_1^{\perp}$  and define $S \in O(n-k)$ giving the change of basis on $\mathbf E_1^{\perp}$ from $\{w_1,\dots,w_{n-k}\}$ to $\{n_1,\dots,n_{n-k}\}$. We define 
\begin{align}\label{explicitmatrix}
    A:=S\Theta U^{\top}.
\end{align}
Using \eqref{explicitgeodesic} and \eqref{rotationvector}, we can see that 
\begin{align}\label{explicitmatrixexp}
    \exp_{\mathbf E_1}(v_A)=\mathbf E_2.
\end{align}
Geometrically, the geodesic $\gamma_A(t):=\exp_{\mathbf E_1}(tv_A)$ just consists of simultaneous rotations from the vectors $p_i \in \mathbf E_1$ to the corresponding vectors $q_i \in \mathbf E_2$, performed inside the planes $\pi_i=\mathrm{span}\{p_i,q_i\}$ for $i=\ell+1,\dots,k$ (nothing is done for $p_1,\dots,p_\ell$ which are in $\mathbf E_1 \cap \mathbf E_2$) with speed given by the angles $\theta_i$. Remark that, if in the definition of $\Theta$ we replace an angle $\theta_i$ with $\theta_i+m\pi$ for any $m \in \Z$, the equation \eqref{explicitmatrixexp} still holds true, but the corresponding geodesic will have greater length for any $m \neq 0$. 

One can show that the geodesic $\gamma_A(t)$ constructed above provides the length minimizing geodesic between $\mathbf E_1$ and $\mathbf E_2$. 
It follows that the Riemannian distance between $\mathbf E_1$ and $\mathbf E_2$ is given by the Frobenius norm of $A$ and this can be expressed in terms of the principal angles between them.
\begin{definition}[Grassmann distance]\label{def:Gdist}We denote by $\dist:G(k,n)\times G(k,n)\to \R$ the Riemannian distance function induced by the Grassmann metric (\cref{def:riemmetric}) and we call this function \emph{the Grassmann distance}. More explicitely, if $\theta_1, \ldots, \theta_k$ are the principal angles between the two $k$--planes $\mathbf E_1$ and $\mathbf E_2$, then
\begin{align}\label{anglesdistance}
    \dist(\mathbf E_1,\mathbf E_2)=\bigg(\sum_{i=1}^k \theta_i^2\bigg)^{\frac{1}{2}}.
\end{align}
\end{definition}

Let us introduce now the notation
\begin{align}\label{Rpi2}
    R_{\frac{\pi}{2}}=\left\{A \in \R^{(n-k)\times k} \ | \ \text{the singular values of $A$ are bounded by $\frac{\pi}{2}$}\right\} \subseteq \R^{(n-k)\times k}
\end{align}
for the set of matrices with bounded singular values. Identifying $R_{\frac{\pi}{2}}$ with a subset of $T_{\mathbf{E}}G(k,n)$ using \cref{def:vA}, we can summarize the above discussion in the following result. 
\begin{proposition}\label{boundedprop}
    The image of $R_{\frac{\pi}{2}}$ under $\exp_{\mathbf E}:T_{\mathbf E}G(k,n) \longrightarrow G(k,n)$ is the whole Grassmannian. Moreover, for every $A \in R_{\frac{\pi}{2}}$, the principal angles between $\mathbf E$ and $\exp_{\mathbf E}(v_A)$ coincide with the singular values of $A$. 
\end{proposition}
\begin{remark}The content of the previous proposition can be visualized using the following commutative diagram that, for a fixed $\mathbf{E}$, relates the Frobenius distance on the tangent space and the Grassmann distance:
$$
\begin{tikzcd}
R_{\frac{\pi}{2}} \arrow[rr, "\exp_{\mathbf{E}}"] \arrow[rd, "\|\cdot\|_{\mathrm{F}}"'] &    & {G(k,n)} \arrow[ld, "{\dist(\mathbf{E}, \cdot)}"] \\
                                                                               & \R &                                                  \end{tikzcd},\quad \dist(\mathbf{E}, \mathrm{exp}_\mathbf{E}(v_A))=g_\mathbf{E}(v_A, v_A)^{\frac{1}{2}}=\|A\|_{\mathrm{F}}.$$
                                                                               \end{remark}
    
%\begin{remark}[Orthogonally invariant distances]From \eqref{anglesdistance} we see that the Grassmann distance depends on the principal angles bewteen the points. By \cite{Bjorck} we have that the cosines of the angles can be computed via a SVD, which is a semialgebraic operation in the coordinates of the two points in $G(k,n)$. Nevertheless, to obtain the angles we need to apply the arccosine function, which is not semialgebraic. It follows that the Grassmann distance is not semialgebraic in the coordinates of the points (see also \cref{sec:nonsemi}).
%\end{remark}
    
\subsection{O--minimality of principal angles}\label{ominimalsection}The Grassmannian $G(k,n)$ can be embedded in $\R\mathrm{P}^N$, with $N={n\choose k}-1$, via the Pl\"ucker embedding. The image of this embedding is a real algebraic set (in particular semialgebraic). As we noted above, the principal angles \emph{are not} semialgebraic functions (we will show in \cref{sec:nonsemi} that the whole nearest point problem to an algebraic set, in general, is not semialgebraic). However, we show in this section that they are definable in the o--minimal structure of globally subanalytic sets.
\begin{proposition}[o--minimality of the Grassmann distance]\label{ominimalitylemma}
    The principal angles are globally subanalaytic and so is the Grassmann distance function. 
\end{proposition}
\begin{proof}
    By definition of globally subanalytic function, it is enough to prove that $\mathrm{graph}(\theta_i(\mathbf L,\cdot))$ is globally subanalytic for every $\mathbf L \in G(k,n)$, every $i=1,\dots,k$ and every $1\leq k\leq n$. By the orthogonal invariance of principal angles, we can fix a specific $k$--plane $\mathbf L_0 \in G(k,n)$ and assume without loss of generality that $\mathbf L=\mathbf L_0$. 
    
    Given a set of $k$ vectors $a_1,\dots,a_k \in \R^n$, denote by $A$ the matrix having those vectors as columns
    \[A=(a_1,\dots,a_k)=\begin{pmatrix}
        A_1 \\ A_2
    \end{pmatrix},\]
    where $A_1$ is a $k\times k$ matrix and $A_2$ is a $(n-k)\times k$ matrix. Then the graph of $\theta_i(\mathbf L_0, \cdot)$ can be described as
    \begin{align}
        \mathrm{graph}(\theta_i(\mathbf L_0,\cdot))&=\{(\mathbf E,t)\in G(k,n)\times \R \ | \ \exists a_1,\dots,a_k \in \R^n, \ \mathbf E=[a_1 \dots a_k], \\& \quad\quad(A_1^{\top}A_1+A_2^{\top}A_2)=\mathds{1}, \ t=\arccos((\lambda_i(A_1A_1^{\top}))^{\frac{1}{2}})\} \\
        &= \{(\mathbf E,t)\in G(k,n)\times \R \ | \ \exists a_1,\dots,a_k \in \R^n, \ \mathbf E =[a_1 \dots a_k], \\& \quad \quad(A_1^{\top}A_1+A_2^{\top}A_2)=\mathds{1}, \ \cos t=\lambda_i(A_1A_1^{\top})^{\frac{1}{2}})\},
    \end{align}
    where $\lambda_i(\cdot)$ denotes the function giving the $i$--th eigenvalue of a matrix, which is a semialgebraic function. Remark that by the orthogonality condition on the columns of $A$, we only need the restriction of the cosine function to the interval $[0,\frac{\pi}{2}]$. Therefore the functions $\theta_i(\mathbf L,\cdot)$ are globally subanalytic. 
    
 The fact that the Grassmann distance is globally subanalytic follows now from \eqref{anglesdistance} and the fact that sums, products and compositions of definable functions in an o--minimal structure are still definable.
\end{proof}

\section{The Grassmann Distance Complexity}\label{finiteness section}
   Assume $X \subseteq G(k,n)$ is a (possibly stratified) submanifold of $G(k,n)$. For a generic point $\mathbf L \notin X$ we want to minimize the Grassmann distance from $\mathbf L$ to $X$. In other words, we are asked to minimize the function
        \[\dist_{\mathbf L}:G(k,n)\to \R, \quad \dist_{\mathbf{L}}(\mathbf E):=\dist(\mathbf{L}, \mathbf{E})\]
        constrained to the set $X$.
        
        If the function $\dist_{\mathbf L}|_{X}$ was smooth, the nearest points to $\mathbf{L}$ in $X$ can be found among the critical points of $\dist_{\mathbf L}|_{X}$. In general, however, $\dist_{\mathbf L}|_{X}$ is not smooth, this is because $\dist_{\mathbf{L}}$ itself is not smooth. In fact, the set of points where $\dist_{\mathbf{L}}$ is not differentiable consists of $\mathbf{L}$ and its cut locus (\cref{propo:distsmooth}).    

Still the function $\dist_{\mathbf L}$ is locally Lipschitz and therefore, by Rademacher's Theorem, it is differentiable almost everywhere. The same is true for $\dist_{\mathbf{L}}|_X$. In particular, following Clarke's \cite{clarke}, we can \emph{define} the critical points of $\dist_{\mathbf L}|_X$ through the notion of its subdifferential (see \cref{def:subdifferential}): we say that a point $\mathbf{E}\in X$ is critical for $\dist_{\mathbf L}|_X$ if its  subdifferential at $\mathbf{E}$ contains the zero vector. With this definition, minima are critical points. 

We will show in \cref{ignoringsection} that  the minimizers we are looking for, for a generic $\mathbf{L}$ cannot be on the cut locus of $\mathbf{L}$ and therefore, for the purposes of the \emph{minimization} problem, we can exclude them (\cref{ignorethecut}). However, the non--differentiability of $\dist_{\mathbf{L}}$  opens the question of the study of the structure of the subdifferential of $\dist_{\mathbf{L}}$ and its restriction to $X$ (for instance, maxima of this function can well be non--differentiability points, even in the generic case). We will take over this study, which is of independent interest, in \cref{subdifferentialsection} and then come back to the nearest point problem in  \cref{ignoringsection}.

\subsection{The subdifferential of the distance function}\label{subdifferentialsection}We begin this section stating the relevant definition. 

\subsubsection{The subdifferential}Here we assume that a locally Lipschitz function is defined on a Riemannian manfiold $X$ (so that it makes sense to talk about ``Lipschitzianity'' and of the gradient of a function $f:X\to \R$). Recall again that locally Lipschitz functions are differentiable almost everywhere, by Rademacher's Theorem, and therefore the following definition makes sense.
   \begin{definition}[The subdifferential and critical points]\label{def:subdifferential}The \textit{subdifferential} of a locally Lipschitz function $f:X\to \R$ at a point $p$ of its domain is defined as 
        \begin{align}\label{subdiffdef}
    \partial_pf := \mathrm{co}\left\{\underset{p_n \rightarrow p}{\mathrm{lim}} \nabla f(p_n) \ | \ (p_n)_n\subseteq \omega(f) \ \text{and the limit exists}\right\},
\end{align}
where $\mathrm{co}$ denotes the convex hull, $\omega(f)$ is the set differentiability points for $f$ and $\nabla f$ is the gradient of $f$.
We say that a point $p$ is critical for $f$ if and only if $0 \in \partial_p f$.
\end{definition}
Note that, if $p \in \omega(f)$, then $\partial_pf=\{\nabla f(p)\}$ so that the notion of subdifferential directly extends that of gradient at points of non-differentiability. 

\subsubsection{Differentiable points of the distance function}To state the basic result concerning differentiability properties of $\dist_{\mathbf{L}}$ we first recall the following definition.

\begin{definition}[Cut locus]\label{def:cut}Let $(M, g)$ be a Riemannian manifold and $p\in M$. The \emph{cut--locus} of $p$, denoted by $\mathrm{cut}(p)$, is the closure of the  set of points in $M$ that are connected to $p$ by two or more distinct shortest geodesics.
\end{definition}

It is a standard fact in differential geometry that, in a Riemannian manifold $(M, g)$, for a given $p\in M$, away from $\{p\}\cup \mathrm{cut}(p)$ the Riemannian distance function from $p$ is smooth with gradient at $x$ given by the tangent vector to the unique length minimizing geodesic joining $p$ with $x$.
From this we get the following proposition.

\begin{proposition}\label{propo:distsmooth} The function $\dist_{\mathbf L}$ is smooth on $G(k,n) \setminus (\mathrm{cut}(\mathbf L) \cup \mathbf L)$. Moreover, at a differentiability point $\mathbf E$ the gradient of $\dist_{\mathbf L}$ is given by the tangent vector to the unique length minimizing geodesic going from $\mathbf L$ to $\mathbf E$.
\end{proposition}

It is known \cite{WongGrassmann} that the cut locus of $\mathbf L \in G(k,n)$ with the Grassmann metric consists of the planes $\mathbf E$ having at least one principal angle with $\mathbf L$ equal to $\frac{\pi}{2}$:
\begin{align}\label{cutlocus}
        \mathrm{cut}(\mathbf L)=\mleft\{\mathbf E \in G(k,n) \ \bigg| \ \theta_k(\mathbf L, \mathbf E) = \frac{\pi}{2}\mright\}.
    \end{align}
    The cut locus of $\mathbf{L}$ is in fact an algebraic variety.
    
    \begin{lemma}\label{cutlemma}
        For any point $\mathbf L \in G(k,n)$ the cut locus $\mathrm{cut}(\mathbf L)$ is a Schubert variety of codimension one. 
    \end{lemma}
    \begin{proof}
        By the properties of principal angles, the condition in \eqref{cutlocus} is equivalent to the fact $\mathrm{dim}(\mathbf E \cap \mathbf L^{\perp}) \geq 1$. Considering a complete flag where the element of dimension $n-k$ is given by $\mathbf L^{\perp}$, we can see that 
        \begin{align}\label{cut}
            \mathrm{cut}(\mathbf L)=\mleft\{\mathbf E \in G(k,n) \ | \ \mathrm{dim}(\mathbf E \cap \mathbf L^{\perp})\geq 1\mright\}
        \end{align}
        is the Schubert variety associated to the Young diagram consisting of just one box with respect to the chosen flag. Therefore it is of codimension one (see \cite{Fultontableux} for details on Schubert varieties and Young diagrams).
    \end{proof}
The advantage of the description given by \cref{cutlemma} is that we can use algebraic geometry to stratify the cut locus as 
\begin{align}\label{cutstrata}
        \mathrm{cut}(\mathbf L) = \bigsqcup_{i=1}^k\Omega(j)
    \end{align}
    where 
    \begin{align}\label{cutstrataexplicit}
        \Omega(j)=\mleft\{\mathbf E \in G(k,n) \ | \ \mathrm{dim}(\mathbf E \cap \mathbf L^{\perp})=j\mright\}.
    \end{align}
    Each $\Omega(j)$ is the smooth part of the Schubert variety associated to the Young diagram corresponding to the partition $(j,\dots,j,0,\dots,0)$ where the first $j$ entries are $j$, with respect to a flag containing $\mathbf L^{\perp}$ as the $n-k$-dimensional element; in particular $\Omega(1)$ is the smooth locus of $\mathrm{cut}(\mathbf L)$. Moreover we see that the codimension of each strata $\Omega(j)$ in $G(k,n)$ is $j^2$. In terms of principal angles, the strata can be described as 
    \begin{align}
        \Omega(j)=\mleft\{\mathbf E \in G(k,n) \ | \ \theta_k(\mathbf L,\mathbf E)=\dots = \theta_{k-j+1}(\mathbf L,\mathbf E)=\frac{\pi}{2}>\theta_{k-j}(\mathbf L,\mathbf E)\mright\}.
    \end{align}
    Another advantage of describing the strata of the cut locus as the smooth loci of some Schubert varieties, is that we can explicitly describe the tangent space at smooth points of them. More precisely, if we have a  Schubert variety $\Omega \subseteq G(k,n)$ defined by the single condition 
    \begin{align}
        \Omega = \{\mathbf E \in G(k,n) \ | \ \mathrm{dim}(\mathbf E \cap \mathbf V_i) \geq r\}
    \end{align}
    for some fixed $i$--dimensional subspace $\mathbf V_i$, then at a smooth point $\mathbf E \in \Omega$ (i.e. a point such that $\mathrm{dim}(\mathbf E \cap \mathbf V_i)=r$), using the description of tangent vectors as linear maps, we have \cite[Proposition 4.3]{Feher}
\begin{align}\label{tangentschubert}
        T_{\mathbf E}\Omega = \{\varphi:\mathbf E \longrightarrow \mathbf E^{\perp} \ | \ \varphi(\mathbf E \cap \mathbf V_i) \subseteq \mathbf E^{\perp} \cap \mathbf V_i\},
    \end{align}
    that is, under the isomorphism from \cref{def:hom}, those maps who preserve the intersection between $\mathbf E$ and $\mathbf V_i$ still inside $\mathbf V_i$. \\
    \subsubsection{The structure of the subdifferential of the Grassmann distance.}The following fact is elementary and can be proved exactly as in \cite[Proposition 2.1.2]{generalizedmorse}.
    \begin{proposition}\label{propo:asin}Let $\mathbf{L}\in G(k,n)$. For every $\mathbf S \in G(k,n)$
    \begin{align}\label{subdiff}
        \partial_{\mathbf S}\delta_{\mathbf L} = \mathrm{co}\mleft\{D_v\exp_{\mathbf L}\mleft(\frac{v}{\norm{v}}\mright) \ \bigg| \ v \in \exp_{\mathbf L}^{-1}(\mathbf S) \cap R_{\frac{\pi}{2}}\mright\},
    \end{align}
    where $\exp_{\mathbf L}$ is the exponential map of $G(k,n)$ at $\mathbf L$ and we use the notation $R_{\frac{\pi}{2}}$ from \eqref{Rpi2}, identifying tangent vectors and matrices via the isomorphism \eqref{matrixiso}. In particular, if $\mathbf S \notin (\mathrm{cut}(\mathbf L) \cup \mathbf L)$, then $\partial_{\mathbf S}\delta_{\mathbf L}=\{\nabla \delta_{\mathbf L}(\mathbf S)\}$.
    \end{proposition}
    
    From \cref{propo:asin} it follows that, in order to get a description of the subdifferential of $\dist_{\mathbf{L}}$ at points  $\mathbf S \in \mathrm{cut}(\mathbf L)$, we need first to understand the set of preimages $\exp_{\mathbf L}^{-1}(\mathbf S)\cap R_{\frac{\pi}{2}}$. These preimages are described by the following proposition.
    
\begin{proposition}\label{preimagelemma}
        Let $\mathbf S \in \Omega(j) \subseteq \mathrm{cut}(\mathbf L)$ and pick a representative $\widehat{L}\in p^{-1}(\mathbf{L})$. Using the isomorphism given by \cref{def:vA}, pick an element $A_0\in \exp_{\mathbf{L}}^{-1}(\mathbf{S})\cap R_{\frac{\pi}{2}}.$ Let $A_0=U\Sigma V^\top$ be an SVD for $A_0$ with $\Sigma=(\theta_1, \ldots, \theta_{k-j}, \tfrac{\pi}{2}, \ldots, \tfrac{\pi}{2})$. Then
        \[\label{AWdef}\exp_{\mathbf{L}}^{-1}(\mathbf{S})\cap R_{\frac{\pi}{2}}=\left\{A_{W}:=U\Sigma \mleft(\begin{array}{cc}
            \mathds 1_{k-j} & 0 \\
             0 & W 
        \end{array}\mright)V^{\top}\,\bigg|\, W\in O(j)\right\}.\] 
        
     Moreover, under the isomorphism of \cref{def:vA}, if $W_1\neq W_2$ the matrices $A_{W_1}$ and $A_{W_2}$ correspond to different geodesics connecting $\mathbf L$ to $\mathbf S$, i.e. for every $0<t<1$ we have $\exp_{\mathbf L}(tv_{A_{W_1}})\neq \exp_{\mathbf L}(tv_{A_{W_2}})$.
        \end{proposition}
\begin{proof}Let $A\in\exp_{\mathbf{L}}^{-1}(\mathbf{S})$ with $A\neq A_0$ and with SVD $A=\widetilde U \widetilde \Sigma \widetilde V^{\top}$. We want to understand the relation between $A_0$ and $A$. 

Note that all the points in $\exp_{\mathbf L}^{-1}(\mathbf S)\cap R_{\frac{\pi}{2}}$ will actually be on the boundary $\partial R_{\frac{\pi}{2}}$ by \cref{boundedprop}.

By the results of \cref{grassmanngeometrysection}, we know that: the columns of $\widetilde V$ correspond to the coordinates of the principal vectors of $\mathbf L$ (with $\mathbf S$) with respect to the basis given by $\widehat L$; the matrix $\widetilde U$ contains the coordinates of the vectors orthogonal to the principal ones of $\mathbf L$ in the planes spanned by the couple of principal vectors (the vectors $n_i$ in \cref{grassmanngeometrysection}); the matrix $\widetilde \Sigma$ contains the principal angles between $\mathbf L$ and $\mathbf S$. 

In particular we must have $\Sigma = \widetilde \Sigma$, while $\widetilde V$ and $\widetilde U$ must correspond to a different choice of principal vectors with respect to $V$ and $U$. We then have to understand how much freedom do we have when choosing principal vectors and when different choices lead to different matrices (compare with \cite[Proposition B.1]{grassmannhandbook}):
\begin{itemize}[label={--}]
        \item if a principal angle has multiplicity $1$, the corresponding pair of principal vectors is uniquely determined and the corresponding columns in $U$ and $V$ will be the same as in $\widetilde U$ and $\widetilde V$; 
        \item if a principal angle $\mu < \frac{\pi}{2}$ has multiplicity $m$, the corresponding principal vectors are determined up to an orthogonal transformation in $O(m)$, which has to be the same on both spaces. More explicitly, if $\{p_{i+1},\dots,p_{i+m}\}$ and $\{q_{i+1},\dots,q_{i+m}\}$ are principal vectors corresponding to the angle $\mu$, any other choice of principal vectors for the same angle $\{\widetilde p_{i+1},\dots,\widetilde p_{i+m}\}$ and $\{\widetilde q_{i+1},\dots,\widetilde q_{i+m}\}$ must be of the form 
        \begin{align}
            \mleft(\widetilde p_{i+1} \dots \widetilde p_{i+m}\mright) &= \mleft(p_{i+1} \dots  p_{i+m}\mright)\ Q \\
            \mleft(\widetilde q_{i+1} \dots \widetilde q_{i+m}\mright) &= \mleft(q_{i+1} \dots  q_{i+m}\mright)\ Q
        \end{align}
        for some $Q \in O(m)$. The fact that the matrix is the same for both the $p$'s and the $q$'s implies that this different choice of principal vectors in the SVD has no effect, as the matrix $Q$ cancels out when multiplying back $\widetilde U$ and $\widetilde V$ (in order to obtain the matrix $A$); 
        \item if an angle of $\frac{\pi}{2}$ has multiplicity $j$ and $\{p_{k-j+1},\dots,p_k\}$ and $\{q_{k-j+1},\dots,q_k\}$ are principal vectors corresponding to that angle, then any vector in $\mathrm{span}\{p_{k-j+1},\dots,p_k\}$ is orthogonal to any vector in $\mathrm{span}\{q_{k-j+1},\dots,q_k\}$. This means that if $\{\widetilde p_{k-j+1},\dots, \widetilde p_k\}$ and $\{\widetilde q_{k-j+1},\dots,\widetilde q_k\}$ are another choice of principal vectors corresponding to the angle of $\frac{\pi}{2}$, the orthogonal changes to apply to the original basis to get the new one need not be related. In other words 
        \begin{align}
            \mleft(\widetilde p_{k-j+1} \dots \widetilde p_{k}\mright) &= \mleft(p_{k-j+1} \dots  p_{k}\mright)\ W_1 \\
            \mleft(\widetilde q_{k-j+1} \dots \widetilde q_{k}\mright) &= \mleft(q_{k-j+1} \dots  q_{k}\mright)\ W_2
        \end{align}
        for some $W_1, W_2 \in O(j)$ arbitrary. This change now produces a matrix $A$ which is different from the original $A_0$. Denoting by $\Theta=\mathrm{diag}(\theta_1,\dots,\theta_{k-j})\in \R^{(k-j)\times(k-j)}$ we have  
        \begin{align}
            A=\widetilde U \Sigma \widetilde V = & \   U\mleft(\begin{array}{ccc}
                \mathds 1_{k-j} & 0 & 0 \\
                0 & W_1 & 0 \\ 
                0 & 0 & \mathds 1_{n-2k}
            \end{array}\mright)\mleft(\begin{array}{cc}
                 \Theta & 0 \\
                 0 & \frac{\pi}{2}\mathds 1_j \\
                 0 & 0 
            \end{array}\mright)\mleft(\begin{array}{cc}
                \mathds 1_{k-j} & 0 \\ 
                 0 & W_2^{\top}
            \end{array}\mright)V^{\top} = \\
            = & \ U \Sigma \mleft(\begin{array}{cc}
                 \mathds 1_{k-j} & 0 \\
                 0 & W_1W_2^{\top}
            \end{array}\mright) V^{\top}.
        \end{align}
        Remark that what matters in order to have a different matrix are not the specific orthogonal changes $W_1$ and $W_2$, but rather the relative change in the principal vectors given by $W_1W_2^{\top}$. 
    \end{itemize}
      It follows from this  that any $A \in \partial R_{\frac{\pi}{2}}$ such that $\exp_{\mathbf L}(v_A)=\mathbf S$ must be of the form 
    \begin{align}\label{AWdef}
        A=U\Sigma \mleft(\begin{array}{cc}
            \mathds 1_{k-j} & 0 \\
             0 & W 
        \end{array}\mright)V^{\top}
    \end{align}
    for some $W \in O(j)$, where $A_0=U \Sigma V^{\top}$ is a fixed matrix in $\partial R_{\frac{\pi}{2}}$ such that $\exp_{\mathbf L}(v_{A_0})=\mathbf S$. Viceversa any $W \in O(j)$ gives a different matrix A.
    
    The second part of the statement is immediate to verify and we leave it to the reader.
\end{proof}
We proceed now with the description of the subdifferential. This is slightly more complicated than the formulation of the previous result, because we need to ``compose'' with the differential of the exponential map.

\begin{theorem}\label{subdiffpropo}
        Let $\mathbf S \in \Omega(j) \subseteq \mathrm{cut}(\mathbf L)$ and pick a representative $\widehat S \in p^{-1}(\mathbf S)$. Using the isomorphism given in \cref{def:vA}, pick an element $B_0\in \exp_{\mathbf S}^{-1}(\mathbf L)\cap R_{\frac{\pi}{2}}.$ Let $ U \Sigma V^{\top}$ be an SVD for $B_0$ and define the matrix         \begin{align}\label{Cdef}
        C:=& \mleft(\begin{array}{cc}
             V & 0 \\
             0 &  U \Sigma
        \end{array}\mright), 
        \end{align}
        and, for every $W\in \R^{j\times j}$, the matrix 
        \begin{align}
        \overline{ W }:=& \mleft(\begin{array}{cc}
            0 & -\mleft(\begin{array}{cc}
                \mathds 1_{k-j} & 0  \\
                0 & W
            \end{array}\mright) \\ \label{bardef}
            \mleft(\begin{array}{cc}
                \mathds 1_{k-j} & 0  \\
                0 & W^{\top}
            \end{array}\mright) & 0
        \end{array}\mright).
    \end{align}

Then the subdifferential of the distance function from $\mathbf L$ at the point $\mathbf S$ is given by 
        \begin{align}\label{finalsubdiff}
            \partial_{\mathbf S}\delta_{\mathbf L} = -\frac{1}{\delta(\mathbf L, \mathbf S)}D_{\widehat S}p\mleft(\widehat S C \mleft\{\overline{W} \, | \,  \text{$W$ has singular values $\leq 1$} \ \mright\}C^{\top}\mright).
        \end{align}
    \end{theorem}
    \begin{proof}\cref{preimagelemma} gives a description of all the tangent vectors in $T_{\mathbf L}G(k,n)$ corresponding to length minimizing geodesics connecting $\mathbf L$ and $\mathbf S$ for any $\mathbf S \in \mathrm{cut}(\mathbf L)$. The expression \eqref{subdiff} for a point $\mathbf S \in \Omega(j)$ thus becomes
    \begin{align}\label{Wsubdiff}
        \partial_{\mathbf S}\delta_{\mathbf L} = \mathrm{co}\mleft\{D_{v_{A_W}}\exp_{\mathbf L}\mleft(\frac{v_{A_W}}{\norm{v_{A_W}}}\mright) \,\bigg|\, \ W \in O(j)\mright\}.
    \end{align}
     Denote the length minimizing geodesic corresponding to the tangent vector $v_{A_W}$ as $\gamma_W$, i.e. $\gamma_W(t):= \exp_{\mathbf L}(tv_{A_W})$, $\gamma_W(0)=\mathbf L$ and $\gamma_W(1)=\mathbf S$. Then we can compute the vectors in \eqref{Wsubdiff} as follows. Define the curve $\sigma:(-\varepsilon,\varepsilon)\longrightarrow T_{\mathbf L}G(k,n)$ given by
    \begin{align}
        \sigma(t):= v_{A_W} + t\frac{v_{A_w}}{\norm{v_{A_W}}}.
    \end{align}
    Then $\sigma(0)= v_{A_W}$ and $\dot\sigma(0)=\frac{v_{A_W}}{\norm {v_{A_W}}}$ so that $\sigma$ is an adapted to curve to $\frac{v_{A_W}}{\norm {v_{A_W}}} \in T_{v_{A_W}}(T_{\mathbf L}G(k,n)) \cong T_{\mathbf L}G(k,n)$. We can compute the differential as 
    \begin{align}
        D_{v_{A_W}}\exp_{\mathbf L}\mleft(\frac{v_{A_W}}{\norm{v_{A_W}}}\mright)& = \frac{d}{dt}\exp_{\mathbf L}(\sigma(t))\bigg|_{t=0} = \frac{d}{dt}\exp_{\mathbf L}\mleft(v_{A_W} + t\frac{v_{A_w}}{\norm{v_{A_W}}}\mright)\bigg|_{t=0} \\
        & = \frac{d}{dt}\exp_{\mathbf L}\mleft(\mleft(1+\frac{t}{\norm{v_{A_W}}}\mright)v_{A_W}\mright)\bigg|_{t=0} \\
        & = \frac{d}{dt}\gamma_W\mleft(1+\frac{t}{\norm{v_{A_W}}}\mright)\bigg|_{t=0} \\
        & = \frac{1}{\norm{v_{A_W}}}\ \dot \gamma_W(1),
    \end{align}
    which is the tangent vector to the geodesic $\gamma_W$ at the point $\mathbf S$, rescaled by the speed. The subdifferential \eqref{Wsubdiff} can be written as 
    \begin{align}\label{Wsubdiffbis}
        \partial_{\mathbf S}\delta_{\mathbf L} = \mathrm{co}\mleft\{ \frac{1}{\norm{v_{A_W}}}\ \dot \gamma_W(1) \,\bigg|\,  W \in O(j)\mright\}.
    \end{align}
    
    Now we invert the motion along the geodesics to go from $\mathbf S$ to $\mathbf L$. Define $\widetilde \gamma_W(t) := \gamma_W(1-t)$, so that $\widetilde \gamma_W(0)=\mathbf S$, $\widetilde \gamma_W(1)=\mathbf L$ and $\overset{.}{\widetilde \gamma}_W(t)=-\dot \gamma_W(1-t)$, which implies that $\widetilde \gamma_W$ is a lenght minimizing geodesic connecting $\mathbf S$ to $\mathbf L$. Remark that the speed of $\widetilde \gamma_W$ is the same as that of $\gamma_W$. Therefore $\dot{\gamma}_W(1)=-\overset{.}{\widetilde \gamma}_W(0)$ and the subdifferential \eqref{Wsubdiffbis} can be described as the convex hull of the tangent vectors in $T_{\mathbf S}G(k,n) \cap R_{\frac{\pi}{2}}$ corresponding to length minimizing geodesics from $\mathbf S$ to $\mathbf L$, rescaled by their norm and multiplied by $-1$.
    
The computation of such vectors is completely analogous to that we carried out in \cref{preimagelemma}. For every $W \in O(j)$ define 
    \begin{align}
        B_W := U\Sigma \mleft(\begin{array}{cc}
            \mathds 1_{k-j} & 0 \\
             0 & W^\top 
        \end{array}\mright)V^{\top}.
    \end{align}
By \cref{preimagelemma} we have
   \begin{align}
        \exp_{\mathbf S}^{-1}(\mathbf L) \cap R_{\frac{\pi}{2}} = \mleft\{v_{B_W} \ | \ W \in O(j) \mright\},
    \end{align}
    the reason being that if $\mathbf S$ is in the $j$-th stratum of $\mathrm{cut}(\mathbf L)$, viceversa $\mathbf L$ is in the $j$-th stratum of $\mathrm{cut}(\mathbf S)$ by symmetry of principal angles. 
    Remark that 
    \begin{align}
        \norm{v_{A_W}} = \delta(\mathbf L, \exp_{\mathbf L}(v_{A_W}))=\delta(\mathbf L, \mathbf S),
    \end{align}
    so that \eqref{Wsubdiffbis} becomes
    \begin{align}
       \partial_{\mathbf S}\delta_{\mathbf L} = -\frac{1}{\delta(\mathbf L, \mathbf S)}\ \mathrm{co}\mleft\{ v_{B_W} \ | \ W \in O(j) \mright\}.
    \end{align}
    
    Recall that, by the construction of the isomorphism from \cref{def:vA}, we have 
    \begin{align}
    v_B = D_{\widehat S}p\mleft(\widehat S\mleft(\begin{array}{c|c}
            0 & -B^{\top} \\ \hline
            B & 0
        \end{array}\mright)\mright) = D_{\widehat S}p(\widehat S M_B),
    \end{align}
    where the matrix $M_B$ is defined in  \eqref{eq:MA}.
    Using the linearity of the differential we obtain 
    \begin{align}
        \partial_{\mathbf S}\delta_{\mathbf L}&=-\frac{1}{\delta(\mathbf L, \mathbf S)}\ \mathrm{co}\mleft\{D_{\widehat S}p(\widehat S M_{B_W}) \ | \ W \in O(j)\mright\} = \\
        & = -\frac{1}{\delta(\mathbf L, \mathbf S)}D_{\widehat S}p \mleft(\widehat S \ \mathrm{co}\mleft\{M_{B_W} \ | \ W \in O(j)\mright\}\mright).
    \end{align}
       Given the definition of $B_W$, we can check that 
    \begin{align}
        M_{B_W} = C \overline W C^{\top}. 
    \end{align}
   Therefore, as the matrix $C$ is fixed and does not depend on $W \in O(j)$, we obtain
   \begin{align}
       \mathrm{co}\mleft\{M_{B_W} \ | \ W \in O(j)\mright\} = C \ \mathrm{co}\mleft\{\overline W \ | \ W \in O(j)\mright\}\ C^{\top}.
   \end{align}
    To understand the convex hull of the matrices $\overline W$, it is enough to look at the convex hull of the matrices
    \begin{align}
        \mleft(\begin{array}{cc}
                \mathds 1_{k-j} & 0  \\
                0 & W
            \end{array}\mright),\quad W\in O(j).
    \end{align}
Pick $\lambda_1,\dots,\lambda_r \geq 0$ such that $\lambda_1+\dots+\lambda_r=1$ and $W_1,\dots, W_r \in O(j)$. Then 
    \begin{align}
        \sum_{i=1}^r \lambda_i\mleft(\begin{array}{cc}
                \mathds 1_{k-j} & 0  \\
                0 & W_i
            \end{array}\mright) = \sum_{i=1}^r \mleft(\begin{array}{cc}
                \lambda_i \mathds 1_{k-j} & 0  \\
                0 & \lambda_i W_i
            \end{array}\mright) = \mleft(\begin{array}{cc}
                \mathds 1_{k-j} & 0  \\
                0 & \sum_{i=1}^r \lambda_i W_i
            \end{array}\mright),
    \end{align}
    from which we obtain 
    \begin{align}
        \mathrm{co}\mleft\{\mleft(\begin{array}{cc}
                \mathds 1_{k-j} & 0  \\
                0 & W
            \end{array}\mright) \ \bigg| \ W \in O(j)\mright\} = \mleft\{\mleft(\begin{array}{cc}
                \mathds 1_{k-j} & 0  \\
                0 & B
            \end{array}\mright) \ \bigg| \ B \in \mathrm{co}(O(j))\mright\}.
    \end{align}
    It is known that the convex hull of the orthogonal group is the set of matrices with operator norm smaller or equal than $1$ (see \cite[Section 3.4]{JNRSR} and \cite[Theorem 2.1]{convex}), which is the same as the set of matrices with singular values bounded by $1$.
    \end{proof}
    As a corollary, we deduce the following geometric description of the subdifferential of $\dist_\mathbf{L}$.
    \begin{corollary}\label{subdiffcoro}
        For every $\mathbf L \in G(k,n)$ and $\mathbf S \in \Omega(j) \subseteq \mathrm{cut}(\mathbf L)$ with $j \in \{1,\dots,k\}$, the subdifferential $\partial_{\mathbf S}\delta_{\mathbf L}$ is a convex set linearly isomorphic to the convex hull of the orthogonal group $O(j)$, in particular it is of dimension $j^2$.
    \end{corollary}
    \begin{proof}
        We use the same notation as in \cref{subdiffpropo}. We define the matrix 
        \begin{align}
            Y_B:=U \Sigma \mleft(\begin{array}{cc}
                \mathds 1_{k-j} & 0 \\
                0 & B^{\top}
            \end{array}\mright)V^{\top} \in \R^{(n-k)\times k}
        \end{align}
        for any $B \in \mathrm{co}\{O(j)\}$. Then by \eqref{finalsubdiff} we obtain 
        \begin{align}
            \mathrm{dim}(\partial_{\mathbf S}\delta_{\mathbf L})=\mathrm{dim}\mleft(D_{\widehat S}p\mleft(\mleft\{\widehat S\mleft(\begin{array}{cc}
                0 & -Y_B^{\top} \\
                Y_B & 0
            \end{array}\mright) \ | \ B \in \mathrm{co}\{O(j)\}\mright\}\mright)\mright).
        \end{align}
        By \eqref{orthogonalkernel} we have that the set of matrices to which we are applying $D_{\widehat S}p$ lie in the orthogonal to its kernel, so that $D_{\widehat S}p$ preserves the dimension of such set. Since the matrix $\widehat S$ is orthogonal, we obtain 
        \begin{align}
            \mathrm{dim}(\partial_{\mathbf S}\delta_{\mathbf L}) = \mathrm{dim}\mleft(\mleft\{\mleft(\begin{array}{cc}
                0 & -Y_B^{\top} \\
                Y_B & 0
            \end{array}\mright) \ | \ B \in \mathrm{co}\{O(j)\}\mright\}\mright) = \mathrm{dim}\mleft(\mleft\{Y_B \ | \ B \in \mathrm{co}\{O(j)\}\mright\}\mright).
        \end{align}
        Since $ U$ and $V$ are orthogonal transformations, we only need to check how the matrix $\Sigma=\mathrm{diag}(\theta_1,\dots,\theta_{k-j},\frac{\pi}{2},\dots,\frac{\pi}{2})$ acts. It is easy to see that given the block structure of the product defining $Y_B$, the possible presence of $0$ entries in $\Sigma$, corresponding to a possible non-trivial intersection between $\mathbf S$ and $\mathbf L$, does not affect the part of the product involving the matrix $B$. As a consequence
        \begin{align}
            \mathrm{dim}(\partial_{\mathbf S}\delta_{\mathbf L}) = \mathrm{dim}\mleft(\mleft\{Y_B \ | \ B \in \mathrm{co}\{O(j)\}\mright\}\mright) = \mathrm{dim}(\mathrm{co}\{O(j)\})=j^2.
        \end{align} \end{proof}

\subsection{The restriction of the distance function to a submanifold}
 We move now to the study of the subdifferential of the restriction of $\dist_{\mathbf{L}}$ to a smooth submanifold $X\subset G(k,n)$. As we will see, in order to have a neat description of the subdifferential, we will need to impose some nondegeneracy conditions.
 
 \begin{definition}We say that a smooth submanifold $X \subseteq G(k,n)$ is transverse to $\mathrm{cut}(\mathbf L)$ if it is transverse to each stratum $\Omega(j)$ of the stratification \eqref{cutstrata}.
  \end{definition}
 We begin with the following easy observation.
 
 \begin{lemma}\label{generictrasnversality}
    Let $X \subseteq G(k,n)$ be a smooth submanifold. Then, for a generic $\mathbf L \in G(k,n)$, the manifold $X$ is transverse to $\mathrm{cut}(\mathbf L)$ .
\end{lemma}
\begin{proof}
    Let $g \in O(n)$ be such that $\mathbf L = g^{-1}\cdot \mathbf L_0$. Since the action of $O(n)$ on $G(k,n)$ is by diffeomorphisms which are also isometries, we have that $X$ is transverse to $\mathrm{cut}(\mathbf L)$ if and only if $g \cdot X$ is transverse to $\mathrm{cut}(\mathbf L_0)$, hence it is enough to check this condition for the generic $g \in O(n)$. Consider the map 
    \begin{align}
        F:O(n) \times X &\longrightarrow G(k,n) \\
        (g,\mathbf S) & \longmapsto g \cdot \mathbf S.
    \end{align}
    Since the projection map $O(n) \longrightarrow G(k,n)$ is a (Riemannian) submersion, the $O(n)$ part is already enough to give the surjectivity of the differential of $F$. It follows that $F$ is transverse to $\mathrm{cut}(\mathbf L_0)$. By the Parametric Transversality Theorem \cite[Chapter 3, Theorem 2.7]{hirsch}, the map $F(g,\cdot):X \longrightarrow G(k,n)$ is transverse to $\mathrm{cut}(\mathbf L_0)$ for almost every $g \in O(n)$, which is the same as $g \cdot X$ being transverse to $\mathrm{cut}(\mathbf L_0)$. 
    
\end{proof}

 Next, we use the genericity assumptions to prove the following lemma, which is an analogue of   \cite[Proposition 2.1.10]{generalizedmorse}, where the authors deal with the Euclidean distance function. 
 \begin{lemma}\label{subdiffprojectionlemma}
        Let $X$ be a smooth submanifold of $G(k,n)$. For a generic $\mathbf L \in G(k,n)$ and every $\mathbf S \in X$, we have 
        \begin{align}
            \partial_{\mathbf S}(\dist_{\mathbf L}|_X) = \mathrm{proj}_{T_{\mathbf S}X}(\partial_{\mathbf S}\dist_{\mathbf L}),
        \end{align}
        where $\mathrm{proj}_{T_{\mathbf S}X}$ is the orthogonal projection onto the tangent space to $X$ at $\mathbf S$. 
    \end{lemma}
    \begin{proof}

    Remark that since $\mathbf L$ is just a point, $X$ is trivially transverse to $\mathbf L$. By \cref{generictrasnversality},  for the generic $\mathbf L$, the manifold $X$ is also transverse to $\mathrm{cut}(\mathbf L)$. The definition of subdifferential \eqref{subdiff} in this case reads as 
        \begin{align}\label{subdiff1}
            \partial_{\mathbf S}(\dist_{\mathbf L}|_X) = \mathrm{co}\left\{\underset{\mathbf S_n \rightarrow \mathbf S}{\mathrm{lim}} \nabla \dist_{\mathbf L}|_X(\mathbf S_n) \ | \ (\mathbf S_n)\subseteq \omega(\delta_{\mathbf L}|_X) \ \text{and the limit exists}\right\}.
        \end{align}
       By \cref{propo:distsmooth}, we have $\omega(\dist_{\mathbf L})=G(k,n) \setminus (\mathbf L \cup \mathrm{cut}(\mathbf L))$. Since $\mathrm{cut}(\mathbf L)$ has codimension $1$, by transversality $X \cap (\mathbf L \cup \mathrm{cut}(\mathbf L))$ has codimension at least $1$ in $X$. Moreover $\omega(\dist_{\mathbf L})\cap X \subseteq \omega(\dist_{\mathbf L}|_X)$ since, if $\dist_{\mathbf L}$ is differentiable at a point $\mathbf E \in X$, so is its restriciton to $X$. 
       
       These two facts together imply that any point in $\omega(\dist_{\mathbf L}|_X)$ can be approached by a sequence of points in $\omega(\dist_{\mathbf L}) \cap X$. As a consequence we obtain 
        \begin{align}\label{subdiff2}
            \eqref{subdiff1} = \mathrm{co}\left\{\underset{\mathbf S_n \rightarrow \mathbf S}{\mathrm{lim}} \nabla \dist_{\mathbf L}|_X(\mathbf S_n) \ | \ (\mathbf S_n)\subseteq \omega(\dist_{\mathbf L})\cap X \ \text{and the limit exists}\right\}.
        \end{align}
    At any differentiability point $\mathbf E \in \omega(\dist_{\mathbf L})\cap X$ we have that 
    \begin{align}
            \nabla \dist_{\mathbf L}|_X (\mathbf E) = \mathrm{proj}_{T_{\mathbf E}X}(\nabla \dist_{\mathbf L}(\mathbf E)).
        \end{align}
  Since the projection map $(\mathbf E,v)\longmapsto \mathrm{proj}_{T_{\mathbf E}X}(v)$ is continuous, we obtain
    \begin{align}
            \notag \eqref{subdiff2}&=\mathrm{co}\left\{\underset{\mathbf S_n \rightarrow \mathbf S}{\mathrm{lim}} \mathrm{proj}_{T_{\mathbf S_n}X}(\nabla \dist_{\mathbf L}(\mathbf S_n)) \ | \ (\mathbf S_n)\subseteq \omega(\dist_{\mathbf L})\cap X \ \text{and the limit exists}\right\}  \\ \notag
            &= \mathrm{co}\left\{\underset{\mathbf S_n \rightarrow \mathbf S}{\mathrm{lim}} \mathrm{proj}_{T_{\mathbf S}X}(\nabla \dist_{\mathbf L}(\mathbf S_n)) \ | \ (\mathbf S_n)\subseteq \omega(\dist_{\mathbf L})\cap X \ \text{and the limit exists}\right\}  \\
            &\label{subdiff3}= \mathrm{proj}_{T_{\mathbf S}X}\mleft(\mathrm{co}\left\{\underset{\mathbf S_n \rightarrow \mathbf S}{\mathrm{lim}}\nabla \delta_{\mathbf L}(\mathbf S_n) \ | \ (\mathbf S_n)\subseteq \omega(\dist_{\mathbf L})\cap X \ \text{and the limit exists}\right\}\mright),
        \end{align}
        by linearity of the projection. 
        
        Now let $U_{\mathbf S}\subseteq G(k,n)$ be a small neighbourhood of $\mathbf S$ and consider the disjoint union of connected components 
        \begin{align}
            U_{\mathbf S} \setminus (\mathbf L \cup \mathrm{cut}(\mathbf L)) = \bigsqcup_{i=1}^r A_i.
        \end{align}
        The function $\dist_{\mathbf L}$ is smooth on each $A_i$ and we can define
        \begin{align}
            V_i := \big\{ \underset{\mathbf S_n \rightarrow \mathbf S}{\mathrm{lim}}\nabla \delta_{\mathbf L}(\mathbf S_n) \ | \ (\mathbf S_n)\subseteq A_i \big\}
        \end{align}
        for $i=1,\dots,r$. For $U_{\mathbf S}$ small enough we have that 
        \begin{align}
            \partial_{\mathbf S}\dist_{\mathbf L} = \mathrm{co}\{V_1, \dots, V_r\}.
        \end{align}
        By transversality, for every $i=1,\dots,r$ the intersection $X \cap A_i$ is non empty and there exists a sequence $(\mathbf S_n)_n\subseteq X \cap A_i$ that converges to $\mathbf S$. It follows that 
        \begin{align}
            &\mathrm{co}\left\{\underset{\mathbf S_n \rightarrow \mathbf S}{\mathrm{lim}}\nabla \dist_{\mathbf L}(\mathbf S_n) \ | \ (\mathbf S_n)\subseteq \omega(\dist_{\mathbf L})\cap X \ \text{and the limit exists}\right\}  \\ =&\,\mathrm{co}\left\{\underset{\mathbf S_n \rightarrow \mathbf S}{\mathrm{lim}}\nabla \dist_{\mathbf L}(\mathbf S_n) \ | \ (\mathbf S_n)\subseteq \bigsqcup_{i=1}^r A_i\cap X \ \text{and the limit exists}\right\}  \\ =&\,\mathrm{co}\left\{\mathrm{proj}_{T_{\mathbf S}X}(V_1 \cup \dots \cup V_r)\right\} = 
            \mathrm{proj}_{T_{\mathbf S}X}(\mathrm{co}\{V_1,\dots,V_r\}) 
            \\ =&\,\mathrm{proj}_{T_{\mathbf S}X}(\partial_{\mathbf S}\delta_{\mathbf L}).
        \end{align}
        As a consequence, from \eqref{subdiff3} we obtain
        \begin{align}
            \partial_{\mathbf S}(\dist_{\mathbf L}|_X) = \mathrm{proj}_{T_{\mathbf S}X}(\partial_{\mathbf S}\dist_{\mathbf L}).
        \end{align}
        
    \end{proof}
Recalling our definition of critical points (\cref{def:subdifferential}), from \cref{subdiffprojectionlemma} it immediately follows a characterization of the critical points of the restricted distance function. 
\begin{corollary}\label{criticallemma}
        Let $X$ be a smooth submanifold of $G(k,n)$. For a generic $\mathbf L \in G(k,n)$, a point $\mathbf S \in X$ is critical for $\dist_{\mathbf L}|_X$ if and only if $N_{\mathbf S}X \cap \partial_{\mathbf S}\dist_{\mathbf L} \neq \emptyset$.
    \end{corollary}

\subsection{The nearest point problem: ignoring the cut locus}\label{ignoringsection}Let us now go back to the nearest point problem from a generic $\mathbf{L}$ to a submanifold $X\subset G(k,n)$. The structure of the subdifferential of $\dist_{\mathbf{L}}$ complicates the study of critical points of $\dist_{\mathbf{L}}|_X$, even for a generic $\mathbf{L}$. However, in this section we show that, if we are truly interested in the minimization problem, this becomes again a standard  minimization problem with smooth constraint (for instance, one can use the Lagrange Multiplier's Rule).
    
 More precisely, the next  result shows that whenever the transversality condition holds, the minimum of $\dist_{\mathbf L}|_X$ cannot be on $\mathrm{cut}(\mathbf L)\cap X$ (which is where nondifferentiability points of $\dist_{\mathbf L}|_X$ can occur). Remark that this need not be true without the transversality hypothesis; this is the case for example of smooth submanifolds contained in the cut locus. 
\begin{theorem}\label{nominimacut}
    Let $\mathbf L \in G(k,n)$ and $X \subseteq G(k,n)$ be a smooth submanifold. Assume that $X$ is transverse to $\mathrm{cut}(\mathbf L)$. Then no point in $X \cap \mathrm{cut}(\mathbf L)$ can be a (local) minimum for the function $\dist_{\mathbf L}|_X$.  
\end{theorem}
\begin{proof}
    Let $\mathbf E \in  X \cap \mathrm{cut}(\mathbf L)$ and assume in particular that $\mathbf E \in \Omega(j)$; remark that $X$ need not intersect all the strata of $\mathrm{cut}(\mathbf L)$. Denote by $\{p_1,\dots,p_k\}$ and $\{q_1,\dots,q_k\}$ some orthonormal basis of principal vectors for $\mathbf L$ and $\mathbf E$, with principal angles $0\leq\theta_1\leq\dots\leq\theta_{k-j}< \frac{\pi}{2}$ and the last $j$ ones being $\frac{\pi}{2}$. Let $\gamma:[0,1]\longrightarrow G(k,n)$ be the lenght minimizing geodesic such that $\gamma(0)=\mathbf E$ and $\gamma(1)=\mathbf L$. For $i=1,\dots,k-j$ let $n_i \in \langle q_i,p_i\rangle$ be the unit vector orthogonal to $q_i$ such that $p_i = q_i \cos(\theta_i) + n_i \sin(\theta_i)$. Then one can check that, under the isomorphism from \cref{def:hom}, $\dot{\gamma}(0)$ corresponds to the linear map $\varphi : \mathbf E \longrightarrow \mathbf E^{\perp}$ which acts like 
    \begin{align}
        \varphi(q_1)&= \theta_1 \,n_1 \\
        & \ \ \vdots \\
        \varphi(q_{k-j})&= \theta_{k-j}\,n_{k-j} \\
        \varphi(q_{k-j+1})&= \frac{\pi}{2}\,p_{k-j+1} \\ 
        & \ \ \vdots \\
        \varphi(q_k)&= \frac{\pi}{2}\,p_k.
    \end{align}
Given \eqref{cutstrataexplicit} and \eqref{tangentschubert}, we have that 
\begin{align}\label{tangentcut}
    T_{\mathbf E}\Omega(j)=\{\varphi:\mathbf E \longrightarrow \mathbf E^{\perp} \ | \ \varphi(\mathbf E \cap \mathbf L^{\perp})\subseteq \mathbf E^{\perp}\cap \mathbf L^{\perp}\}.
\end{align}
By construction $\mathbf E \cap \mathbf L^{\perp} = \langle q_{k-j+1},\dots,q_k\rangle$ and this is mapped to $\langle p_{k-j+1},\dots,p_k\rangle \subseteq \mathbf L$. Therefore $\dot{\gamma}(0) \notin T_{\mathbf E}\mathrm{cut}(\mathbf L)$. Since by hypothesis $X$ is transverse to $\Omega(j)$ there exists a vector $v \in T_{\mathbf E}X$ such that the scalar product $\langle v, \dot{\gamma}(0)\rangle >0$. Let $\sigma:(-\varepsilon,\varepsilon)\longrightarrow X$ be an adapted curve to $v$, i.e. $\sigma(0)=\mathbf E$ and $\dot{\sigma}(0)=v$. Then for $t>0$ small enough we have 
\begin{align}
    \dist_{\mathbf L}(\mathbf E) > \dist_{\mathbf L}(\sigma(t)),
\end{align}
which implies that $\mathbf E$ cannot be a local minimum for $\dist_{\mathbf L}$ on $X$. 
    
\end{proof}

As a consequence of \cref{nominimacut} and \cref{generictrasnversality}, we obtain the following corollary.
\begin{corollary}\label{ignorethecut}
    Let $X$ be a smooth submanifold of $G(k,n)$. Then for a generic $\mathbf L \in G(k,n)$ the minimizer of the distance function $\dist_{\mathbf L}|_X$ cannot be on $X \cap \mathrm{cut}(\mathbf L)$.
\end{corollary}
\begin{remark}\cref{ignorethecut} is one of the conceptual key results of this work. Indeed, unlike in the algebraic Euclidean setting, it is not true that the number of critical points of $\dist_{\mathbf L}|_X$ for the generic $\mathbf L$ is finite for every smooth submanifold $X$. In \cref{Schubertsection} we will exhibit an example of a Schubert variety admitting infinitely many global maxima (and therefore critical points) for the generic $\mathbf L \in G(k,n)$, such that each of these maxima is on the smooth stratum of the variety. Nevertheless, all these maxima will be on the cut locus of the point $\mathbf L$. We will see in \cref{discretenesssection} that this is not by chance: as soon as we exclude points lying on the cut locus, we generically obtain only a finite amount of critical points. Thanks to \cref{ignorethecut}, we know that only these finitely many points will matter for nearest point problem, as the minimum will be found among them. 
\end{remark}

\subsection{Discreteness of critical points away from the cut locus}\label{discretenesssection}
Under the hypothesis of \cref{criticallemma} we have that a point $\mathbf S \in X$ is critical for $\dist_{\mathbf L}|_X$ if and only if the normal space to $X$ at $\mathbf S$ meets the subdifferential $\partial_{\mathbf S}\dist_{\mathbf L}$. If $\mathbf S \notin \mathrm{cut}(\mathbf L)$ and $\mathbf S \neq \mathbf L$, then as we already remarked the function $\dist_{\mathbf L}$ is smooth at $\mathbf S$ and its subdifferential coincides with the gradient, which is the tangent vector to the unique length minimizing geodesic connecting $\mathbf L$ to $\mathbf S$ at $\mathbf S$. This gives the following corollary.
\begin{corollary}\label{smoothcritical}
    For a smooth submanifold $X \subseteq G(k,n)$ and a generic $\mathbf L \in G(k,n)$, if $\mathbf S \notin \mathrm{cut}(\mathbf L)$ and $\mathbf S \neq \mathbf L$, then $\mathbf S$ is critical for $\dist_{\mathbf L}|_X$ if and only if the unique length minimizing geodesic from $\mathbf L$ to $\mathbf S$ is normal to $X$ at $\mathbf S$.
\end{corollary}
If instead $\mathbf S \in \mathrm{cut}(\mathbf L)$, being critical for $\dist_{\mathbf L}|_X$ does not imply the existence of a geodesic from $\mathbf L$ to $\mathbf S$ being normal to $X$. Indeed, in that case the subdifferential is just the convex hull of the tangent vectors to length minimizing geodesics between $\mathbf L$ and $\mathbf S$ (see \cref{subdifferentialsection}). This means that it can intersect the normal space to $X$ at $\mathbf S$ even if no geodesic is normal. For example, this phenomenon will appear in the specific example that we will describe in \cref{Schubertsection}, when dealing with global maxima of the restricted distance function: no geodesic from the external $\mathbf L$ to a global maximum will be orthogonal to the submanifold, but all of them are in the cut locus. \\

We want now to prove that the set of critical points not on the cut locus is generically finite (\cref{finitecritical} below). Let $X$ be a smooth submanifold of $G(k,n)$. In view of \cref{smoothcritical}, consider the normal bundle to $X$, denoted as $NX$. Consider the \textit{normal exponential map}, which is the restriction of the exponential map to $NX$:
\begin{align}
    \exp|_{NX}: NX \longrightarrow G(k,n).
\end{align}
    It is a smooth map, given the smoothness of the exponential map itself. Then by Sard's Lemma the set of critical values of $\exp|_{NX}$ has measure zero in $G(k,n)$: in other words, a generic $\mathbf L \in G(k,n)$ is a regular value of $\exp|_{NX}$. It is then immediate to prove the following.
\begin{lemma}\label{normalexplemma}
    Let $\mathbf L \in G(k,n)$ be generic. Then the preimage $(\exp|_{NX})^{-1}(\mathbf L)$ is discrete.
\end{lemma}
\begin{proof}
    By Sard's Lemma $\mathbf L$ is a regular value of $\exp|_{NX}$. Observe that the bundle $NX$ and the manifold $G(k,n)$ have the same dimension. It follows that $(\exp|_{NX})^{-1}(\mathbf L)$ is a zero--dimensional submanifold. Discreteness follows.
\end{proof}
\begin{theorem}\label{finitecritical}
    Let $X$ be a smooth compact submanifold of $G(k,n)$ and $\mathbf L \in G(k,n)$ be generic. Then the set of critical points of $\dist_{\mathbf L}|_X$ on $X \setminus \mathrm{cut}(\mathbf L)$ is finite. The same holds true if $X$ is subanalytic (possibly non compact).
\end{theorem}
\begin{proof}
    By \cref{smoothcritical} a critical point $\mathbf S \in X \setminus (\mathrm{cut}(\mathbf L)\cup \mathbf L)$ is characterized by the geodesic connecting it to $\mathbf L$ being normal to $X$. Denoting such geodesic as $\gamma_{\mathbf S}:[0,1]\longrightarrow G(k,n)$ with $\gamma_{\mathbf S}(0)=\mathbf S$ and $\gamma_{\mathbf S}(1)=\mathbf L$, this means that 
    \begin{align}
        \big(\mathbf S, \dot{\gamma}_{\mathbf S}(0)\big) \in \exp|_{NX}^{-1}(\mathbf L)
    \end{align}
    for every such critical point $\mathbf S$. Remark that since $\exp_{\mathbf E}(R_{\frac{\pi}{2}})=G(k,n)$ for every $\mathbf E$, we can restrict the normal exponential map to normal vectors in $R_{\frac{\pi}{2}}$, which is compact in the tangent space at each point. Therefore if also $X$ is compact, the discreteness in \cref{normalexplemma} implies that only finitely many $\mathbf S$ away from $\mathrm{cut}(\mathbf L)$ can be critical for $\dist_{\mathbf L}|_X$. 
    
    If now $X$ is not compact but subanalytic, finiteness follow from o--minimality. \end{proof}
%\cref{finitecritical} tells us that, if we exclude points on the cut locus, generically we will have only finitely many critical points. Nevertheless, it tells nothing about how many such critical points we get. In the next section we will show that the Grassmann distance function $\dist_{\mathbf L}$ is definable in the o--minimal structure of globally subanalytic sets. As a consquence we will see that the number of such critical points is uniformly bounded for a generic $\mathbf L$. 
%\newpage
\subsection{The Grassmann Distance Complexity}
We will now use the results from previous sections, in particular \cref{finitecritical} and the o--minimality of principal angles and the Grassmann distance function (\cref{ominimalitylemma}), to prove that the number of critical points of $\dist_{\mathbf L}|_X$ which are not on $\mathrm{cut}(\mathbf L)$ is uniformly bounded for the generic $\mathbf L \in G(k,n)$. 
\begin{theorem}\label{uniformboundedness}
    Let $X$ be a smooth subanalytic submanifold of $G(k,n)$. There exists a constant $C_X\in \N$ such that for a generic $\mathbf L \in G(k,n)$ we have 
    \begin{align}\label{gdd}
        \#\{\mathbf S \in X\setminus \mathrm{cut}(\mathbf L) \ | \ \text{s.t.} \ \mathbf S \ \text{is critical for} \ \delta_{\mathbf L}|_X\} \leq C_X.
    \end{align}
\end{theorem}
\begin{proof}
     Define the following set
     \begin{align}
         A =\{(\mathbf S,\mathbf L) \in X \times G(k,n) \ | \ \mathbf S \notin \mathrm{cut}(\mathbf L), \ \mathbf S \in \mathrm{Crit}(\delta_{\mathbf L}|_X)\},
     \end{align}
    where $\mathrm{Crit}(f)$ denotes the critical points of a locally Lipschitz function $f$. By \cref{ominimalitylemma} and the definability of $X$, it follows that also the set $A$ is definable. \\
    Consider the projection map on the second factor $\pi_2:A \longrightarrow G(k,n)$. By \cite[Chapter 9, Theorem 1.2]{VanDenDries}, there exist a partition of $G(k,n)$ as a finite disjoint union of definable sets
    \begin{align}
        G(k,n)=\bigsqcup_{i=1}^l M_i,
    \end{align}
    definable sets $F_1,\dots,F_l$ and definable maps $\varphi_i : F_i \times M_i \longrightarrow \pi_2^{-1}(M_i)$ such that the following diagram commutes
    \begin{center}
\begin{tikzcd}
F_i\times M_i \arrow[rd, "p_2"'] \arrow[rr, "\varphi_i"] &     & \pi_2^{-1}(M_i) \arrow[ld, "\pi_2"] \\
                                                         & M_i
                                                         &    
\end{tikzcd}.
\end{center}
    In particular for every $\mathbf L \in M_i$ we have $\pi_2^{-1}(\mathbf L) \cong F_i$, where the isomorphism is subanalytic. The fiber $\pi_2^{-1}(\mathbf L)$ describes the points in $X \setminus \mathrm{cut}(\mathbf L)$ that are critical for $\dist_{\mathbf L}|_X$. Therefore by \cref{finitecritical} for the generic $\mathbf L$ such fiber consists of finitely many points. Hence, if we define 
    \begin{align}
        \widetilde M := \bigcup_{\mathrm{dim}(F_i)=0}M_i,
    \end{align}
    we have that $G(k,n) \setminus \widetilde M$ has codimension at least $1$ in $G(k,n)$. For every $i$ such that $\mathrm{dim}(F_i)=0$ we have that $\lvert F_i \rvert = c_i < +\infty$. Define $C_X := \max\{c_i \ | \ \mathrm{dim}(F_i)=0\}$. Then for every $\mathbf L \in \widetilde M$ it holds that 
    \begin{align}
        \#\{\mathbf S \in X\setminus \mathrm{cut}(\mathbf L) \ | \ \text{s.t.} \ \mathbf S \ \text{is critical for} \ \dist_{\mathbf L}|_X\} \leq C_X
    \end{align}
     and the proof is concluded.
\end{proof}
The meaning of the constant $C_X$ in \cref{uniformboundedness} is the following. When we are looking for solutions to our problem of minimizing $\dist_{\mathbf L}|_X$ for a generic $\mathbf L$, we already observed that one way to proceed is to first look for critical points and then find the minimum among them. \cref{nominimacut} tells us that we do not need to care about critical points that lie on the cut locus of $\mathbf L$. Therefore the constant $C_X$ serves as a bound on the number of points that we will have to check in order to find the minimum. In other words, it bounds the \emph{complexity} of the minimization problem for the submanifold $X$. In the same spirit that led to the definition of the Euclidean Distance Degree for algebraic subvarieties of $\R^n$, we introduce the following notion. 
\begin{definition}[Grassmann Distance Complexity]\label{GDDdef}
    Let $X$ be a smooth subanalytic submanifold of $G(k,n)$. We define the \textit{Grassmann Distance Complexity} of $X$ to be the maximum of the left hand side in \eqref{gdd}. We will denote this as $\mathrm{GDC}(X)$.
\end{definition}
\begin{remark}
    Suppose that $X$ is a stratified submanifold of $G(k,n)$. As in the Euclidean setting, we will be interested in critical points lying in the smooth stratum of $X$. We can therefore extend \cref{GDDdef} to the stratified or singular case by defining $\mathrm{GDC}(X)$ as the Grassmann Distance Degree of the smooth stratum. 
\end{remark}
%In \cref{pfaffiansection} we will show that when $X$ is an algebraic subvariety of $G(k,n)$ we are able to compute an explicit upper bound for $\mathrm{GDD}(X)$, based on Khovanskii's theory of Pfaffian functions. 
\subsection{Non--semialgebraicity}\label{sec:nonsemi} We discuss in this section an example showing that the nearest point to an algebraic variety in $G(k,n)$ is genuinely non--semialgebraic, i.e. that the critical points of  the Grassmann distance to a point constrained to an algebraic set are, in general, solutions of a system of non--algebraic equations.

For this example, we will work in the \emph{oriented Grassmannian} $\widetilde{G}(2,4)$, i.e. the Grassmannian of oriented $2$--planes in $\R^4.$ This is a double cover of $G(2,4)$, the quotient map simply being the map that forgets the orientation. The oriented Grassmannian embeds, via the spherical Pl\"ucker embedding, in the unit sphere in $\R^6$ as the quadric
\[\widetilde{G}(2,4)=\left\{(x, y)\in \R^3\oplus \R^3\,\bigg|\, \|x\|^2=\|y\|^2=1\right\}\simeq S^2\times S^2.\] 
The above discussion extends verbatim to the oriented Grassmannian. Up to a multiple, the Grassmann metric is the product metric of the two sphere factors (\cite{Kozlov}), therefore the (oriented) Grassmann distance is the function
\[\dist:\widetilde{G}(2, 4)\times \widetilde{G}(2,4)\to \R,\]
given for $(x, y), (z, w)\in \widetilde{G}(2, 4)\simeq S^2\times S^2$ by
\[\dist\left((x, y), (z, w)\right)=\sqrt{\left(\arccos\langle x, z\rangle\right)^2+\left(\arccos\langle y, w\rangle\right)^2}.\]
Let now $e_0=(1, 0, 0)\in \R^3$. Then, the Riemmanian sphere $S_R$ of radius $0<R<\frac{\pi}{2}$ centered at $\mathbf{L}:=(e_0, e_0)$ is given by
\[S_R=\left\{(x, y)\in \R^3\oplus \R^3\,\bigg|\, \|x\|^2=\|y\|^2=1, \,
(\arccos x_1)^2+(\arccos y_1)^2=R^2)\right\}.
\]
(Notice that this \emph{is not} an algebraic set in $\R^6$.) The cut locus of $\mathbf{L}$ equals
\[\label{cutg24}\mathrm{cut}(\mathbf{L})=\left(\{-e_0\}\times S^2\right)\cup \left(S^2\times \{-e_0\}\right).\]

Consider now the semialgebraic family $\{X_\beta\}_{\beta>0}\subseteq \widetilde{G}(2,4)$ of algebraic sets defined by
\[X_\beta:=\left\{(x, y)\in \R^3\oplus \R^3\,\bigg|\, \|x\|^2=\|y\|^2=1,\, x_1-\beta y_1=0\right\}.\]
Each $X_\beta$ is a hypersurface in $\widetilde{G}(2, 4)$, given by a linear slice. 

The critical points of $\dist_\mathbf{L}|_{X_\beta}$ on $X_\beta\setminus \mathrm{cut}(\mathbf{L})$ are those points where the gradient of $\dist_\mathbf{L}$ is a multiple of the gradient of defining equation of $X_\beta$ in $S^2\times S^2$. Denoting by $\alpha:[-1, 1]\to \R$ the function
\[\alpha(w):=\frac{\arccos w}{\sqrt{1-w^2}},\]
we see that these are the points in $X_\beta$ such that the following matrix has rank less than $4$:
\[M:=\left(\begin{array}{cccc}x_1 & 0 & \alpha(x_1) & 1 \\x_2 & 0 & 0 & 0 \\x_3 & 0 & 0 & 0 \\0 & y_1 & \alpha(y_1) & -\beta \\0 & y_2 & 0 & 0 \\0 & y_3 & 0 & 0\end{array}\right).\]
The determinant of the $4\times 4$ matrix $M^\top M$ vanishes exactly when $M$ has rank at most $3$. This determinant equals $(x_2^2+x_3^2)(y_2^2+y_3^2)(\alpha(y_1)+\beta \alpha (x_1))^2$. Therefore, the critical points of $\dist_\mathbf{L}|_{X_\beta}$ not on the cut locus \eqref{cutg24} are the points $(x, y)$ in $\R^3\oplus \R^3$ solving the following system of equations:
\[\|x\|^2=\|y\|^2=1,\quad x_1=\beta y_1, \quad \frac{\arccos (y_1)}{\sqrt{1-y_1^2}}+\beta\frac{\arccos (\beta y_1)}{\sqrt{1-\beta^2y_1^2}}=0.\]
From this we see that the family of systems that we obtain, depending on the semialgebraic parameter $\beta>0$, is not a semialgebraic system. 

%\noindent In particular, from the proof it follows that adding the restricted sine and cosine functions to semialgebraic functions would be enough to define an o-minimal structure where principal angles can be defined. As a consequence of the definability of principal angles in above o-minimal structure, we obtain that also the Grassmann distance function from a point is definable in the same family.
%\begin{corollary}\label{distancedefinable}
%    For any $\mathbf L \in G(k,n)$, the Grassmann distance function from $\mathbf L$ $\dist_{\mathbf L}$ is definable in the globally subanalytic family.
%\end{corollary}
%\begin{proof}
%    It follows from the fact that sums, products and compositions of definable functions in an o-minimal structure are still definable. 
%\end{proof}
%

\subsection{Pfaffian bounds} \label{sec:pfaffian}
We discuss in this section how to obtain upper bounds on the Grassmann Distance Complexity of a smooth algebraic subvariety, depending explicitly on the format of its defining equations. This will be achieved using known bounds on the topology of the solution set of a systems of Pfaffian equations and inequalities. Our main reference will be \cite{Gabrielovpfaffian}. 

A \emph{Pfaffian chain} of order $r\geq 0$ and degree $\alpha\geq 1$ is an ordered sequence of analytic functions $\{f_1,\dots,f_r\}$ on an open domain $\mathcal{U} \subseteq \R^N$ such that 
\begin{align}
    \frac{\partial f_i}{\partial x_j}(x) = P_{ij}(x,f_1(x),\dots,f_i(x))
\end{align}
for every $i=1,\dots,r$ and $j=1,\dots,N$, where $P_{ij}(x_1,\dots,x_N,y_1,\dots,y_i)$ are polynomials with $\mathrm{deg}(P_{ij})\leq \alpha$. In other terms, a Pfaffian chain is an ordered sequence of analytic functions satisfying a triangular systems of polynomial PDEs in a given domain. A \emph{Pfaffian function} of order $r \geq 0$ and degrees $(\alpha, \beta)$ on a domain $\mathcal{U} \subseteq \R^N$ is a function $f$ such that 
\begin{align}
    f(x)=P(x,f_1(x),\dots,f_r(x)),
\end{align}
where $\{f_1,\dots,f_r\}$ is a Pfaffian chain of order $r$ and degree $\alpha$ on the domain $\mathcal{U}$ and $P(x_1,\dots,x_N,y_1,\dots,y_r)$ is a polynomial in $N+r$ variables with $\mathrm{deg}(P)\leq \beta$. The triple $(r,\alpha,\beta)$ is called the \emph{format} of the Pfaffian function $f$.
\begin{example}\label{cospfaffian}
    Let $\mathcal{U} =\{x \in \R \ | \ x \neq \pi + 2k\pi \ \forall k \in \Z\}$ and define $f_1(x):= \tan\big(\frac{x}{2}\big)$, $f_2(x):=\cos^2\big(\frac{x}{2}\big)$, $f_3(x):=\cos(x)$ and $f_4(x):=\sin(x)$. We have the following identities:
    \begin{align}
        \notag f'_1(x)&=\frac{1}{2}\big(1+f_1^2(x)\big) \\
        \notag f'_2(x)&=-f_1(x)f_2(x) \\
        \notag f'_3(x)&=-2f_1(x)f_2(x) \\
        \notag f'_4(x)&=f_3(x).
    \end{align}
It follows that $(f_1,f_2,f_3,f_4)$ is a Pfaffian chain of order $4$ and degree $2$ on $\mathcal{U}$.
\end{example}

We will need the following lemma, whose proof is immediate.
\begin{lemma}[\cite{Gabrielovpfaffian}, Lemma 2.5]\label{pfaffianderivative}
    A partial derivative of a Pfaffian function of format $(r,\alpha,\beta)$ is a Pfaffian function of format $(r,\alpha,\alpha+\beta-1)$.
\end{lemma}
A \emph{basic semi--Pfaffian} set $S \subset \R^N$ in a domain $\mathcal{U}\subseteq \R^N$ is a set defined by sign conditions on a family of Pfaffian functions on a common Pfaffian chain in $\mathcal{U}$. A \emph{semi--Pfaffian set} is a finite union of basic semi-Pfaffian sets. (Note that the condition of having a common Pfaffian chain is not restrictive, as the concatenation of Pfaffian chains still gives a Pfaffian chain.)  The topology of semi--Pfaffian sets can be explicitly bounded in terms of the number and format of the Pfaffian functions involved. Several different bounds can be found in the literature (see also \cite{Zell}); we will use the following. 
\begin{theorem}[\cite{Gabrielovpfaffian}, Theorem 3.4]\label{GVbound}
    Let $S \subseteq \mathcal{U} \subseteq \R^N$ be a semi-Pfaffian set defined by a conjunction of equations and strict inequalities involving $s$ different Pfaffian functions having a common Pfaffian chain. Let $(r,\alpha,\beta)$ be the format of the Pfaffian functions. Then the sum of the Betti numbers of $S$, denoted $b(S)$, is bounded by 
    \begin{align}
        b(S) \leq s^N 2^{\frac{r(r-1)}{2}}\mathcal{O}\big(N\beta + \mathrm{min}\{N,r\}\alpha\big)^{N+r}.
    \end{align}
\end{theorem}
From the description of the exponential map in terms of tangent matrices given in \cref{geodesicprop}, the relation between principal angles and singular values given in \cref{boundedprop} and \cref{cospfaffian}, it follows that for an algebraic submanifold $X \hookrightarrow G(k,n)$ the critical points of $\delta_{\mathbf L}|_X$ not in $\mathrm{cut}(\mathbf L)$ admit a Pfaffian description. We will explain how to obtain such a description through the worked out example of a hypersurface. It will be evident how to generalize it to deal with any algebraic submanifold for which the degrees of the defining equations are known. 

Recall the Pl\"ucker embedding of Grassmannians $$\varphi: G(k,n) \hookrightarrow \mathds{P}(\Lambda^k\R^n)\cong \P^{N}$$ sending a plane $\mathbf E$ with orthonormal basis $\{v_1,\dots,v_k\}$ to the class of the wedge product $\varphi(\mathbf E)=[v_1 \wedge \dots \wedge v_k]$. If $\{e_1,\dots,e_n\}$ is the standard basis of $\R^n$, for any $w_1\wedge\dots\wedge w_k \in \R^{N+1}$ we write 
\begin{align}
    w_1\wedge\dots\wedge w_k =\sum_{1\leq i_1<\dots<i_k\leq n}c_{i_1,\dots,i_k}e_{i_1}\wedge\dots\wedge e_{i_k}
\end{align}
and the collection $(c_{i_1,\dots, i_k})$ define a system of coordinates on $\R^{N+1}$. The Pl\"ucker coordinates of $\mathbf E \in G(k,n)$ are the coordinates $[c_{i_1,\dots, i_k}]_{1\leq i_1<\dots<i_k\leq n}$ of $\varphi(\mathbf E) \in \P^N$. The coordinate charts of $\P^N$ are given by 
\begin{align}
    \mathcal{U}_{j_1\dots j_k}:=\{[c_{i_1,\dots,i_k}] \in \P^N \ | \ c_{j_1,\dots,j_k} \neq 0\} &\overset{\psi_{j_1,\dots, j_k}}{\xrightarrow{\hspace{1.5cm}}} \R^N \\
    % &\xmapsto{\hspace{1.5cm}} \frac{1}{c_{j_1,\dots,j_k}}\,\big(c_{i_1,\dots,i_k}\big)_{\{i_1,\dots,i_k\}\neq \{j_1,\dots,j_k\}}.
\end{align}
where $\psi_{j_1,\dots,j_k}$ divides the coordinates by $c_{j_1,\dots,j_k}$ and forgets the $(j_1,\dots,j_k)$--th one.

Without loss of generality, we assume for the rest of this section that $\mathbf L=[ e_1,\dots,e_k]$. From \cref{geodesicprop}, for any $v_A \in T_{\mathbf L}G(k,n) \cong \R^{(n-k)\times k}$ we have 
\begin{align}\label{expidentity}
    \exp_{\mathbf L}(v_A) = \left[\begin{array}{ccc}
                    \cos(\mu_1)v_1 & \dots & \cos(\mu_k)v_k \\ 
                    \sin(\mu_1)u_1 & \dots & \sin(\mu_k)u_k
                \end{array}\right],
\end{align}
where $A=U\Sigma V^{\top}$ is a SVD of $A$, $U=(u_1 \dots u_{n-k}) \in O(n-k)$, $V=(v_1 \dots v_k) \in O(k)$ and $\Sigma = \mathrm{diag}(\mu_1 \dots \mu_k) \in \R^{(n-k)\times k}$. Remark that, in \eqref{expidentity},  only the first $k$ columns of $U$ actually matter, so we can use the reduced SVD and assume that $U$ is in the Stiefel manifold of orthonormal $k$--frames in $\R^{n-k}$, denoted $\mathrm{St}(k,n-k)$. 

From \cref{boundedprop} we also have that $G(k,n)\setminus \mathrm{cut}(\mathbf L)=\exp_{\mathbf L}(\mathrm{int}({R}_{\frac{\pi}{2}}))$, where we recall that $\mathrm{int}({R}_{\frac{\pi}{2}})$ is the set of matrices in $T_{\mathbf L}G(k,n)$ with singular values strictly less than $\frac{\pi}{2}$ and that $\exp_{\mathbf L}$ is a diffeomorophism onto its image when restricted to this set. 
\begin{lemma}\label{chartlemma}
    The image in the Pl\"ucker embedding of $G(k,n) \setminus \mathrm{cut}(\mathbf L)$ is contained in $\mathcal{U}_{1,\dots,k}$. 
\end{lemma}

\begin{proof}
    Every point of $G(k,n) \setminus \mathrm{cut}(\mathbf L)$ can be written as in \eqref{expidentity} for some $A \in \mathrm{int}({R}_{\frac{\pi}{2}})$; in particular $0\leq\mu_i < \frac{\pi}{2}$ for every $i=1,\dots,k$. The Pl\"ucker coordinate with multi--index $(1,\dots,k)$ of \eqref{expidentity} is given by $\cos(\mu_1)\dots\cos(\mu_k)\mathrm{det}(V)$. This is always non--zero, since $V \in O(k)$ and $\cos(\mu_i)\neq 0$ for every $i=1,\dots,k$. This proves the statement.
\end{proof}

Let now $X \hookrightarrow G(k,n)$ be a smooth algebraic hypersurface of degree $d$, in the sense of \cite[Ch. 3, Sect. A]{greenbook}. By \cite[Prop. 2.1, Chapter 3]{greenbook} there exists a homogeneous polynomial $p$ of degree $d$ in the Pl\"ucker coordinates such that $X=\{p=0\}$. The following results provides an upper bound to the $\textrm{GDC}$ of $X$.

\begin{theorem}\label{boundthm}
    For every $0\leq k\leq n$ there exist constants $c_1(k,n), c_2(k,n)>0$ such that for every smooth, degree $d$ hypersurface $X \hookrightarrow G(k,n)$, the following bound holds:
    \begin{align}
        \mathrm{GDC}(X) \leq c_1(k,n) d^{c_2(k,n)}.
    \end{align}
\end{theorem}
The constants $c_1(k,n)$ and $c_2(k,n)$ are given, respectively, by \eqref{c_1} and \eqref{c_2} from the following proof.
\begin{proof}The proof consists of several steps. First we perform a preliminary reduction (a ``change of variables'') so that we reduce to study the set of critical points of a new function $f$ defined on a new set $Y$, in such a way that these critical points admits a semi--Pfaffian description; then we bound the topology of this set of critical points, call it $S$, using \cref{GVbound}; finally we show that each critical point of $\dist_{\mathbf L}|_X$ lying in $X \setminus \mathrm{cut}(\mathbf L)$ corresponds to at most one connected component of $S$ and, therefore, the bound we have obtained for the Betti numbers of $S$ also serves as a bound for $\mathrm{GDC}(X)$.

\emph{First step: a preliminary reduction}

    We need to describe, given a generic $\mathbf L \notin X$, the critical points of $\dist_{\mathbf L}|_X$ lying in $X \setminus \mathrm{cut}(\mathbf L)$. As above, we can assume without loss of generality, that $\mathbf L=[ e_1,\dots,e_k]$. Let $X=\{p=0\}$ for a homogeneous polynomial $p$ of degree $d$ in the Pl\"ucker coordinates.  By \cref{chartlemma}, in order to find these critical points, we can work in the affine chart $(\mathcal{U}_{1,\dots,k}, \psi_{1,\dots,k})$, where we dehomogeneize $p$, obtaining the corresponding polynomial function $\hat p$; thus $X\setminus \mathrm{cut}(\mathbf L)$ is defined by $\hat p=0$. 
    
    Consider now the set $\mathcal{G}:=\mathrm{St}(k,n-k) \times O(k) \times [0,\frac{\pi}{2})^k \subseteq \R^{(n-k)\times k}\times \R^{k\times k}\times \R^{k} \cong \R^{k(n+1)}$ and define the following map, where the variables are now the SVD components of the tangent vectors:
    \begin{align}\label{svdexp}
        \widetilde{\exp}_{\mathbf L}: \qquad \quad \mathcal{G} \quad&\ \xrightarrow{\hspace{1.2cm}} \qquad \qquad \quad G(k,n) \\
        (U,V,\mu) &\ \xmapsto{\hspace{1.2cm}} \left[\begin{array}{ccc}
                    \cos(\mu_1)v_1 & \dots & \cos(\mu_k)v_k \\ 
                    \sin(\mu_1)u_1 & \dots & \sin(\mu_k)u_k
                \end{array}\right],
    \end{align}
    with $U=(u_1,\dots,u_k) \in \mathrm{St}(k,n-k)$, $V=(v_1,\dots,v_k)\in O(k)$ and $\mu=(\mu_1,\dots,\mu_k)$. We express the polynomial function $\hat p$ in these new coordinates (i.e. the entries of $U,V$ and $\mu$) as $\tilde p := \hat p \circ\varphi\circ  \widetilde{\exp}_{\mathbf L}$. 
    
    The relation between the maps involved is depicted in the following commutative diagram (here $
    \varphi$ is the Pl\"ucker embedding):
    $$\begin{tikzcd}
{G(k,n)\setminus \mathrm{cut}(\mathbf{L})} \arrow[rr, "\varphi"]                                                                   &  & {\mathcal{U}_{1, \ldots, k}} \arrow[r, "\widehat{p}"] & \mathbb{R} \\
\mathrm{int}(R_{\frac{\pi}{2}}) \arrow[u, "\mathrm{exp}_\mathbf{L}"']                                                              &  &                                                       &            \\
\mathcal{G} \arrow[u] \arrow[uu, "\widetilde{\mathrm{exp}}_{\mathbf{L}}", bend left=60] \arrow[rrruu, "\widetilde{p}", bend right] &  &                                                       &           
\end{tikzcd}$$

It follows that $Y:=\widetilde{\mathrm{exp}}_{\mathbf{L}}^{-1}(X \setminus \mathrm{cut}(\mathbf L))$ is defined in 
$$\mathcal{U}:= \R^{(n-k)\times k}\times \R^{k\times k}\times \left(0,\frac{\pi}{2}\right)^k $$
by the system of equations
    \begin{equation}\label{eqvincolo}
        \begin{cases}
        \tilde p(U,V,\mu)=0 \\
        U^{\top}U =\mathds 1_k \\
        V^{\top}V = \mathds 1_k
        \end{cases}.
    \end{equation}
In these coordinates the squared distance function is given by 
\begin{align}\label{sqdist}
    f(U,V,\mu):=\dist_{\mathbf L}^2(\widetilde{\exp}_{\mathbf L}(U,V,\mu)) = \mu_1^2+\dots+\mu_k^2.
\end{align}
The main idea for the rest of the proof is now to find critical points of $f|_{Y}$, and then relate them to the critical points of $\dist_{\mathbf L}|_X$ lying in $X \setminus \mathrm{cut}(\mathbf L)$.

\emph{Second step: bound on the Betti numbers of the critical set of $f|_{Y}$.}

The critical points of $f|_{Y}$ can be found through Lagrange multipliers rule, i.e. imposing that the gradient of \eqref{sqdist} is a linear combination of the gradients of the equations in \eqref{eqvincolo}. This gradient is given by 
\begin{align}
    \nabla f (U,V,\mu)=2(0,0,\mu).
\end{align}
To shorten notations, introduce the function $F : \R^{(n-k)\times k}\times \R^{k\times k} \longrightarrow \R^{k(k+1)}$ given by 
\begin{align}
    F(U,V):=((\langle u_i,u_j\rangle-\delta_{ij})_{1\leq i\leq j\leq k},(\langle v_i,v_j\rangle-\delta_{ij})_{1\leq i\leq j\leq k});
\end{align}
then $F(U,V)=0$ are the defining equations for $\mathrm{St}(k,n-k)\times O(k)$, so that \eqref{eqvincolo} can be equivalently written as 
\begin{equation}\label{eqvincolo2}
    \begin{cases}
        \tilde p(U,V,\mu)=0 \\
        F(U,V)=0
    \end{cases}.
\end{equation}
Since the defining equations for Stiefel manifolds and orthogonal groups are regular, at every point $(U,V) \in F^{-1}(0)$ the differential of $F$, denoted by $D_{(U,V)}F$, has maximal rank, i.e. $k(k+1)$. Recall that the rows of $D_{(U,V)}F$ are given by the gradients of the components of $F$. Denote also by $\partial_{(U,V)}\tilde p$ the vector of partial derivatives of $\tilde p$ with respect to the variables in $U$ and $V$, and by $\partial_{\mu}\tilde p$ the vector of partial derivatives of $\tilde p$ with respect to $\mu_1,\dots,\mu_k$. To visualize the criticality conditions, look at the block matrix containing as rows the gradients of the equations in \eqref{eqvincolo2} and $\nabla f$
\begin{equation}\label{criticalmatrix}
    M:=\mleft(\begin{array}{cc}
        \partial_{(U,V)}\tilde p & \partial_{\mu}\tilde p \\
        D_{(U,V)}F & 0 \\
        0 & \mu
    \end{array}
    \mright);
\end{equation}
the criticality conditions are that \eqref{eqvincolo2} is satisfied and the last row of \eqref{criticalmatrix} is linearly dependent from the other rows. 

If $\partial_{\mu}\tilde p=0$, the criticality condition about gradients would force $\mu=0$. This solution would then correspond to the unique point $\mathbf L$, which doesn't lie on $X$. Therefore on the solution set we have $\mu \neq 0$ and $\partial_{\mu}\tilde p\neq 0$ and the upper part of $M$ consisting of the gradient of $\tilde p$ and $F$ has maximal rank, equal to $k(k+1)+1$. It follows that the criticality condition involving $M$ is equivalent to $\partial_{\mu}\tilde p \neq 0$ and all maximal minors of $M$ vanish.  

To enforce the condition $\partial_{\mu}\tilde p\neq 0$, we subdivide the system into $2k$ systems of the following forms
\begin{equation}\label{criticalsystem}
\begin{cases}
    \tilde p(U,V,\mu)=0 \\
    F(U,V)=0 \\
    \partial_{\mu_i}\tilde p >0 \\
    \mathrm{rank}(M)=k(k+1)+1
\end{cases}, \qquad
\begin{cases}
    \tilde p(U,V,\mu)=0 \\
    F(U,V)=0 \\
    \partial_{\mu_i}\tilde p <0 \\
    \mathrm{rank}(M)=k(k+1)+1
\end{cases},
\end{equation}
for $i=1,\dots,k$. It is an easy linear algebra exercise to show that 
\begin{equation}\label{equivlemma}
    \mathrm{rank}(M)=k(k+1)+1 \iff \mathrm{rank}\mleft(\begin{array}{c}
         \partial_{\mu}\tilde p \\
         \mu
    \end{array}\mright)=1 \ \text{and} \ \mathrm{rank}\mleft(\begin{array}{c}
         \partial_{(U,V)}\tilde p \\
         D_{(U,V)}F
    \end{array}\mright)=k(k+1),
\end{equation}
so that in the systems \eqref{criticalsystem} the last condition can be replaced by the right side of \eqref{equivlemma}. 

At this point, the crucial observation is that all the functions involved in \eqref{criticalsystem} and \eqref{equivlemma} are Pfaffian functions in the domain $\mathcal{U}\subseteq \R^{k(n+1)}$ in a common Pfaffian chain. To see this, define the functions
\begin{align}
    f_{1,i}\,(U,V,\mu)&:=\tan\mleft(\frac{\mu_i}{2}\mright),\\
    f_{2,i}\,(U,V,\mu)&:=\cos^2\mleft(\frac{\mu_i}{2}\mright),\\
    f_{3,i}\,(U,V,\mu)&:=\cos(\mu_i), \\
    f_{4,i}\,(U,V,\mu)&:=\sin(\mu_i).
\end{align}
By \cref{cospfaffian} the ordered sequence $\mathfrak{C}=(f_{1,1},\dots,f_{4,1},f_{1,2},\dots,f_{4,k})$ is a Pfaffian chain in $\mathcal{U}\subseteq \R^{k(n+1)}$ of order $4k$ and degree $2$. We now examine all the functions in \eqref{criticalsystem} and \eqref{equivlemma}, computing their pfaffian formats with respect to the chain $\mathfrak{C}$ and the $k(n+1)$ variables $(U,V,\mu)$:

\begin{itemize}
    \item [-] $\tilde p$ is a polynomial of degree at most $d$ in the affine Pl\"ucker coordinates obtained taking the $k\times k$ minors of the matrix in \eqref{svdexp} and dividing by the product $\cos(\mu_1)\dots\cos(\mu_k)$. Each of the entries of the matrix is a degree $2$ polynomial in the functions of the Pfaffian chain above. Therefore $\tilde p$ is a Pfaffian function of order $4k$ and degrees bounded by $(2,2kd)$ on $\mathcal{U}$;
    \item[-] $F(U,V)=0$ corresponds to $k(k+1)$ equations, each being polynomial of degree $2$; 
    \item[-] by \cref{pfaffianderivative} the function $\partial_{\mu_i}\tilde p$ is Pfaffian of format $(4k,2,2kd+1)$ for every $i=1,\dots,k$; 
    \item[-] the condition of $\partial_{\mu}\tilde p$ and $\mu$ forming a $\mathrm{rank}$--$1$ matrix can be expressed through $k-1$ equations of Pfaffian format $(4k,2,2kd+2)$. Indeed depending on which of the systems we are looking at, we know that there exists $i\in \{1,\dots,k\}$ such that the column $(\partial_{\mu_i}\tilde p, \mu_i)^{\top}$ is non-zero. It is then sufficient to impose that the other $k-1$ columns are proportional to this one;
    \item[-] the condition that the matrix of partial derivatives of $\tilde p$ and $F$ with respect to $U,V$ has rank $k(k+1)$ corresponds to all the maximal minors being zero. As the matrix has size $(k(k+1)+1)\times nk$, this gives $\binom{nk}{k(k+1)+1}$ equations of Pfaffian format $(4k,2,2dk+k(k+1)+1)$.
\end{itemize}
From these computations follows that \eqref{criticalsystem} are systems involving $\binom{nk}{k(k+1)+1}+(k+1)^2$ Pfaffian functions of formats bounded by $(4k,2,2dk+k(k+1)+1)$ in $k(n+1)$ variables. 

Denote now by $S_i^+\subseteq \mathcal{U}$ the solution set of the first system in \eqref{criticalsystem} and by $S_i^-$ that of the second one. Define the semi-Pfaffian set 
\begin{align}
    S:= \bigcup_{\substack{i=1,\dots,k \\ *\in\{+,-\}}} S_i^*.
\end{align}
A direct application of \cref{GVbound} provides the following upper bound for the sum of the Betti numbers of $S_i^*$ for every $i=1,\dots,k$ and $*\in \{+,-\}$:
\begin{align}\label{eq:bettii}
    b(S_i^*) \leq \tilde c_1(k,n) d^{c_2(k,n)} ,
\end{align}
where the $\tilde c_1(k,n)$ and $c_2(k,n)$ depend on $k$ and $n$ only and are given by 
\begin{align}
    \tilde c_1(k,n) &:= c\, 2^{8k^2-2k}\mleft(\binom{nk}{k(k+1)+1}+(k+1)^2\mright)^{k(n+1)}(2k^2(n+1))^{k(n+5)}, \\
    \label{c_2} c_2(k,n) &:= k(n+5),
\end{align}
with $c$ being a constant. 

\emph{Third step: relate the critical points of $f|_{Y}$ to the critical points of $\dist_{\mathbf L}|_X$ lying in $X \setminus \mathrm{cut}(\mathbf L)$.
}

Every $(U,V,\mu) \in S$ gives a critical point of $\dist_{\mathbf L}|_X$. Recall indeed that the variables $(U,V,\mu)$ were introduced to play the role of the SVD of tangent matrices in $T_{\mathbf L}G(k,n)$, to which apply the exponential at $\mathbf L$ to obtain the whole $G(k,n)\setminus \mathrm{cut}(\mathbf L)$. Therefore the map that reconstructs critical points of $\dist_{\mathbf L}|_X$ is given by  
\begin{align}
    (U,V,\mu) \in S \xmapsto{\hspace{1cm}} \exp_{\mathbf L}(U \,\mathrm{diag}(\mu_1,\dots,\mu_k)\,V^{\top})\in X \setminus \mathrm{cut}(\mathbf L).
\end{align}
Since this map is continuous, the discreteness of critical points outside of $\mathrm{cut}(\mathbf L)$ given by \cref{normalexplemma} ensures that every connected component of $S$ corresponds to exactly one critical point. Denoting by $b_0(S)$ the number of connected components of $S$, we have 
\begin{align}\label{abstractbound}
    \mathrm{GDC}(X) \leq b_0(S) \leq \sum_{\substack{i=1,\dots,k \\ *\in\{+,-\}}} b_0(S_i^*) \leq \sum_{\substack{i=1,\dots,k \\ *\in\{+,-\}}} b(S_i^*).
\end{align}
The statement of the theorem follows from \eqref{abstractbound} and \eqref{eq:bettii} with 
\begin{align}\label{c_1} 
    c_1(k,n):= 2k \, \tilde c_1(k,n).
\end{align}

\end{proof}
\begin{remark}
    The bound provided by \cref{boundthm} is enormous: already in the case $k=2$, $n=4$ the constant $c_1(2,4)$ has order $10^{50}$! This is a typical feature of the bounds obtained through the theory of Pfaffian functions: the importance of \cref{boundthm} is not the value of the bound, but its existence and computability. Notice however that different choices of variables or ways of writing the critical equations might lead to improved bounds, even though obtaining reasonable bounds this way is hopeless. 
\end{remark}
\begin{remark}
    An important feature of the strategy outlined in the proof of \cref{boundthm} is that it can be easily extended to any algebraic $X$ for which explicit equations are given; the drawback is that the critical condition might get more complicated if $X$ is defined by more than one equation. Notice also that we could extend the procedure to \textit{Pfaffian varieties}: in that case we would need to keep into account also the format of the defining Pfaffian equations. 
\end{remark}

\section{Nearest point problem to simple Schubert varieties}\label{Schubertsection}
    In this section we will apply the previous results to some specific subvarieties of $G(k,n)$. These varieties will be instances of simple Schubert varieties, chosen in such a way that we can exploit the results of \cref{grassmanngeometrysection} to describe them in terms of matrix varieties. 
    
    \subsection{Schubert varieties}Schubert varieties  are defined by requiring non--trivial intersection conditions between the $k$-planes of $G(k,n)$ and the elements of a fixed complete flag in $\R^n$. To every Schubert variety one can associate a Young diagram and viceversa. From the Young diagram associated to the variety one can easily read how many independent conditions are needed in the definition of the variety and its codimension. We refer the reader to \cite{Fultontableux} for more details.
    %Moreover, the smooth strata of Schubert varieties, the so called \emph{Schubert cells}, are a basis for the cohomology of Grassmannians \cite{MilnorStasheff}. 
    
    Here we consider a specific class of Schubert varieties, defined as follows. 
    \begin{definition}[Simple Schubert variety]Fix a $k$--plane $\mathbf W \in G(k,n)$. Define the variety $\Omega_s=\Omega_s(\mathbf{W})$  by
\begin{align}
   \label{omegadef} \Omega_s(\mathbf{W}) := \{\mathbf E \in G(k,n) \ | \ \mathrm{dim}(\mathbf E \cap \mathbf W)\geq s\},
\end{align}
where $s\in \{1,\dots,k-1\}$ (if $s=k$, $\Omega_k =\{\mathbf W\}$). 
\end{definition}
The variety $\Omega_s$ can be stratified as
\[
\Omega_s = \Omega_s^{\mathrm{sm}} \sqcup \left(\bigsqcup_{j=1}^{k-s} \Omega_s^{\mathrm{sing,j}}\right),\]
where 
\[\label{smoothomegas}\Omega_s^{\mathrm{sm}} = \{\mathbf E \in G(k,n) \ | \ \mathrm{dim}(\mathbf E \cap \mathbf W)=s\}, \]
denotes the smooth stratum of $\Omega_s$ and 
\[\label{singomegas}\Omega_s^{\mathrm{sing,j}}= \{\mathbf E \in G(k,n) \ | \ \mathrm{dim}(\mathbf E \cap \mathbf W)=s+j\},\]
    for $j=1,\dots,k-s$, denote the  singular strata. Remark that the deepest singular stratum of $\Omega_s$ consists of the point $\mathbf W$ alone, i.e. $\Omega_s^{\mathrm{sing},k-s}=\{\mathbf W\}$. 
    
    We introduce the notation $\theta_i(\mathbf E):=\theta_i(\mathbf E,\mathbf W)$ for the principal angles between any $\mathbf E \in G(k,n)$ and the fixed $\mathbf W$. Recalling the properties of principal angles, $\Omega_s$ and its strata can be equivalently described in terms of principal angles as
    \begin{align}
\Omega_s=&\{\mathbf E \in G(k,n) \ | \ \theta_1(\mathbf E)=\dots=\theta_s(\mathbf E)=0\}, \\ 
\Omega_s^{\mathrm{sm}}=&\{\mathbf E \in G(k,n) \ | \ \theta_1(\mathbf E)=\dots=\theta_s(\mathbf E)=0, \ \theta_{s+1}(\mathbf E)>0\}, \\
\Omega_s^{\mathrm{sing,j}}=&\{\mathbf E \in G(k,n) \ | \ \theta_1(\mathbf E)=\dots=\theta_{s+j}(\mathbf E)=0, \ \theta_{s+j+1}(\mathbf E)>0\},
\end{align}
for $j=1,\dots,k-s$. \\

As a consequence of \cref{boundedprop}, when applying the exponential map of $G(k,n)$ at the point $\mathbf W$ to a tangent vector $v_A \in T_{\mathbf W}G(k,n)$ for some $A \in \R^{(n-k)\times k}$, the rank of $A$ will correspond to the number of non--zero principal angles between $\mathbf W$ and $\exp_{\mathbf W}(v_A)$, since the rank of a matrix is equal to the number of its non-zero singular values. (From now on we will identify $T_{\mathbf W}G(k,n) \cong \R^{(n-k)\times k}$ via the isomorphism \cref{def:vA}, the choice of a representative $\widehat W \in O(n)$ is implicit.) 

Recalling the notation introduced in \eqref{boundednotation} and the fact that the image of $R_{\frac{\pi}{2}}$ through $\exp_{\mathbf W}$ is the whole $G(k,n)$, we have the following corollary of \cref{boundedprop}.
\begin{corollary}\label{boundedrankomega}
    Let $\exp_{\mathbf W}:\R^{(n-k)\times k} \longrightarrow G(k,n)$ be the exponential map of $G(k,n)$ at $\mathbf W$. Then 
    \begin{align}\label{omegaexp}
\Omega_s = &\exp_{\mathbf W}\left(\mathcal M_{\leq k-s} \cap R_{\frac{\pi}{2}}\right) ,\\
\Omega_s^{\mathrm{sm}} =& \exp_{\mathbf W}\left(\mathcal M_{k-s}\cap R_{\frac{\pi}{2}}\right) ,\\
\Omega_s^{\mathrm{sing},j} =& \exp_{\mathbf W}\left(\mathcal M_{k-s-j} \cap R_{\frac{\pi}{2}}\right),
\end{align}
for $j=1,\dots,k-s$. 
\end{corollary}
In the next subsection we will see how \cref{boundedrankomega} allows to use the techniques coming from matrix rank--approximation, and in particular \cref{EY}, to find some critical points for the Grassmann distance function from a generic $\mathbf L \notin \Omega_s$ restricted to $\Omega_s$. 

\subsection{Eckart--Young theorem and critical points}
    In this section we explore how to exploit the description given by \cref{boundedrankomega} to find critical points for the function $\dist_{\mathbf L}|_{\Omega_s}$ for a generic $\mathbf L \in G(k,n)$. \\

    The first thing to observe is that for a generic $\mathbf L \in G(k,n)$, the intersection $\mathbf L \cap \mathbf W$ is trivial, so that no principal angle between $\mathbf L$ and $\mathbf W$ is zero. The same is true for $\mathbf L \cap \mathbf W^{\perp}$, meaning that no principal angle will be $\frac{\pi}{2}$. Moreover the angles will all be distinct. In particular we have that the generic $\mathbf L$ is not in the cut locus of $\mathbf W$ and $\exp_{\mathbf W}^{-1}(\mathbf L) \cap R_{\frac{\pi}{2}}$ consists of only one matrix with full rank. We denote this matrix by $A_{\mathbf L}$:
    \[\label{eq:AL}A_{\mathbf{L}}:= \exp_{\mathbf W}^{-1}(\mathbf L) \cap R_{\frac{\pi}{2}}.\]

    By \eqref{omegaexp} we have that $\exp_{\mathbf W}^{-1}(\Omega_s)\cap R_{\frac{\pi}{2}}$ corresponds to $\mathcal{M}_{\leq k-s} \cap R_{\frac{\pi}{2}}$. As we want to find critical points for the restriction of the distance function from $\mathbf L$ to $\Omega_s$, we might wonder whether this problem can be translated to the matrix case in $T_{\mathbf W}G(k,n)$ via the inverse of the exponential map $\exp_{\mathbf W}^{-1}$, where $\mathbf L$ corresponds to the full-rank matrix $A_{\mathbf L}$ and $\Omega_s$ corresponds to matrices with rank bounded by $k-s$. 
    
    The (crazy at this point!) idea would be the following:
\begin{enumerate}[label=(\roman*)]
\item given $\mathbf L \notin \Omega_s$ generic, find the unique full--rank matrix $A_{\mathbf L} \in R_{\frac{\pi}{2}}$ such that $\exp_{\mathbf W}(A_{\mathbf L})=\mathbf L$; 
    \item use Eckart--Young (\cref{EY}) to approximate $A_\mathbf{L}$ with rank $k-s$ matrices;
    \item map back on $G(k,n)$ the rank--approximated matrices with $\exp_{\mathbf W}$ (these will be smooth points in $\Omega_s$).
\end{enumerate}
A priori, there is no reason to believe that the points found via above strategy are critical points for $\dist_{\mathbf L}|_{\Omega_s}$, since one expects a distortion effect in the metric (and the critical point equation) given by applying the exponential map. Suprisingly, this procedure actually produces critical points, as shown in the following result. 
\begin{theorem}\label{standardEYomegas}
    Let $\mathbf L \notin \Omega_s$ be generic and $A_{\mathbf L} \in R_{\frac{\pi}{2}}$ be such that $\exp_{\mathbf W}(A_{\mathbf L})=\mathbf L$. For any $I \subseteq \{1,\dots,k\}$ with $\lvert I\rvert=k-s$ let $A_{\mathbf L,I} \in \mathcal{M}_{k-s}$ be the rank--approximation obtained from $A_{\mathbf L}$ via \cref{EY}. Define $\mathbf L_I = \exp_{\mathbf W}(A_{\mathbf L,I}) \in \Omega_s^{\mathrm{sm}}$. Then $\mathbf L_I$ is a critical point for $\dist_{\mathbf L}|_{\Omega_s}$. 
\end{theorem}
%Before proving the theorem, let us try to understand what is happening from the geometric point of view when we approximate $A_{\mathbf L}$ with some $A_{\mathbf L,I}$. When we apply $\exp_{\mathbf W}$ to $A_{\mathbf L}$, we rotate the principal vectors of $\mathbf W$ onto those of $\mathbf L$ with the speed prescribed by the corresponding principal angles $0<\theta_1(\mathbf L) < \dots < \theta_k(\mathbf L) < \frac{\pi}{2}$. When we pass from $A_{\mathbf L}$ to $A_{\mathbf L,I}$, the singular values, and therefore by \cref{boundedprop} the principal angles, corresponding to indices not in $I$ are set to $0$. This means that applying $\exp_{\mathbf W}$ to $A_{\mathbf L,I}$, while the principal vectors corresponding to indices in $I$ still get rotated to the corresponding ones in $\mathbf L$, the ones corresponding to indices not in $I$ will now stay still. Remark that, starting from $\mathbf W$, rotating the principal vectors of $\mathbf W$ towards those of $\mathbf L$ descreases the distance from $\mathbf L$. Therefore, setting some singular values in $A_{\mathbf L}$ to zero \lq\lq preserves\rq\rq\ the starting distance between $\mathbf W$ and $\mathbf L$ corresponding to those principal angles, while keeping a singular value unchanged allows to perform the rotation and annihilate the contribution of that angle to the distance. \\
\begin{proof}
The proof will be divided into several steps. First, given $\mathbf W$ and $\mathbf L$, we will compute the matrix $A_{\mathbf L}$. We will then switch to $A_{\mathbf L,I}$ and compute $\mathbf L_{I}$. As a last step we will compute the length minimizing geodesic connecting $\mathbf L_I$ to $\mathbf L$ and prove that this geodesic  starts orthogonal to $\Omega_s$. 

To compute $A_{\mathbf L}$ we proceed as in \cref{grassmanngeometrysection}. Assume we fixed representatives $\widehat W = (w_1 \dots w_k \ b_1 \dots b_{n-k}) \in p^{-1}(\mathbf W)$ and $\widehat L \in p^{-1}(\mathbf L)$. Let $0 < \theta_1 < \dots < \theta_k < \frac{\pi}{2}$ be the principal angles bewteen $\mathbf W$ and $\mathbf L$, i.e. $\theta_i:=\theta_i(\mathbf W, \mathbf L)$ for $i=1,\dots,k$. Notice that by the genericity of $\mathbf L$ we can assume they are all distinct and different from $0$ and $\frac{\pi}{2}$. Let $\Sigma=\mathrm{diag}(\theta_1,\dots,\theta_k)\in \R^{(n-k)\times k}$ and $\{p_1,\dots,p_k\}$ and $\{q_1,\dots,q_k\}$ be basis of principal vectors for $\mathbf W$ and $\mathbf L$ respectively. As in \cref{grassmanngeometrysection} define $\pi_i:=\langle p_i,q_i\rangle$ and $n_i \in \pi_i$ as the unit vector such that 
\begin{align}
    q_i = p_i \cos(\theta_i) + n_i \sin(\theta_i).
\end{align}
We can complete $\{n_1,\dots,n_k\}$ to an orthonormal basis of $\mathbf W^{\perp}$, denoted $\{n_1,\dots,n_{n-k}\}$. 

Define $V = (v_1 \dots v_k) \in O(n)$ the change of basis matrix from $\{w_1, \dots w_k\}$ to $\{p_1,\dots,p_k\}$, i.e. $(w_1 \dots w_k) V = (p_1 \dots p_k)$. Similarly define $U=(u_1 \dots u_{n-k})$ as the change of basis matrix from $\{b_1,\dots,b_{n-k}\}$ to $\{n_1,\dots,n_{n-k}\}$. We know by \cref{grassmanngeometrysection} that
\begin{align}
    A_{\mathbf L} = U\Sigma V^{\top}, \qquad \qquad\exp_{\mathbf W}(A_{\mathbf L})=\mathbf L.
\end{align}

For any $I=\{i_1,\dots,i_{k-s}\}\subseteq \{1,\dots,k\}$ let $\widetilde \Sigma = \mathrm{diag}(\widetilde \theta_1,\dots,\widetilde \theta_k)$ such that $\widetilde \theta_i=\theta_i$ if $i \in I$ and $\widetilde \theta_i=0$ otherwise, with ordering $i_1 < \dots < i_{k-s}$. Define $A_{\mathbf L,I}:= U\widetilde \Sigma V^{\top}$ and $\mathbf L_I:=\exp_{\mathbf W}(A_{\mathbf L,I})$. By construction we have $\mathrm{rank}(A_{\mathbf L,I})=k-s$ and therefore $\mathbf L_I \in \Omega_s^{\mathrm{sm}}$. 

By \cref{geodesicprop} we have
\begin{align}
    \mathbf L_I = \left[\widehat W\mleft(\begin{array}{ccc}
            \cos(\widetilde \theta_1)v_1 & \dots & \cos(\widetilde \theta_k)v_k \\
            \sin(\widetilde \theta_1)u_1 & \dots & \sin(\widetilde \theta_k)u_k
        \end{array}\mright)\right].
\end{align}
Explicitly, an orthonormal basis for $\mathbf L_I$ is given by 
\begin{align}
    &i \in I: \qquad \widehat W \begin{pmatrix}
        \cos(\theta_i)v_i \\
        \sin(\theta_i)u_i
    \end{pmatrix} = \cos(\theta_i)p_i + \sin(\theta_i)n_i = q_i, \\
    &i \notin I: \qquad \widehat W \begin{pmatrix}
        v_i \\ 0
    \end{pmatrix} = p_i.
\end{align}

We now want to describe the unique length minimizing geodesic connecting $\mathbf L_I$ to $\mathbf L$. We denote it as $\gamma_I:[0,1]\longrightarrow G(k,n)$ with $\gamma_I(0)=\mathbf L_I$ and $\gamma_I(1)=\mathbf L$. We use the procedure already described in \cref{grassmanngeometrysection}. Denote by $\{1,\dots,k\} \setminus I = \{j_1,\dots,j_s\}$ with ordering $j_1 <\dots < j_s$. As orthonormal bases we use $\beta_1:=\{q_{i_1},\dots,q_{i_{k-s}},p_{j_1},\dots,p_{j_s}\}$ for $\mathbf L_I$ and $\{q_{i_1},\dots,q_{i_{k-s}},q_{j_1},\dots,q_{j_s}\}$ for $\mathbf L$. Then we have 
    \begin{align}
        \mleft(\begin{array}{c}
            q_{i_1}^{\top} \\
            \vdots \\
            q_{i_{k-s}}^{\top} \\ \\
            q_{j_1}^{\top} \\
            \vdots \\
            q_{j_s}^{\top}
        \end{array}\mright)\bigg(q_{i_1} \dots q_{i_{k-s}} \ p_{j_1} \dots p_{j_s}
        \bigg) = \mathrm{diag}\big(1, \dots, 1, \cos(\theta_{j_1}),\dots,\cos(\theta_{j_s})\big),
    \end{align}
which is already in SVD form. This means that the basis above are already bases of principal vectors for $\mathbf L_I$ and $\mathbf L$, with the first $k-s$ principal angles being $0$ and the remaining $s$ being $0<\theta_{j_1}<\dots<\theta_{j_s}<\frac{\pi}{2}$. 

We complete the orthonormal set $\{n_{j_1},\dots,n_{j_s}\}\subseteq \mathbf L_I^{\perp}$ to an orthonormal basis of $\mathbf L_I^{\perp}$ denoted by $\beta_2:=\{f_1,\dots,f_{k-s},n_{j_1},\dots,n_{j_s},f_{k+1},\dots,f_{n-k}\}$. To write $\gamma_I$ we use the representative $\widehat L_I \in p^{-1}(\mathbf L_I)$ given by the basis $\beta_1$ and $\beta_2$, to fix the isomorphism $T_{\mathbf L_I}G(k,n) \cong \R^{(n-k)\times k}$. Then, the matrix $M_I:= \mathrm{diag}(0,\dots,0,\theta_{j_1},\dots,\theta_{j_s}) \in \R^{(n-k)\times k}$ is such that $\exp_{\mathbf L_I}(v_{M_I})=\mathbf L$. Thus, we are only left to check that $v_{M_I}$ is in the normal space to $\Omega_s$ at $\mathbf L_I$.

Using the identification of tangent vectors $v_A$ as linear maps $\varphi_A$ from $\mathbf L_I$ to $\mathbf L_I^{\perp}$, from \eqref{tangentschubert} we obtain 
\[
    T_{\mathbf L_I}\Omega_s = \{\varphi_A:\mathbf L_I \longmapsto \mathbf L_I^{\perp} \ | \ \varphi_A(\mathbf L_I \cap \mathbf W)\subseteq \mathbf L_I^{\perp}\cap\mathbf W\}, \]
   \begin{align}N_{\mathbf L_I}\Omega_s = \{\varphi_A:\mathbf L_I \longrightarrow \mathbf L_I^{\perp} \ |& \ \varphi_A(\mathbf L_I \cap \mathbf W) \subseteq (\mathbf L_I^{\perp}\cap \mathbf W)^{\perp}\cap \mathbf L_I^{\perp}\quad \text{and}  \\& \varphi_A((\mathbf L_I \cap \mathbf W)^{\perp}\cap \mathbf L_I)=\{0\}\}.
\end{align}
By construction, a basis for $\mathbf L_I \cap \mathbf W$ is given by $\{p_{j_1},\dots,p_{j_s}\}$. The matrix $M_I$ represents a linear map $\varphi_{M_I}:\mathbf L_I \longrightarrow \mathbf L_I^{\perp}$ with respect to the basis $\beta_1$ and $\beta_2$. Therefore we have 
\begin{align}
    &\varphi_{M_I}(q_{i_1})=0 \\
    &\qquad\vdots\\
    &\varphi_{M_I}(q_{i_{k-s}})=0 \\
    &\varphi_{M_I}(p_{j_1})=\theta_{j_1}n_{j_1} \\
    &\qquad \vdots \\
    &\varphi_{M_I}(p_{j_s})=\theta_{j_s}n_{j_s}.
\end{align}
Since $n_{j_1},\dots,n_{j_s}$ are orthogonal to both $\mathbf L_I$ and $\mathbf W$, it follows that $M_I \in N_{\mathbf L_I}\Omega_s$ and $\mathbf L_I$ is therefore critical for $\dist_{\mathbf L}|_{\Omega_s}$ by \cref{propo:distsmooth} and \cref{criticallemma}.
\end{proof}

\begin{remark}\label{maximaremark}
    Eckart--Young theorem does not contemplate the possibility of \lq\lq increasing\rq\rq\ a singular value, because the variety $\mathcal{M}_{\leq k-s}$ is unbounded. In particular, maxima cannot be found by increasing singular values, because we could always increase them more and move further away. But in our setting, we actually restrict to the compact set $\mathcal{M}_{\leq k-s} \cap R_{\frac{\pi}{2}}$ and it makes sense to ask what happens if instead of modifying a singular value of $A_{\mathbf L}$ by sending it to $0$, thus not performing any rotation, we modify it with the aim of producing a point of maximum. In order to achieve this, instead of rotating from the principal vectors for $\mathbf W$ to those for $\mathbf L$ or staying still, we have to rotate in the opposite direction until we create an angle of $\frac{\pi}{2}$, which is the maximum achievable. In other words, if the original angle between $\mathbf W$ and $\mathbf L$ is $\theta_i$, we have to rotate the $i$--th principal vector $p_i$ of $\mathbf W$ in the opposite direction with respect to that of $\mathbf L$ $q_i$ of an angle of $\frac{\pi}{2}-\theta_i$. We believe that this new kind of modification of $A_{\mathbf L}$ again produces critical points for $\dist_{\mathbf L}|_{\Omega_s}$, with the difference that now they can also be at a distance from $\mathbf L$ greater than that of $\mathbf W$ (which was not the case in \cref{standardEYomegas}). The two approaches, the one aimed at minimizing the distance in some directions, and the one that aims to maximize it, can be mixed to produce new type of critical points. All these newly created points will be in the cut locus of $\mathbf L$, therefore we cannot use the characterization of critical points through normal geodesics. In order to show that these are critical points one should explicitly compute the subdifferential and check that it intersects the normal space. Since we already know that no such points are minimizers, we will not perform the computation here. Nevertheless, we will use this idea when computing global maximizers in \cref{maximatheorem}.
\end{remark}
\begin{remark}
    \cref{standardEYomegas} tells us that the points we obtain are critical points, but these \emph{are not} all the critical points outside of the cut locus, in general. Nevertheless, it provides a lower bound to the GDC of the algebraic subvarieties $\Omega_s$. In particular, \cref{standardEYomegas} gives a critical point for every choice of a subset of $k-s$ indices among $\{1,\dots,k\}$, which amounts to $\binom{k}{s}$ critical points. Thus we have 
    \begin{align}
        \mathrm{GDC}(\Omega_s) \geq \binom{k}{s}.
    \end{align}
\end{remark}
In the next subsection we will describe the local structure of the Schubert variety $\Omega_s$ around $\mathbf W$. When the generic point is chosen close enough to $\mathbf W$, this will allow us to characterize all the critical points close to $\mathbf W$: they will be \emph{exactly} those provided by \cref{standardEYomegas}.

\subsection{A local structure result}
    We begin this section by describing a degenerate setting for our problem. Instead of picking a generic point $\mathbf L \notin \Omega_s$, we look at what happens when the point from which we compute the distance is $\mathbf W \in \Omega_s$, i.e. the point corresponding to the deepest singularity of $\Omega_s$. In this case we can show that the situation for the critical points of $\dist_{\mathbf W}$ restricted to $\Omega_s$ is the following: 
\begin{itemize}[label={--}]
    \item there is a unique minimum, which is trivially $\mathbf W$ itself;
    \item there are no other local minima or saddle points;
    \item there is a $G(s,k)\times G(k-s,n-k)$ of global maximizers and no other local ones.
\end{itemize}
To see this, remark that for every point $\mathbf E \in \Omega_s$, the first $s$ principal angles with $\mathbf W$ must be $0$, so that the maximum distance possible from $\mathbf W$ is given by $\sqrt{(k-s)}\frac{\pi}{2}$, corresponding to when the remaining $k-s$ principal angles are $\frac{\pi}{2}$. By the properties of principal angles, if $(k-s)$ principal angles are $\frac{\pi}{2}$, it means that $\mathbf E \cap \mathbf W^{\perp}$ has dimension $k-s$. Therefore any global maximizer $\mathbf E$ can be constructed as the orthogonal direct sum of an $s$--dimensional subspace of $\mathbf W$ and a $(k-s)$--dimensional subspace of $\mathbf W^{\perp}$. It follows that the set of global maximizers can be identified with a $G(s,k) \times G(k-s,n-k)$. To see that no other points can be critical, it is enough to observe that the geodesics connecting $\mathbf W$ and any other point $\mathbf E \in \Omega_s$ are entirely contained in $\Omega_s$, so that no other point apart from the global maximizers can be critical.\\

Now we ask the following question: what happens if we choose a point $\mathbf L \notin \Omega_s$ sufficiently close to $\mathbf W$? To answer this question, we first investigate the local structure of the subvariety $\Omega_s$ near its deepest singularity $\mathbf W$. \\

Let $D:=D(0,1)$ be the unit disk in $T_{\mathbf W}G(k,n)$ and $D_{\varepsilon}$ be the disk of radius $\varepsilon$. For every $\varepsilon >0$ define the map 
\begin{align}
\phi_{\varepsilon} : D \overset{\cdot \varepsilon}{\longrightarrow} D_{\varepsilon} \overset{\exp_{\mathbf W}}{\longrightarrow}B(\mathbf W, \varepsilon),
\end{align}
where $B(\mathbf W, \varepsilon)$ is the geodesic ball of radius $\varepsilon$ in $G(k,n)$ centered at $\mathbf W$. The family of maps $\phi_{\varepsilon}$ thus rescale the unit disk by a factor $\varepsilon$ and then apply the exponential  to map the rescaled vectors to $G(k,n)$. We can define a corresponding family of metrics $\{g_{\varepsilon}\}_{\varepsilon \geq0}$ on $D$ as
\begin{align}
    g_{\varepsilon}:=\frac{1}{\varepsilon^2}\phi_{\varepsilon}^*g,
\end{align}
where $g$ is the Grassmann metric on the geodesic ball $B(\mathbf W,\varepsilon)$ and $\phi_{\varepsilon}^*$ denotes the pullback via $\phi_{\varepsilon}$. The meaning of $g_{\varepsilon}$ is that it gives information on the distortion of the metric on $T_{\mathbf W}G(k,n)$ caused by the application of the exponential map. Denoting by $g_F$ the Frobenius metric on $T_{\mathbf W}G(k,n)$, we have the following result.
\begin{lemma}\label{metricconvergence}
    The family of metrics $g_{\varepsilon}$ converges $C^{\infty}$ to $g_F$ as $\varepsilon \longrightarrow 0$.
\end{lemma}
\begin{proof}
    Let $A \in D$, $B_1,B_2 \in T_AD\cong \R^{(n-k)\times k}$. Then
    \begin{align}
        (g_{\varepsilon})_A(B_1,B_2)=&\left(\frac{1}{\varepsilon^2}\phi_{\varepsilon}^*g\right)_A(B_1,B_2)=\, \frac{1}{\varepsilon^2}g_{\phi_{\varepsilon}(A)}(D_A\phi_{\varepsilon}(B_1),D_A\phi_{\varepsilon}(B_2))= \\
        =&\frac{1}{\varepsilon^2}g_{\phi_{\varepsilon}(A)}(\varepsilon D_{\varepsilon A} \exp_{\mathbf W}(B_1),\varepsilon D_{\varepsilon A} \exp_{\mathbf W}(B_2)) = \\
        =&g_{\phi_{\varepsilon}(A)}(D_{\varepsilon A} \exp_{\mathbf W}(B_1), D_{\varepsilon A} \exp_{\mathbf W}(B_2)).
    \end{align}
    Since $D_0\exp_{\mathbf W}$ is the identity, we obtain the wanted $C^{\infty}$ convergence. 
\end{proof}
\begin{proposition}\label{convergenceprop}
    There exists $\varepsilon >0$ such that for the generic $\mathbf L \in B(\mathbf W, \widetilde \varepsilon)$, $\mathbf L \notin \Omega_s$ with $\widetilde \varepsilon < \varepsilon$ and $\mathbf L = \exp_{\mathbf W}(v_A)$, the critical points for $\dist_{\mathbf L}|_{\Omega_s}$ contained in $B(\mathbf W, \varepsilon)$ are approximations of the critical points for $\widehat{\dist}_A|_{\mathcal M_{k-s}}$ coming from Eckart--Young theorem. 
\end{proposition}
\begin{remark}
    Given a matrix $A \in \R^{(n-k)\times k}$, all its rank-approximations coming from Eckart--Young theorem have norm smaller than $\|A\|$, as they are obtained by canceling out some of the singular values. 
\end{remark}
\begin{proof}[\textit{Proof of \cref{convergenceprop}}]We give a sketch of the proof.
    Observe that the distance between $\mathbf L$ and its cut locus is always $\frac{\pi}{2}$. Therefore for $\varepsilon$ small enough, the ball $B(\mathbf W,\varepsilon)$ will not intersect the cut locus of any $\mathbf L \in B(\mathbf W,\widetilde\varepsilon)$ for $\widetilde \varepsilon < \varepsilon$. The result then follows directly from \cref{metricconvergence} and the fact that critical points away from the cut locus are nondegenerate. 
\end{proof}
\begin{remark}By \cref{convergenceprop} we have that critical points close to $\mathbf W$ will be given by approximations of the Eckart--Young ones. But \cref{standardEYomegas} tells us that the critical points coming from Eckart-Young are actually critical points also in our setting, suggesting that the exponential mapping has no distortion effect on critical points. The reason why this happens, admittedly, is still not clear to the authors and will be subject to further investigation.  \end{remark}

As a corollary of \cref{convergenceprop} and \cref{standardEYomegas} we have that for $\mathbf L$ close enough to $\mathbf W$, the minimum for $\dist_{\mathbf L}|_{\Omega_s}$ is given by the minimum in Eckart--Young theorem, as no other critical points closer to $\mathbf L$ than the ones coming from \cref{standardEYomegas} can be found. In the next section we will see that this is a general fact: for the generic $\mathbf L$ the minimum will always be the one provided by \cref{standardEYomegas}.

\subsection{Global minimum}

We will now state a theorem that provides the global minimum for the restriction of $\dist_{\mathbf L}$ to $\Omega_s$, thus solving the nearest point problem for simple Schubert varieties.

%The geometric idea is the following. Every point $\mathbf E \in \Omega_s$ has to to satisfy the condition that $\mathbf E \cap \mathbf W$ is of dimension at least $s$. In other words $\mathbf E$ has to contain an $s$--dimensional subspace of $\mathbf W$. Suppose that $0<\theta_1(\mathbf L)<\dots<\theta_k(\mathbf L)<\frac{\pi}{2}$ are the principal angles between $\mathbf L$ and $\mathbf W$. Then we claim that, since $\mathbf E$ contains an $s$--dimensional subspace of $\mathbf W$, the squared distance between $\mathbf E$ and $\mathbf L$ must be at least $\sum_{i=1}^s \theta_i(\mathbf L)^2$, that is to say the angles between $\mathbf E$ and $\mathbf L$ have to be at least the same as the smallest $s$ angles between $\mathbf W$ and $\mathbf L$. As this quantity is actually achieved by the critical point corresponding to the Eckart--Young global minimizer via \cref{standardEYomegas}, we state the following.
\begin{theorem}\label{conjecture}
    For the generic $\mathbf L \notin \Omega_s$, the restriction of $\dist_{\mathbf L}$ to $\Omega_s$ admits a unique global minimizer, corresponding to the minimum given by the Eckart--Young Theorem through \cref{standardEYomegas}. Explicitly, if $0<\theta_1(\mathbf L)< \dots< \theta_k(\mathbf L) < \frac{\pi}{2}$ are the angles between $\mathbf L$ and $\mathbf W$ with principal vectors $\{p_1,\dots,p_k\}$ and $\{q_1,\dots,q_k\}$ for $\mathbf L$ and $\mathbf W$ respectively, we have 
    \begin{align}
        \min \dist_{\mathbf L}|_{\Omega_s} = \mleft(\theta_1(\mathbf L)^2+\dots+\theta_s(\mathbf L)^2\mright)^{\frac{1}{2}},
    \end{align}
    with the unique global minimizer being 
    \begin{align}\label{minimizer}
        \mathbf E_m = \mathrm{span}\{q_1,\dots,q_s,p_{s+1},\dots,p_k\}.
    \end{align}
\end{theorem}
The proof will use the following proposition (\cite[Corollary 3.1.3]{matrixanalysis}).
\begin{proposition}[{\cite[Corollary 3.1.3]{matrixanalysis}}]\label{cor313}
    Let $B \in \R^{n\times k}$ and let $B_r$ denote a submatrix of $B$ obtained by deleting $r$ columns from $B$. Ordering the singular values of a matrix in decreasing order, we have 
    \begin{align}
        \sigma_i(B)\geq \sigma_i(B_r) \geq \sigma_{i+r}(B),
    \end{align}
    for every $i=1,\dots,k$, where for a matrix $X \in \R^{p\times q}$ we set $\sigma_j(X)=0$ if $j>\min\{p,q\}$.
\end{proposition}
Recall also that the algorithm used to define principal vectors and angles can be extended to define principal vectors and angles between subspaces of different dimensions exactly in the same way as it was done in \cite{Lim}.  Here we need a  modification of \cite[Corollary 10]{Lim}, that we state using the same notation of \cite{Lim}.
\begin{corollary}\label{anglescorollary}
    Let $\mathbf A \in G(k,n)$, $\mathbf B \in G(k,n)$ and $\widetilde{\mathbf B} \in G(\widetilde k, n)$ such that $\widetilde k \leq k$ and $\widetilde{\mathbf B}\subseteq \mathbf B$. Then we have 
    \begin{align}\label{anglesinequalities}
        \theta_i(\mathbf A, \mathbf B) \leq \theta_i(\mathbf A,\widetilde{\mathbf B}) \leq \theta_{i+k-\widetilde k}(\mathbf A, \mathbf B),
    \end{align}
    for every $i=1,\dots,\widetilde k$.
\end{corollary}
\begin{proof}
    Let $\{w_1,\dots,w_{\widetilde k}\}$ be an orthonormal basis for $\widetilde{\mathbf B}$ and complete it to an orthonormal basis of $\mathbf B$ $\{w_1,\dots,w_k\}$. Let $\widetilde B =(w_1 \dots w_{\widetilde k})$ and $B=(w_1 \dots w_k)$. Let $A =(a_1 \dots a_k) \in \R^{n\times k}$ be a matrix whose columns are an orhtonormal basis of $\mathbf A$. Then we know that  
    \begin{align}\label{arccosangle}
        \theta_i(\mathbf A, \mathbf B)&=\arccos(\sigma_i(A^{\top}B)) \quad i=1,\dots,k, \\ \notag
        \theta_i(\mathbf A,\widetilde{\mathbf B})&=\arccos(\sigma_i(A^{\top}\widetilde B)) \quad i=1,\dots,\widetilde k.
    \end{align}
    Since $A^{\top}\widetilde B$ is obtained from $A^{\top}B$ by deleting $k-\widetilde k$ columns, by \cref{cor313} we have 
    \begin{align} \label{sigmaineq}
        \sigma_i(A^{\top}B) \geq \sigma_i(A^{\top}\widetilde B) \geq \sigma_{i+k-\widetilde k}(A^{\top}B),
    \end{align}
    for every $i=1,\dots,\widetilde k$. Since all the singular values are in $[0,1]$ and $\arccos$ is descreasing on that interval, from \eqref{arccosangle} and \eqref{sigmaineq} we obtain \eqref{anglesinequalities}. 
\end{proof}
\begin{proof}[Proof of \cref{conjecture}]
    Let $\mathbf E \in \Omega_s$ and denote $\mathbf V:=\mathbf E \cap \mathbf W$. Since $\mathbf E \in \Omega_s$, we have $\mathrm{dim}(\mathbf V)=\widetilde s \geq s$. Let $0\leq\theta_1(\mathbf V,\mathbf L) \leq\dots \leq \theta_{\widetilde s}(\mathbf V,\mathbf L)\leq \frac{\pi}{2}$ be the principal angles between $\mathbf V$ and $\mathbf L$. Recall that 
    \begin{align}
        \dist_{\mathbf L}(\mathbf E) = \mleft(\theta_1(\mathbf E,\mathbf L)^2+\dots+\theta_k(\mathbf E,\mathbf L)^2\mright)^{\frac{1}{2}}.
    \end{align}
    Since $\mathbf V$ is a subspace of $\mathbf W$, the left side of \eqref{anglesinequalities} implies that
    \begin{align}
        \theta_i(\mathbf L) \leq \theta_i(\mathbf V,\mathbf L) \qquad i=1,\dots,\widetilde s.
    \end{align}
    But $\mathbf V$ can also be regarded as a subspace of $\mathbf E$, hence the right hand side of \eqref{anglesinequalities} implies that 
    \begin{align}
        \theta_i(\mathbf V,\mathbf L) \leq \theta_{i+k-\widetilde s}(\mathbf E,\mathbf L) \qquad i=1,\dots,\widetilde s.
    \end{align}
    The two inequalities above together give
    \begin{align}
        \theta_1(\mathbf L) &\leq \theta_{k-\widetilde s +1}(\mathbf E,\mathbf L) \\
        &\vdots \\
        \theta_{\widetilde s}(\mathbf L) &\leq \theta_k(\mathbf E, \mathbf L).
    \end{align}
     In turn, this gives the inequality
    \begin{align}
        \notag \dist_{\mathbf L}(\mathbf E)&= \mleft(\theta_1(\mathbf E,\mathbf L)^2+\dots+\theta_k(\mathbf E,\mathbf L)^2\mright)^{\frac{1}{2}} \\ \notag &\geq \mleft(\theta_1(\mathbf E,\mathbf L)^2+\dots+\theta_{k-\widetilde s}(\mathbf E,\mathbf L)^2+\theta_1(\mathbf L)^2+\dots+\theta_{\widetilde s}(\mathbf L)^2\mright)^{\frac{1}{2}} \\ \notag
        &\geq \mleft(\theta_1(\mathbf L)^2+\dots+\theta_{\widetilde s}(\mathbf L)^2\mright)^{\frac{1}{2}} \\ \label{51}
        &\geq \mleft(\theta_1(\mathbf L)^2+\dots+\theta_s(\mathbf L)^2\mright)^{\frac{1}{2}}
    \end{align}
It is easy to check that the plane $\mathbf E_m$ defined in \eqref{minimizer} achieves this lower bound by computing principal angles in the usual way, showing that 
\begin{align}
    \dist_{\mathbf L}(\mathbf E_m)=\min \dist_{\mathbf L}|_{\Omega_s} = \mleft(\theta_1(\mathbf L)^2+\dots+\theta_s(\mathbf L)^2\mright)^{\frac{1}{2}}.
\end{align}
Moreover, by construction, $\mathbf E_m$ is the point obtained in \cref{standardEYomegas} by choosing the approximation corresponding to the minimizer provided by \cref{EY}. 

To prove that $\mathbf E_m$ is the unique minimizer, we start by observing that since all the angles $\theta_i(\mathbf L)$ are positive, if $\widetilde s > s$ the inequality \eqref{51} is strict and the minimum cannot be achieved. Therefore for every minimizer it holds that $\widetilde s=s$, i.e. it is a smooth point of $\Omega_s$. 

Another necessary condition for the minimum to be achieved by a point $\mathbf E$ is that 
\begin{align}
    \theta_1(\mathbf E,\mathbf L)=\dots=\theta_{k-s}(\mathbf E,\mathbf L)=0,
\end{align}
which means that $\mathrm{\dim}(\mathbf E \cap \mathbf L)=k-s$. Together with $\mathbf E \in \Omega_s$ and the intersection between $\mathbf L$ and $\mathbf W$ being trivial, this implies that $\mathbf E$ must be an orthogonal direct sum
\begin{align}
    \mathbf E=\mathbf E_{\mathbf L} \oplus \mathbf E_{\mathbf W},
\end{align}
where $\mathbf E_{\mathbf L}$ is a $(k-s)$--dimensional subspace of $\mathbf L$ and $\mathbf E_{\mathbf W}$ is an $s$--dimensional subspace of $\mathbf W$. By construction, for a minimizer we have that 
\begin{align}
    \theta_{k-s+1}(\mathbf E,\mathbf L)&=\theta_1(\mathbf E_{\mathbf W},\mathbf L) = \theta_1(\mathbf L), \\ 
    & \ \, \vdots \\
    \theta_k(\mathbf E,\mathbf L)&=\theta_s(\mathbf E_{\mathbf W},\mathbf L)=\theta_s(\mathbf L).
\end{align}
Since $\mathbf E_{\mathbf W}\subseteq \mathbf W$ and every angle $\theta_i(\mathbf L)$ has multiplicitly one, it follows that $\mathbf E_{\mathbf W}$ is uniquely determined as the span of the principal vectors of $\mathbf W$ corresponding to the angles $\theta_1(\mathbf L),\dots,\theta_s(\mathbf L)$ with respect to $\mathbf L$, that is 
\begin{align}
    \mathbf E_{\mathbf W}=\mathrm{span}\{q_1,\dots,q_s\}.
\end{align}
Then $\mathbf E_{\mathbf L}$ is uniquely determined as 
\begin{align}
    \mathbf E_{\mathbf L}=\mathbf L \cap (\mathbf E_{\mathbf W})^{\perp} = \mathrm{span}\{p_{s+1},\dots,p_k\}.
\end{align}
It follows that the minimizer $\mathbf E_m$ is unique. 
\end{proof}
\subsection{Global maximum}Using \cref{anglescorollary} in a different order than in above proof and the idea explained in \cref{maximaremark}, we can also find explicitly the global maximizers for $\dist_{\mathbf L}|_{\Omega_s}$. In particular we will see that there can be infinitely many of them. This is not in contrast with the finiteness of the Grassmann Distance Complexity, since they all belong to the cut locus of the external point $\mathbf L$. Moreover, one could check that no geodesic from $\mathbf L$ to a global maximum is normal to $\Omega_s$, as we expect from \cref{normalexplemma}. 
\begin{theorem}\label{maximatheorem}
    For the generic $\mathbf L \notin \Omega_s$, the restriction of $\dist_{\mathbf L}$ to $\Omega_s$ admits a Grassmannian $G(k-s,n-k-s)$ of global maximizers, which can be explicitly computed. If $0<\theta_1(\mathbf L)<\dots<\theta_k(\mathbf L)<\frac{\pi}{2}$ are the principal angles between $\mathbf L$ and $\mathbf W$, we have 
    \begin{align}
        \max \dist_{\mathbf L}|_{\Omega_s} = \mleft(\theta_{k-s+1}(\mathbf L)^2+\dots+\theta_k(\mathbf L)^2+(k-s)\frac{\pi}{4}^2\mright)^{\frac{1}{2}}.
    \end{align}
\end{theorem}
\begin{proof}
    Let $\mathbf E \in \Omega_s$ and denote $\mathbf V:= \mathbf E \cap \mathbf W$. Since $\mathbf E \in \Omega_s$, we have $\mathrm{dim}(\mathbf V)=\widetilde s\geq s$. \\
    Since $\mathbf V$ is a subspace of $\mathbf E$, the left hand side of \eqref{anglesinequalities} implies that 
    \begin{align}
        \theta_i(\mathbf L, \mathbf E) \leq \theta_i(\mathbf L,\mathbf V) \ \ i=1,\dots,\widetilde s.
    \end{align}
    But $\mathbf V$ is also a subspace of $\mathbf W$, so that the right hand side of \eqref{anglesinequalities} gives 
    \begin{align}
        \theta_i(\mathbf L,\mathbf V) \leq \theta_{i+k-\widetilde s}(\mathbf L) \ \ i=1,\dots,\widetilde s.
    \end{align}
    These inequalities together give 
    \begin{align}
        \theta_1(\mathbf L,\mathbf E)&\leq \theta_{k-\widetilde s +1}(\mathbf L) < \frac{\pi}{2}, \\ 
        & \ \ \vdots \\
        \theta_{\widetilde s}(\mathbf L,\mathbf E) &\leq \theta_k(\mathbf L) < \frac{\pi}{2},
    \end{align}
    and the trivial bounds for the remaining angles $\theta_j(\mathbf L,\mathbf E)\leq \frac{\pi}{2}$ for $j=\widetilde s+1,\dots,k$. We have the following chain of inequalities
    \begin{align}
        \dist^2_{\mathbf L}(\mathbf E)&=\theta_1(\mathbf L,\mathbf E)^2+\dots+\theta_k(\mathbf L,\mathbf E)^2 \\
        &\leq \theta_{k-\widetilde s+1}(\mathbf L)^2 + \dots + \theta_k(\mathbf L)^2 + \theta_{\widetilde s +1}(\mathbf L, \mathbf E)^2+\dots+\theta_k(\mathbf L, \mathbf E)^2 \\
        &\leq \theta_{k-\widetilde s+1}(\mathbf L)^2 + \dots + \theta_k(\mathbf L)^2 + (k-\widetilde s)\,\frac{\pi}{4}^2 \\
        &\leq \theta_{k-s+1}(\mathbf L)^2 + \dots + \theta_k(\mathbf L)^2 + (k-s)\,\frac{\pi}{4}^2,
    \end{align}
    where the last inequality follows from $s \leq \widetilde s$ and all angles being bounded by $\frac{\pi}{2}$; moreover, since by genericity $\theta_j(\mathbf L)< \frac{\pi}{2}$ for every $j=1,\dots,k$, this inequality is strict if $s < \widetilde s$. Since these inequalities hold for every $\mathbf E \in \Omega_s$, it follows that 
    \begin{align}\label{maxbound}
        \max \dist_{\mathbf L}|_{\Omega_s}\leq \mleft(\theta_{k-s+1}(\mathbf L)^2 + \dots + \theta_k(\mathbf L)^2 + (k-s)\,\frac{\pi}{4}^2\mright)^{\frac{1}{2}}.
    \end{align}
    We will now build a family of points in $\Omega_s$ achieving this bound, proving that it is the maximum. \\
    As already pointed out, in order for the inequality to not be strict, we need to have $\widetilde s=s$, that is $\mathbf E$ has to be a smooth point of $\Omega_s$. Moreover we need 
    \begin{align}
        \theta_1(\mathbf L,\mathbf E) &= \theta_{k-s+1}(\mathbf L) \\ 
        & \ \ \vdots \\
        \theta_s(\mathbf L,\mathbf E) &= \theta_k(\mathbf L) \\
        \theta_{s+1}(\mathbf L,\mathbf E)=&\dots=\theta_k(\mathbf L,\mathbf E)=\frac{\pi}{2}.
    \end{align}
    In particular $\mathrm{dim}(\mathbf E \cap \mathbf L^{\perp})=k-s$. This forces $\mathbf E$ to be a direct sum 
    \begin{align}
        \mathbf E = \mathbf E_{\mathbf L^{\perp}} \oplus \mathbf E_{\mathbf W},
    \end{align}
    where $\mathbf E_{\mathbf L^{\perp}}$ is a $k-s$-dimensional subspace of $\mathbf L^{\perp}$ and $\mathbf E_{\mathbf W}$ is a $s$-dimensional subspace of $\mathbf W$. For such an $\mathbf E$ we can see that 
    \begin{align}
        \theta_i(\mathbf L,\mathbf E)=\theta_i(\mathbf L,\mathbf E_{\mathbf W})=\theta_{k-s+i}(\mathbf L)
    \end{align}
    for every $i=1,\dots,s$. Since every $\theta_j(\mathbf L)$ has multiplicity one and these are the maximal ones, $\mathbf E_{\mathbf W}$ is uniquely determined as the subspace corresponding to the last $s$ principal vectors of $\mathbf W$ with respect to $\mathbf L$. Denote them as $\{q_{k-s+1},\dots,q_k\}$. To build $\mathbf E_{\mathbf L^{\perp}}$ we have to select a $(k-s)$--dimensional subspace of 
    \begin{align}
        \mathbf A:=\mathbf L^{\perp} \cap \langle q_{k-s+1},\dots,q_k\rangle^{\perp}.
    \end{align}
    The space $\mathbf A$ has generically dimension $n-k-s$. Let $G(k-s,\mathbf A) \cong G(k-s,n-k-s)$ be the Grassmannian of $k-s$-subspaces in $\mathbf A$ and pick any $\mathbf B \in G(k-s,\mathbf A)$. Define $\mathbf E_{\mathbf B}:= \mathbf B \oplus \mathbf E_{\mathbf W} = \mathbf B \oplus \langle q_{k-s+1},\dots,q_k \rangle$. It is easy to check via the usual algorithm that for every choice of $\mathbf B \in G(k-s,\mathbf A)$ the plane $\mathbf E_{\mathbf B}$ achieves the bound \eqref{maxbound}. These are the unique ones achieving it, as their construction highlighted. 
\end{proof}

\begin{remark}
    From the proof of \cref{conjecture} it follows that the minimum distance from $\mathbf L$ achieved on the strata of $\Omega_s$ increases strictly every time we go to a deeper singular stratum. In the same way, from the proof of \cref{maximatheorem} we see that the maximum distance strictly decreases. The geometric idea behind this phenomenon is that every time we move to a deeper stratum, we are increasing the constraints on the intersection between the points and $\mathbf W$, which reflects on more constraints on the principal angles between those points and $\mathbf L$. We can also explicitly describe what the minimizers and maximizers are on each stratum, just reading through the proofs above.
\end{remark}

\section{Table of notations}
\begin{xltabular}{\textwidth}{p{0.19\textwidth} X}
$\langle A_1,A_2\rangle_F$ \dotfill & Frobenius scalar product between the two matrices $A_1$ and $A_2$. \\
$\widehat{\dist}(A_1, A_2)$ \dotfill & distance between two matrices $A_1, A_2\in \R^{n\times m}$ induced by the Frobenius scalar product. \\
$\widehat{\dist}_A:\R^{n\times m}\to \R$  & distance function from $A\in \R^{n\times m}$, defined by $\widehat{\dist}_A(X):=\|X-A\|_F.$ \\
$\mathcal{M}_{\leq r}$ \dotfill & Matrices of rank at most $r$ (\cref{def:rankconstraint}).\\
$\mathcal{M}_{r}$ \dotfill & Matrices of rank $r$ (\cref{def:rankconstraint}).\\
$A_I$\dotfill& for an index set $I\subseteq \{1,\ldots, m\}$ the $n\times m$ matrix defined as in \cref{EY}.\\
$\sigma_1(A),\dots,\sigma_m(A)$& singular values of a matrix $A\in \R^{n\times m}$ with $m\leq n$.\\
$G(k,n)$ \dotfill & The real Grassmannian of $k$--planes in $\R^n$.\\
$[v_1 \dots v_k]$\dotfill& $k$--plane spanned by the orthonormal vectors $v_1, \ldots, v_k\in \R^n$.\\
$\mathbf{E}, \mathbf{L}, \mathbf{W}, \mathbf{S}$ \dotfill & elements of the Grassmannian $G(k,n)$.\\
$\widehat{E}$\dotfill & representative in $O(n)$ of a point $\mathbf{E}\in G(k,n)\simeq O(n)/O(k)\times O(n-k)$.\\
$M_A$\dotfill & $n\times n$ matrix defined through the matrix $A\in\R^{(n-k)\times k}$ as in \eqref{eq:MA}.\\
$v_A$\dotfill & a tangent vector to $G(k,n)$ defined through the matrix $A\in\R^{(n-k)\times k}$ as in \cref{def:vA}.\\
$\varphi_A$ \dotfill &an element of  $\mathrm{Hom}(\mathbf{E}, \mathbf{E}^\perp)$ defined through the matrix $A\in\R^{(n-k)\times k}$ as in \cref{def:hom}.\\
$g_\mathbf{E}$\dotfill&the Grassmann Riemannian metric at the point $\mathbf{E}\in G(k,n)$ (\cref{def:riemmetric}).\\
$e^X$\dotfill& Exponential of the matrix $X\in \R^{n\times n}.$\\
$ \exp_{\mathbf E}$\dotfill & Riemannian exponential of $G(k,n)$ at the point $\mathbf{E}\in G(k,n)$.\\
$\theta_i(\mathbf{E}_1, \mathbf{E}_2)$\dotfill& $i$--th principal angle between the two $k$--planes $\mathbf{E}_1, \mathbf{E}_2\in G(k,n)$ (\cref{def:principal}).\\
$\dist(\mathbf{E}_1, \mathbf{E}_2)$ \dotfill & Grassmann Riemannian distance between two $k$--planes $\mathbf{E}_1, \mathbf{E}_2\in G(k,n)$ induced by the Grassmann Riemannian metric (\cref{def:Gdist}).\\
$R_{\frac{\pi}{2}}$\dotfill& the set of matrices $A\in \R^{(n-k)\times k}$ with singular values bounded by $\frac{\pi}{2}.$\\
$\dist_{\mathbf{L}}:G(k,n)\to \R$ & distance function from $\mathbf{L}\in G(k,n)$, defined by $\dist_{\mathbf{L}}(\mathbf{E}):=\dist(\mathbf{L}, \mathbf{E}).$ \\
$\partial_pf\dotfill$ & subdifferential at $p$ of a locally Lipschitz function $f:X\to \R$ (see \cref{def:subdifferential}).\\
$\omega(f)\dotfill$& for a locally Lipschitz function $f:X\to \R$, the set of its differentiability points (see \cref{def:subdifferential}).\\
$\mathrm{cut}(\mathbf{L})\dotfill$ & Cut--locus of $\mathbf{L}$ in $G(k,n)$, see \cref{def:cut}.\\
$\Omega(j)\dotfill$& For $j=1, \ldots, k$, stratum of the cut--locus defined by \eqref{cutstrataexplicit}.\\
$A_W\dotfill$& For an orthogonal matrix $W\in O(j)$, the matrix defined in \eqref{AWdef}.\\
$NX\dotfill$&Normal bundle to a smooth manifold $X\subset G(k,n)$.\\
$\Omega_s(\mathbf{W})\dotfill$&For $s=1, \ldots, k$ and $\mathbf{W}\in G(k,n)$, the Schubert variety defined by $\{\dim(\mathbf{E}\cap \mathbf{W})\geq s\}$, see \eqref{omegadef}.\\
$\mathrm{St}(k,n-k)\dotfill$&The Stiefel manifold of orthonormal $k$--frames in $\R^{n-k}$, \cref{sec:pfaffian}.\\
$A_\mathbf{L}\dotfill$&The matrix defined by \eqref{eq:AL}.\\
\end{xltabular}

\bibliographystyle{plain}
\bibliography{bibliototale}

\end{document}